\documentclass[12pt, reqno]{amsart}

\usepackage{euler}

\usepackage{cite}
\usepackage{amssymb}
\usepackage{amsmath}
\usepackage{amscd}
\usepackage{cancel}
\usepackage{float}
\usepackage{graphicx}
\usepackage{amsfonts}
\usepackage{color}

\usepackage{amsfonts}
\usepackage[cmtip,arrow]{xy}
\usepackage[normalem]{ulem}

\newtheorem{theorem}{Theorem}
\newtheorem{corollary}{Corollary}
\newtheorem{lemma}{Lemma}
\newtheorem{proposition}{Proposition}
\newtheorem{definition}{Definition}
\newtheorem{remark}{Remark}
\newtheorem{example}{Example}

\newtheorem{notation}{Notation}

\newcommand{\CC}{\mathbb{C}}
\newcommand{\FF}{\mathbb{F}}
\newcommand{\KK}{\mathbb{K}}

\newcommand{\NN}{\mathbb{N}}
\newcommand{\RR}{\mathbb{R}}
\newcommand{\ZZ}{\mathbb{Z}}

\def\a{\mathsf{a}\,}
\def\b{\mathsf{b}\,}

\def\r{\mathsf{r}\,}
\def\s{\mathsf{s}\,}
\def\t{\mathsf{t}\,}
\def\u{\mathsf{u}\,}
\def\v{\mathsf{v}\,}
\def\w{\mathsf{w}\,}
\def\x{\mathsf{x}\,}

\def\z{\mathsf{z}\,}

\def\Pn{\mathdf{P}_n}

\def\I{\mathrm{I}}
\def\P{\mathrm{P}}

\def\X{\mathrm{X}}

\def\tr{\mathrm{tr}}

\def\Pn{\mathsf{P}_n}
\newcommand{\Hn}{\mathrm{H}_n}
\newcommand{\Hnn}{\mathrm{H}_{n+1}}

\newcommand{\Yn}{{\rm Y}_{d,n}}

\newcommand{\En}{\mathcal{E}_n}

\begin{document}

\title{Polynomial  invariants for links and tied links}

 \author{Jes\'us Juyumaya}
 \address{IMUV, Universidad de Valpara\'{\i}so,
 Gran Breta\~{n}a 1111, Valpara\'{i}so 2340000, Chile.}
 \email{juyumaya@gmail.com}

\thanks{
}


\date{}

\maketitle

Notes  based on lessons  given   at  {\sc Escuela \lq Fico González Acuña\rq\ de Nudos y 3-variedades},
Mérida Yucatán, México,   7--10 (2015) and  {\sc Encuentro de nudos, trenzas y \'algebras}, Oaxaca--M\'exico, 3--10 October (2018).

I would like to thank for the invitation and hospitality received from Mario  Eudave
at Merida and  Bruno Cisneros at Oaxaca. Also, I would like to thank
 Francesca Aicardi  and Nicoletta Zar for their contribution in the writing of these notes.

\begin{center}
\maketitle
\begin{minipage}{10cm}
\tableofcontents
\end{minipage}
\end{center}

\section{Background}

\begin{definition}
A  knot is a subset of $\RR^3$
 homeomorphic to $S^1$. A link is a disjoint union of $n$ knots, $n$ is called the number of  components of the link.
\end{definition}
\begin{center}
\begin{figure}[H]
\includegraphics[scale=.5]{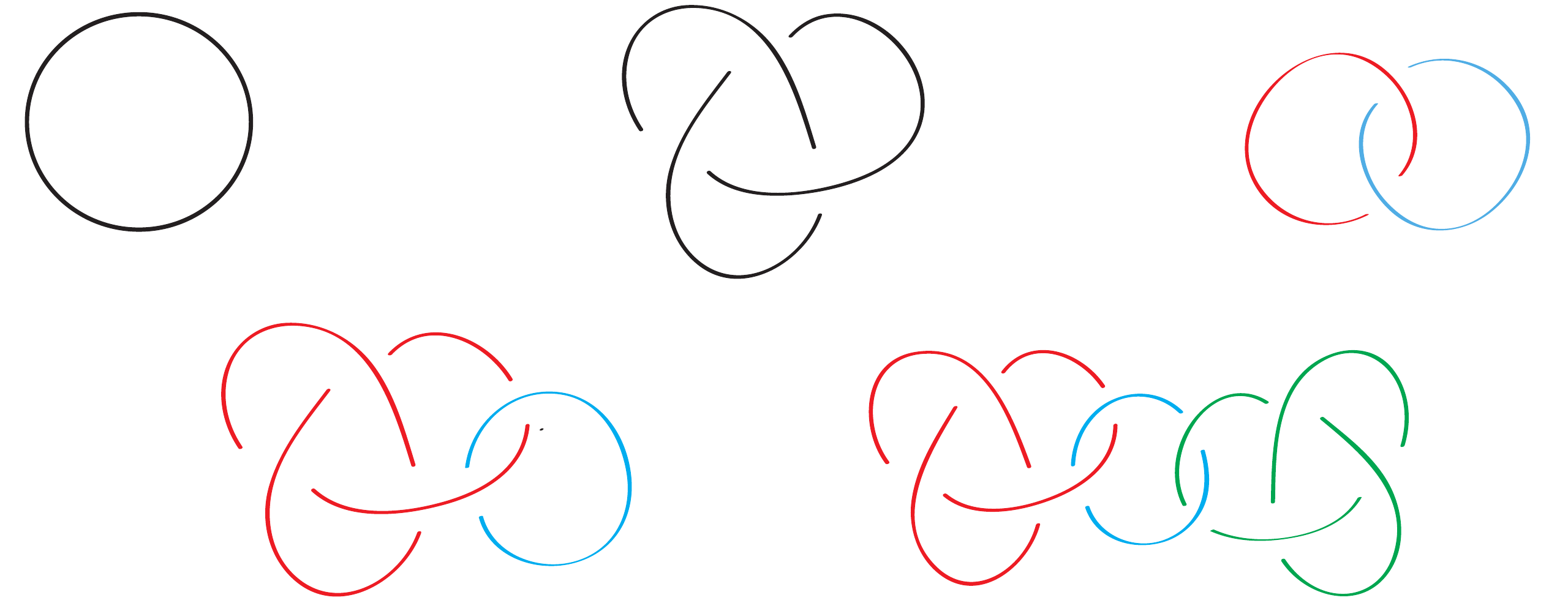}
\caption{}\label{KnotsLinks}
\end{figure}
\end{center}
In Fig. \ref{KnotsLinks}:	 the first knot is called the unknot, the second  is the left trefoil knot, the third is the  Hopf link, the fourth is a link with two components and the fifth is a link with three components.

\smallbreak
Since a knot is a simple closed curve we can provide it with an orientation, in such case we say that the {\it knot is oriented}. If each component of a  link is oriented, we say that  {\it the link is oriented}.
\begin{center}
\begin{figure}[H]
\includegraphics
[scale=.5]{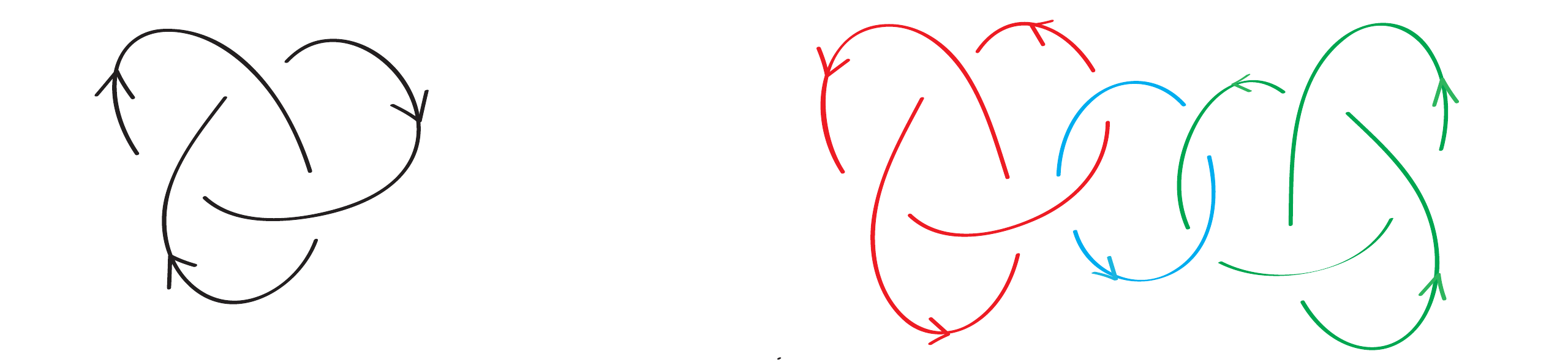}
\caption{}\label{KnotsLinksOriented}
\end{figure}
\end{center}
\begin{definition}
Two links $L_1$ and $L_2$ (both oriented or not) are called ambient isotopic (or equivalent)  if there exists an ambient isotopy  of $\mathbb{R}^3$ that carries $L_1$ in $L_2$, that is, there is  a continue map $\phi: \mathbb{R}^3\times\left[0,1\right]\longrightarrow\mathbb{R}^3$ such that:
\begin{enumerate}
\item $\phi (L_1,0)=L_1$ and $\phi (L_1,1)=L_2$,
\item For every $t\in[0,1]$, the maps $\phi_t$'s are continues and injective, where $\phi_t$ denotes to $\phi$ restricted to $\mathbb{R}^3\times\{t\}$.
\end{enumerate}
The fact that $L_1$ and $L_2$ are ambient isotopic is denoted by $L_1\sim L_2$.
\end{definition}

From now on we denote, respectively,  by $\mathfrak{L}$  and $\vert\mathfrak{L}\vert$ the set of all oriented links  and unoriented links in $\mathbb{R}^3$.
\begin{proposition}
The relation $\sim$ is an equivalence relation on $\mathfrak{L}$.
\end{proposition}

The main problem in knot theory is to find a subset of representatives of equivalence class of $\mathfrak{L}$  under the relations $\sim$, or equivalently, given two links to decide if they are ambient isotopic  or not. This problem is far from solution; however, we have a tool that helps to decide when  two links are not ambient isotopic: the invariant of links.
\begin{definition}
Let $\mathrm{Set}$  be a  set, an invariant of link is a function $\I: \mathfrak{L}\longrightarrow \mathrm{Set}$ such that:
$$
\text{If} \quad L_1 \sim L_2,\quad \text{then}\quad \I  (L_1) = \I (L_2).
$$
In the case that $\mathrm{Set}$ is  a polynomial ring or a field of rational functions, we say that $\I$ is a polynomial invariant of links.
\end{definition}
 To have an invariant when  $\mathrm{Set}$ is well understood allows  to compare easily  the images of the links by $\I$; therefore,  $\I$ is useful in the sense  that if the images of the links
 are different then the  links are not equivalent.  Some famous classical invariants are:
the linking number, the 3--coloration, the fundamental group of a link and the following polynomial invariants:
\begin{enumerate}
\item The Alexander polynomial,
\item The Jones polynomial,
\item The Homflypt polynomial,
\item The Kauffman polynomial.
\end{enumerate}

The Alexander and Jones invariants are polynomials in one variable. The Homflypt and Kauffman polynomial are in two variables. The Jones polynomial can be obtained as a specialization both of the Kauffman and Homflypt polynomials; the Alexander polynomial can be obtained  also as a specialization  of the Homflypt polynomial, but not as a specialization of the normalized Kauffman polynomial . We note that all these polynomial  invariants are defined for oriented links.

\section{Planar diagrams of links }
The links can be studied through diagrams in $\RR^2$: we associate to each link a {\it generic projection} on
 a plane, this projection is  provided with codes to indicate, at every double point, what  portion of  the  projected curve  (shadows) is coming from an over or under crossing. Thanks to a theorem of  Kurt  Reidemeister, the study of links can be translated to the study of  diagrams of links, allowing thus a combinatoric treating   of  the study of links.
\begin{definition}
A  generic   projection  of a curve of $\RR^3$ on a plane is one that only admits simple  crossings.
\end{definition}
\begin{center}
\begin{figure}[H]
\includegraphics
[scale=.5]
{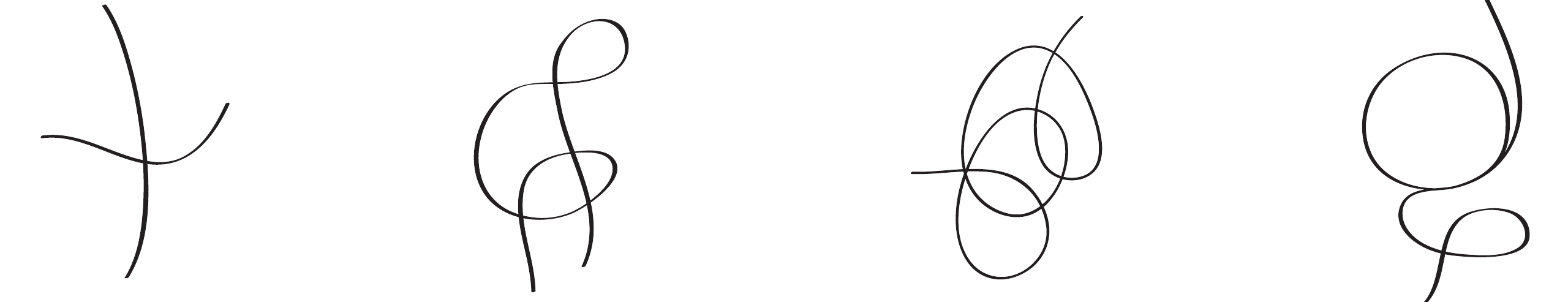}
\caption{}
\label{ProjectionMT}
\end{figure}
\end{center}
In Fig. \ref{ProjectionMT}, the first figure is a simple crossing, the second figure  is a generic projection of a curve and the third and fourth figures are not generic projections of a curve.
\begin{definition}
A diagram of a link is a generic projection (or shadow)   of it on a plane, where each further simple crossing  is codified, depending  if it is originated from an over or under  crossing in the link, by one of the two {\it crossing codifications}  below.
\begin{center}
\begin{figure}[H]
\includegraphics
[scale=.4]
{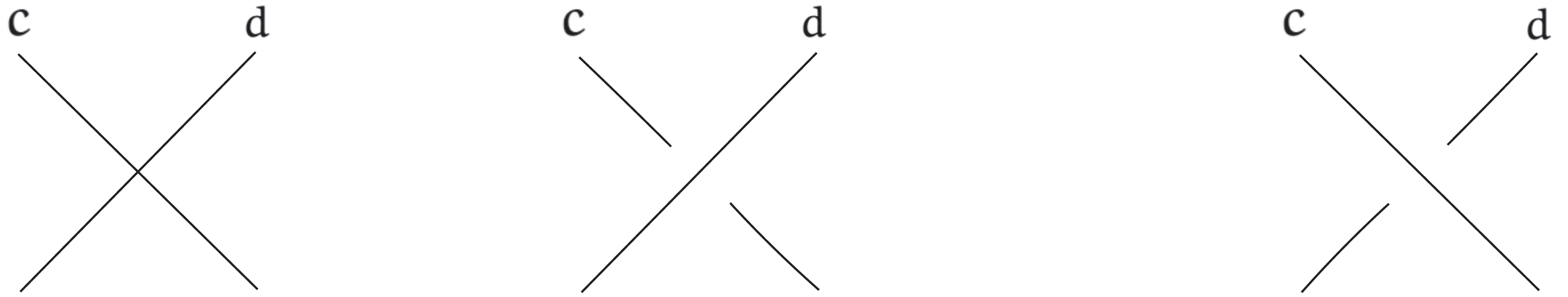}
\caption{}
\label{CodificationCrossing}
\end{figure}
\end{center}
More precisely, in  Fig. \ref{CodificationCrossing},
the first picture  is a simple crossing in the diagram  produced by the projection  of  the curves $c$ and $d$  of the link, the second   picture shows the code used  to indicate that the curve $c$ is under the curve $d$ in the link and the last  picture is to indicate that the curve $c$ is over  the curve   $d$ in the link.
\end{definition}
\begin{example}
Fig.  \ref{ProyectionTrefoil} shows the procedure to obtain a projection of a left trefoil knot.
\begin{center}
\begin{figure}[H]
\includegraphics
[scale=.35]{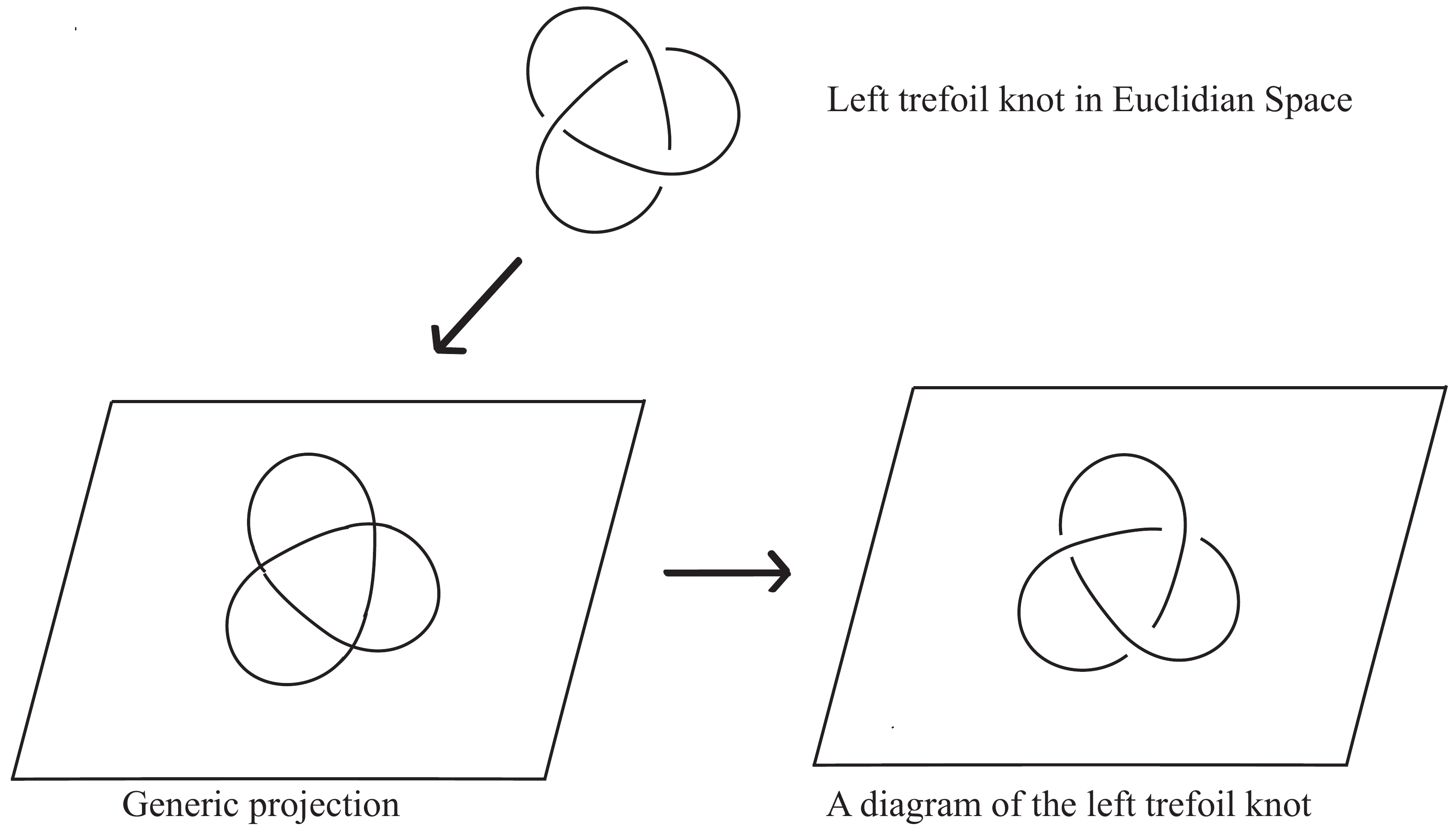}
\caption{}
\label{ProyectionTrefoil}
\end{figure}
\end{center}
\end{example}
Oriented  links yields  oriented diagrams; thus a diagram is oriented if every crossing codification  is in one of the following situations:
\begin{center}
\begin{figure}[H]
\includegraphics
[scale=.4]
{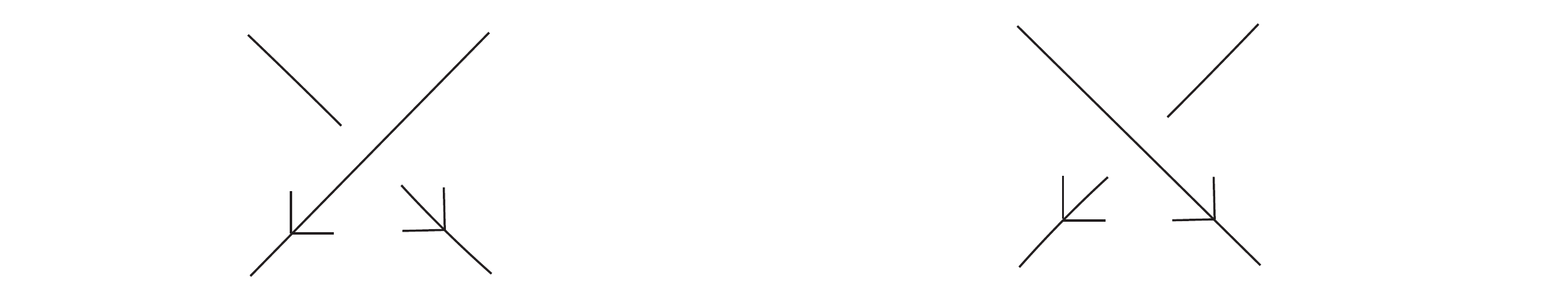}
\caption{}
\label{Fig5}
\end{figure}
\end{center}
The first  codification is  called (by convention)  {\it positive crossing} and the other one  {\it negative crossing}.
The {\it sign of the crossing} is $1$ if it is a positive crossing, otherwise is $-1$.
\begin{definition}\label{writhe}
Let $D$ be an oriented diagram, the writhe of $D$, denoted by $w(D)$,  is the sum over the sign of all crossings of $D$:
$$
w(D) := \sum \mathrm{sig}(p),
$$
where $p$ running on the crossing of $D$ and $\mathrm{sig}(p)$ is the sign of the crossing $p$.
\end{definition}

\begin{definition}\label{ReidemeisterMov}
Two diagrams of unoriented links are R--isotopic if one of them can be transformed in the other, by  so--called Reidemeister moves \rm{R0, R1, R2 }and/or \rm{R3}, where in the case unoriented links are:
\begin{center}
\begin{figure}[H]
\includegraphics
[scale=.5]
{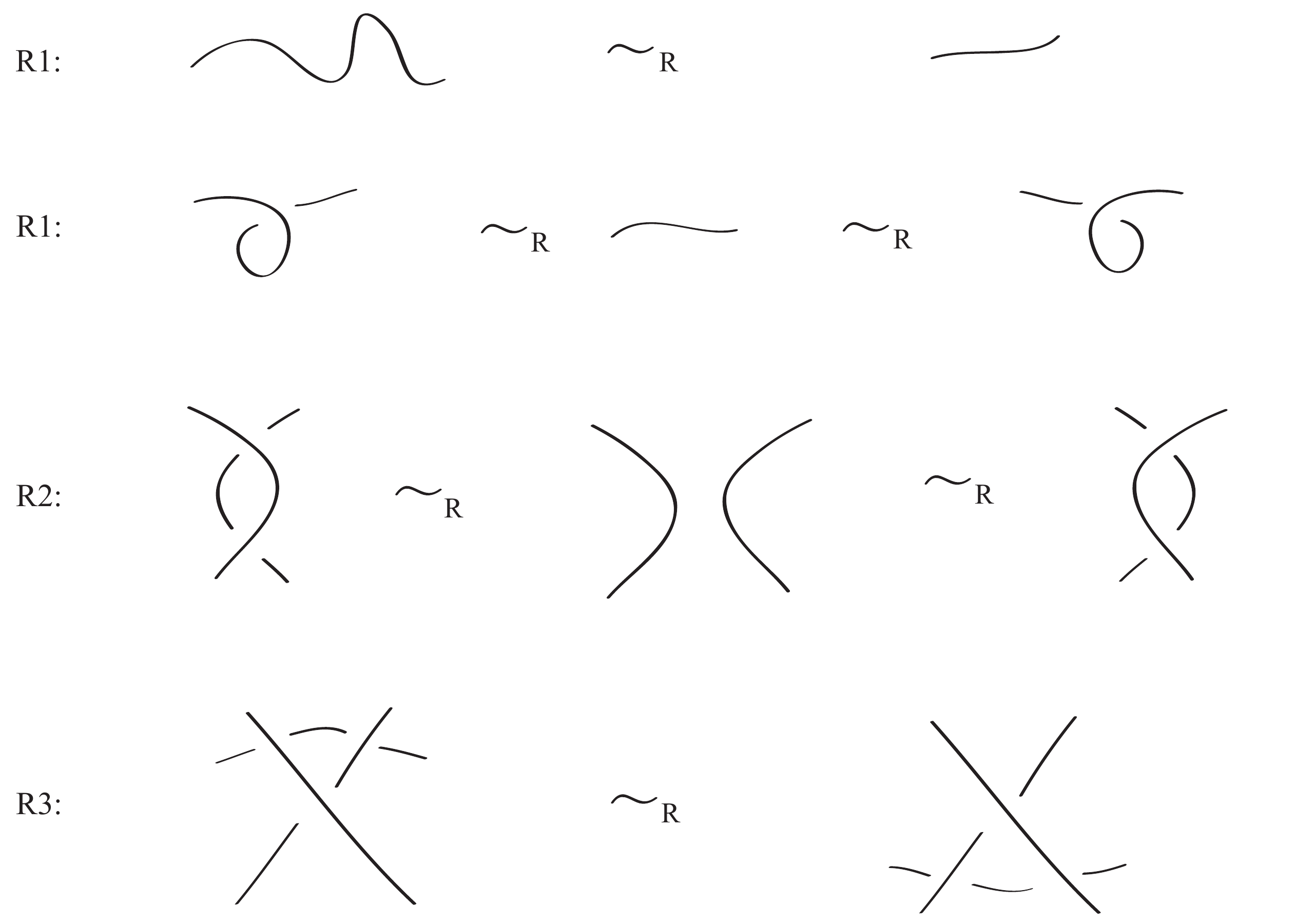}
\caption{}
\label{Reidemeister}
\end{figure}
\end{center}
We have the analogous definition for diagrams of oriented links but by  adding now all possible orientations
to  the moves R1--R3 above. For instance, some of these oriented Reidemeister are shown in Fig. \ref{OrientedReidemeister}.
\begin{center}
\begin{figure}[H]
\includegraphics
[scale=.5]
{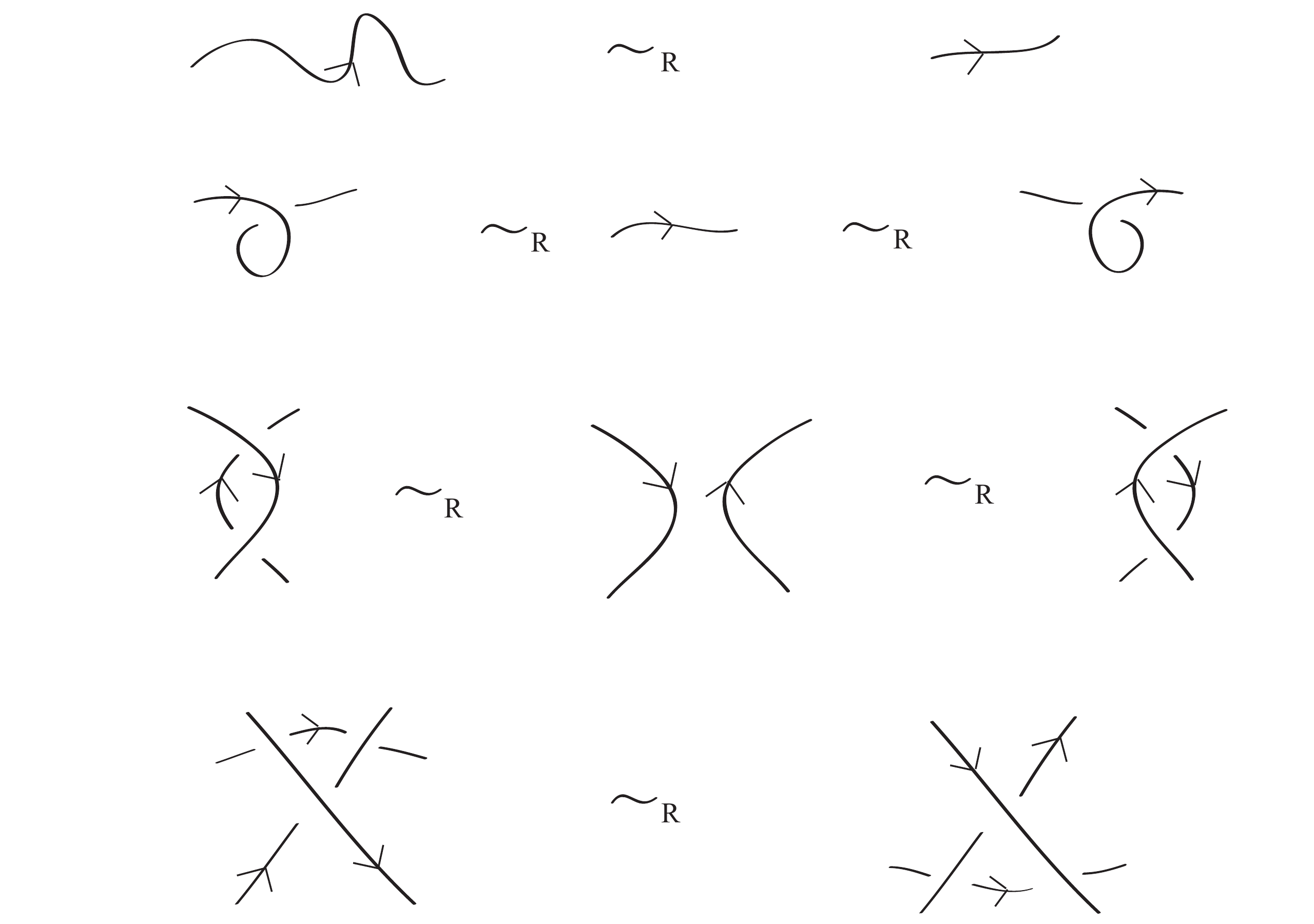}
\caption{}
\label{OrientedReidemeister}
\end{figure}
\end{center}
\end{definition}

\begin{proposition}
The relation of  R--isotopic,  denoted by $\sim_{\rm{R}}$, defines an equivalence relation on the set of planar diagrams.
\end{proposition}

\begin{notation}
\begin{enumerate}
\item
If $D$ is an oriented diagram, we denote by $\vert  D\vert$  the unoriented diagram obtained by forgetting  the orientation in $D$.
\item
We denote, respectively, by $\mathfrak{D}$ and $\vert \mathfrak{D}\vert$ the set of diagrams of, respectively,  oriented and unoriented links in $\RR^3$.
\end{enumerate}
\end{notation}
\begin{theorem}[Reidemeister, 1932]\label{ReidemeisterTh}
Let $L$ and $L'$ be two links (oriented or not) and set $D$ and $D'$  diagrams, respectively,  of $L$ and  $L'$, we have:
$$
L \sim L' \quad \text{ if and only if} \quad D \sim_{\rm{R}} D'.
$$
So, $\mathfrak{L}/\sim$ is in  bijection with $\mathfrak{D}/\sim_{\rm{R}}$.
\end{theorem}
\begin{example}
\begin{figure}
\begin{figure}[H]
\includegraphics
[scale=.45]
{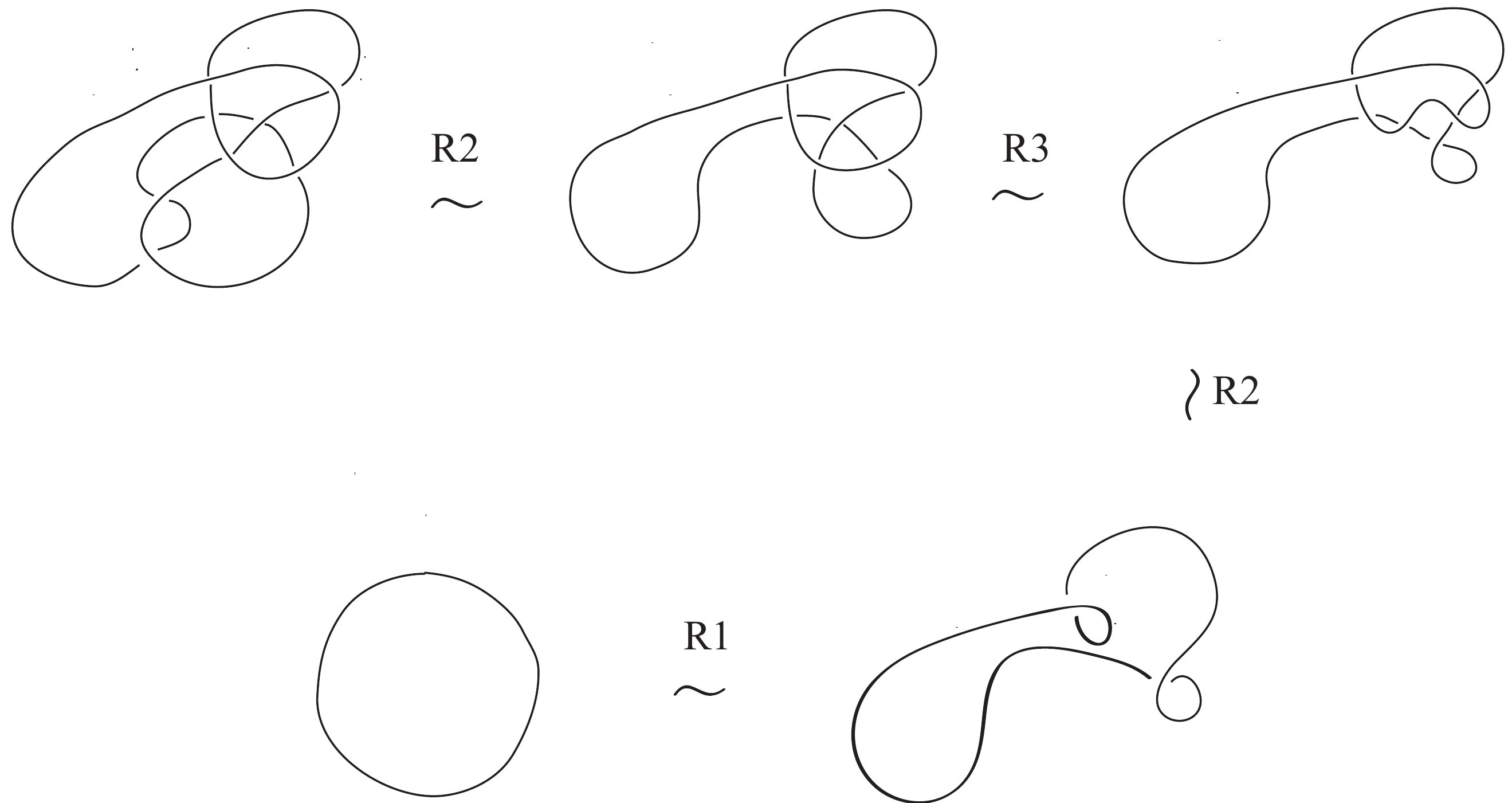}
\caption{}
\label{ExampleReidemeister}
\end{figure}
\end{figure}
\end{example}
The Reidemeister theorem says that  constructing  an invariant of links is equivalent to defining a function $\I:\mathfrak{D} \longrightarrow \mathrm{Set}$, such that it  takes the same values on diagrams that differ in R1, R2 and/or R3.
\begin{definition}
We say that  $\I$ is an invariant of regular isotopy  if the values of $\I$ does not change on unoriented links that are equal, up to the  moves   $R2$ and $R3$.
\end{definition}

\begin{proposition}\label{wregularisotopy}
$w$ is an invariant of regular isotopy.
\end{proposition}
\begin{proof}
To check that  $w$   agrees with the oriented move R2, it is enough to observe that introducing any orientation  to  $\includegraphics[scale=0.15,trim=0 0.2cm 0 0]{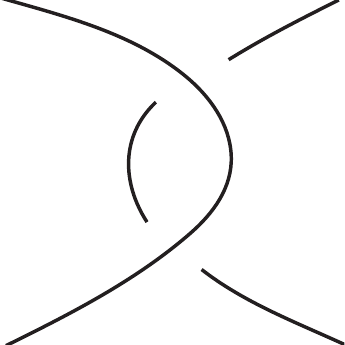}$, it turns out  that the sum of the   crossing signs is $0$; the same happens with  $\includegraphics[scale=0.15,trim=0 0.2cm 0 0]{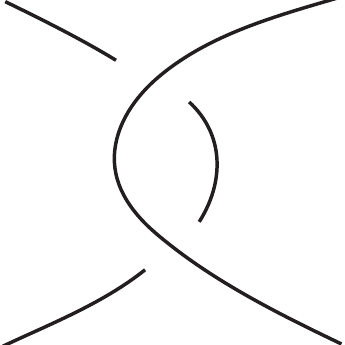}$.
Also, it is easy to see that by introducing any orientation to one of  the crossing configurations of R3, the sum of the signs of the three crossings in it  doesn't change after applying the move R3.
\end{proof}
 The following lemma gives a recipe, due to L. Kauffman,  to produce invariants of oriented links. This recipe was applied to define the Kauffman and  Dubrovnik polynomials.
\begin{lemma}[{\cite[Lemma 2.1]{kaTAMS}}] \label{KauffmanRecipe}
Let $\mathfrak{R}$ be a ring  and $a$ an invertible element of $\mathfrak{R}$.  Suppose that $f$ is a regular isotopy invariant of unoriented  diagram links  taking values in $\mathfrak{R}$, such that:
$$
f(\includegraphics[scale=0.2,trim=0 0.4cm 0 0]{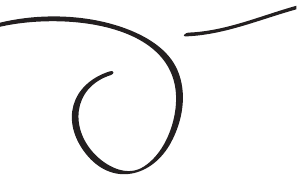}) = af(\includegraphics[scale=0.2,trim=0 0.4cm 0 0]{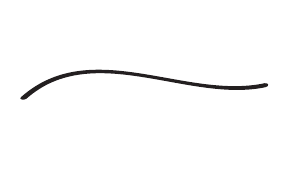})
\quad \text{and} \quad f(\includegraphics[scale=0.2,trim=0 0.4cm 0 0]{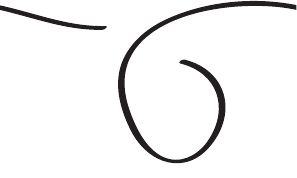})= a^{-1}f(\includegraphics[scale=0.2,trim=0 0.4cm 0 0]{Curler0.pdf}).
$$
Then,  the function $F$ defined as
$$
F (L) := a^{-w(D)} f (\vert D\vert )
$$
is an invariant for  oriented  links, where $D$ is a diagram for the oriented link $L$.
\end{lemma}
\section{Bracket polynomial}
Let $D$ be the diagram of a link. We will attach to $D$ a polynomial in the variables $A$, $B$ and  $z$  by using the following inductive process:
\begin{enumerate}
\item
We consider a crossing of $D$ to   assign the variables $A$ and $B$: 
\begin{center}
\begin{figure}
\includegraphics[scale=0.5]{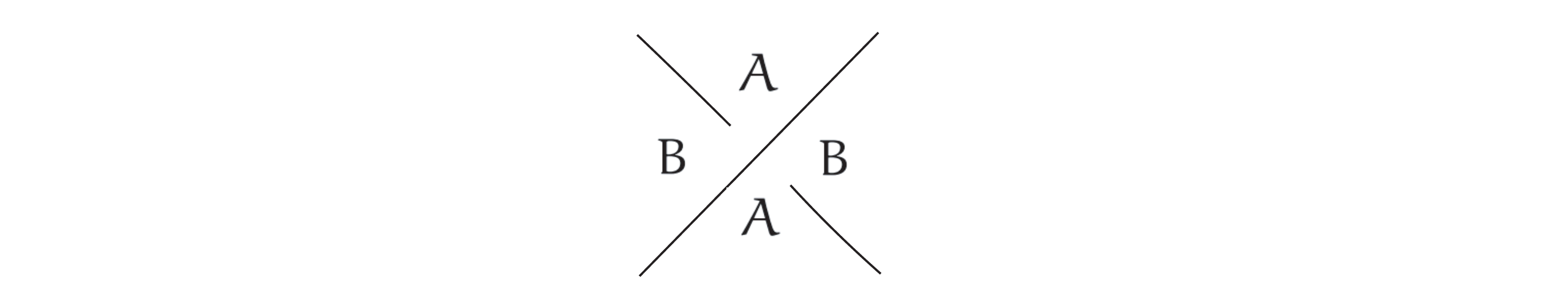}
\caption{}
\label{AandBcrossing}
\end{figure}
\end{center}
\noindent That is, we assign the variable $A$ to the region obtained by sweeping  the continuous arc  in counterclockwise and the variable $B$ to the remaining regions.
 \item
We smooth the crossing  (1) by replacing it with the following two configurations:
\begin{center}
\begin{figure}[H]
\includegraphics
[scale=.4]
{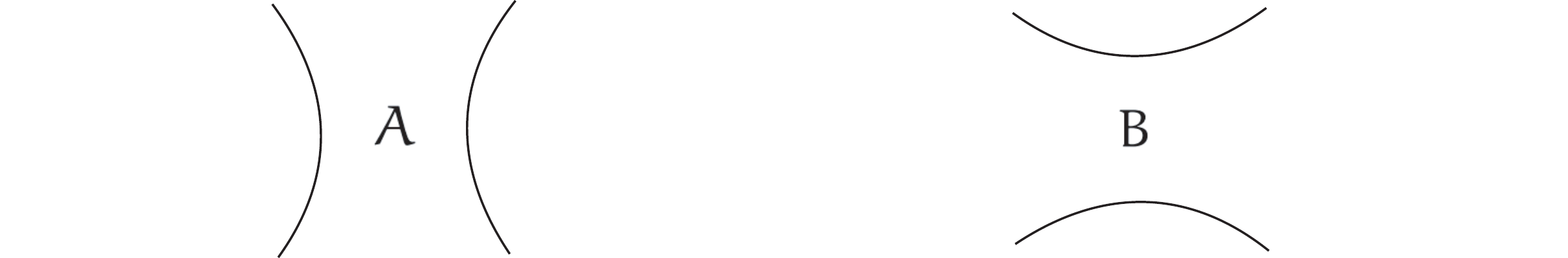}
\caption{}
\label{Fig9}
\end{figure}
\end{center}
 \noindent We obtain thus two links,  say $D_1$ and $D_2$, to which we assign  the  marks  $A$ and $B$ respectively. $D_1$ and $D_2$ have both  a crossing number  less than $D$.
 \item
 We choose now a crossing in $D_1$ and $D_2$ and we apply again the process (2); then we obtain 4 diagrams with marks $A$ and $B$. We apply again this process until we have only links without crossing.
\end{enumerate}
  Observe that if the original diagram   has $n$--crossings, by applying (1)--(3) we obtain finally $2^n$ diagrams, without crossings, having   marks    $A$ and $B$. We call these  $2^n$ diagrams (with their marks) the {\it states} of $D$, the set of states of $D$  is denoted by $\mathrm{St}(D)$.

\begin{example}
Below the procedure  (1)--(3)  described above for   the Hopf link of Fig. \ref{KnotsLinks}.
\begin{center}
\begin{figure}[H]
\includegraphics
[scale=.45]
{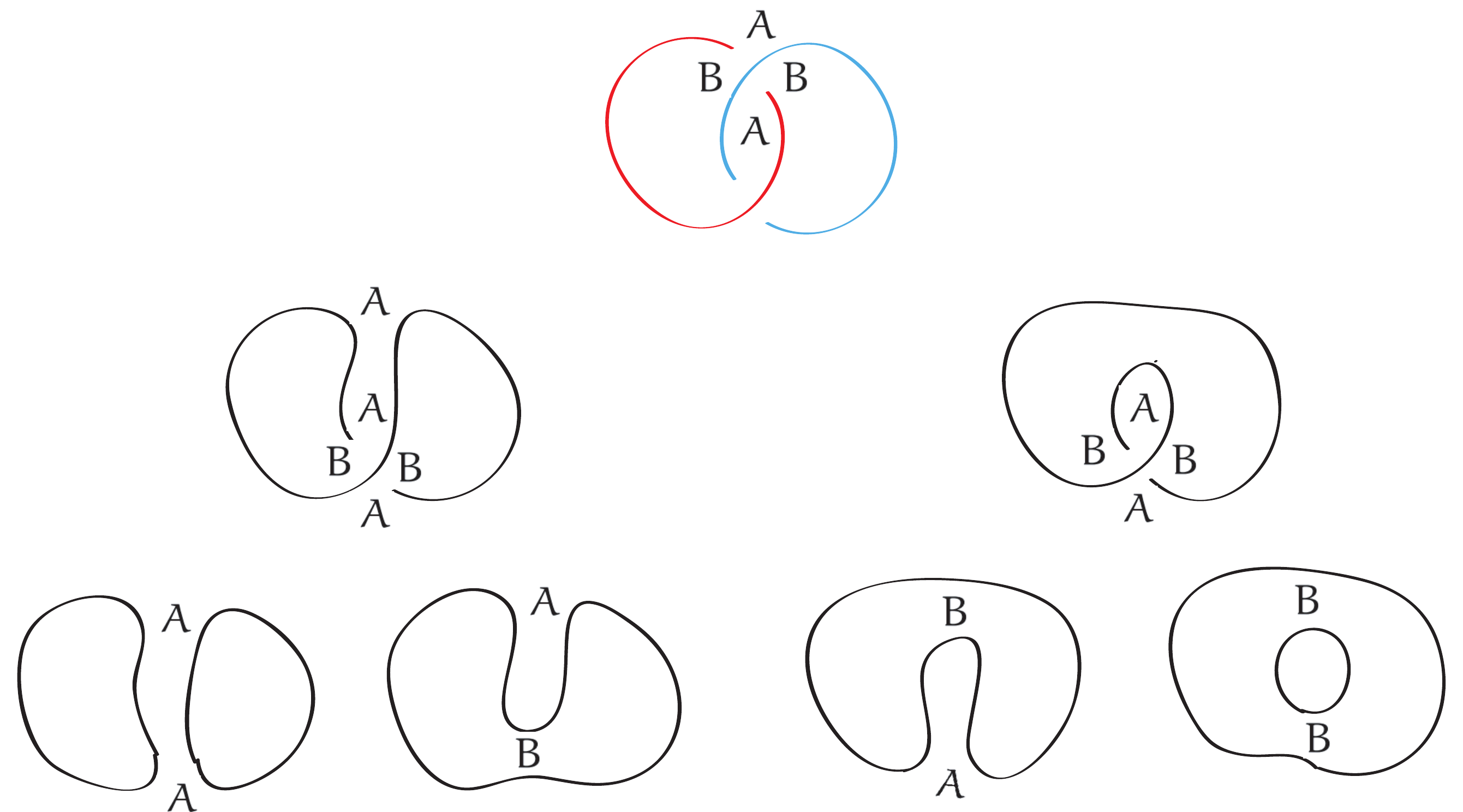}
\end{figure}
\end{center}
\end{example}

Given $S\in \mathrm{St}(D)$, we denote by $\vert S\vert$ the number of components of $S$ and by $\langle D, S  \rangle$ the product (commutative) of the marks appearing in $S$.

\begin{definition}
The Bracket polynomial of an unoriented diagram  $D$, denoted by $\langle D\rangle$, is the polynomial in $\ZZ[A,B,z]$ defined by
$$
\langle D\rangle = \sum_{S\in \mathrm{St}(D)}z^{\vert S\vert } \langle D, S  \rangle.
$$
\end{definition}
\begin{example}
For the Hopf link, denoted by $H$, we have
$$
 \mathrm{St}(H)= \left\{ \includegraphics[scale=0.2,trim=0 1.5cm 0 0]{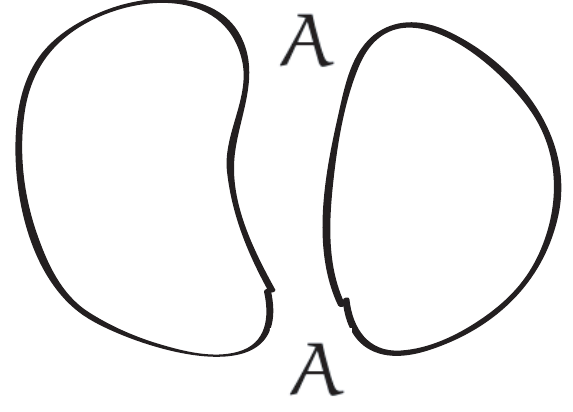},  \includegraphics[scale=0.2,trim=0 1.5cm 0 0]{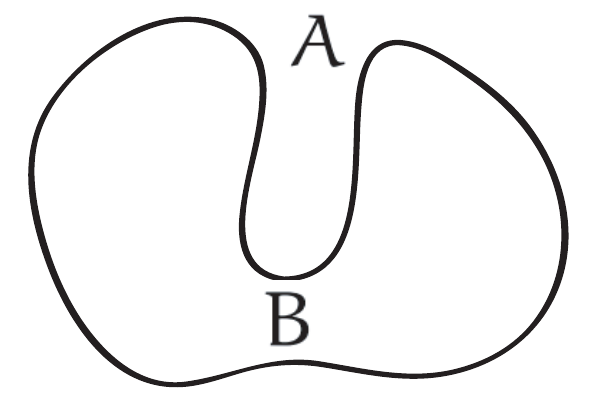},  \includegraphics[scale=0.2,trim=0 1.5cm 0 0]{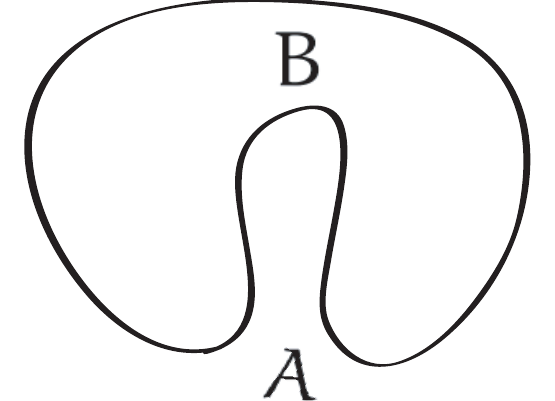},  \includegraphics[scale=0.2,trim=0 1.5cm 0 0]{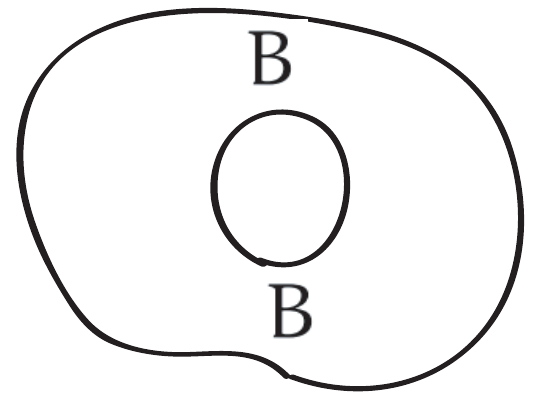} \right\}.
$$
Thus, in this  example  it seems  that  $|S|$ in  the number of components $-1$, then
$$
\langle D\rangle  = z^1A^2 + z^0AB+ z^0AB + z^1B^2=zA^2 + 2AB + zB^2.
$$
\end{example}
\begin{proposition}
\begin{enumerate}	
\item $\left\langle \bigcirc \sqcup D \right\rangle = z \left \langle D\right\rangle$,
\item $\left\langle \,\includegraphics[scale=0.12,trim=0 .8cm 0 0]{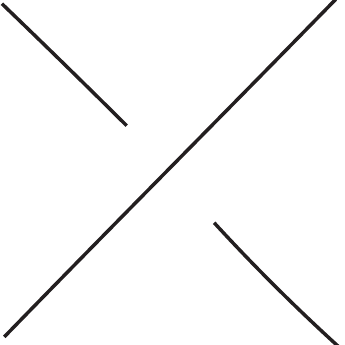} \, \right\rangle = B \left\langle \, \includegraphics[scale=0.12,trim=0 .8cm 0 0]{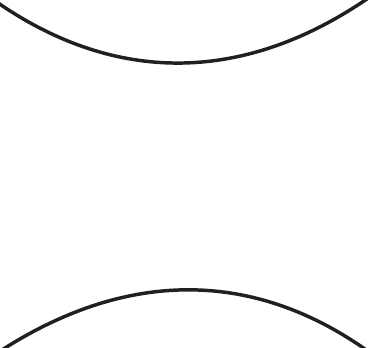} \,\right\rangle +
A\left\langle \, \includegraphics[scale=0.12,trim=0 .8cm 0 0]{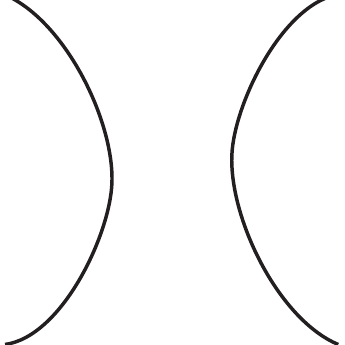}   \,\right\rangle$,
\item $\left\langle\,\includegraphics[scale=0.12,trim=0 .8cm 0 0]{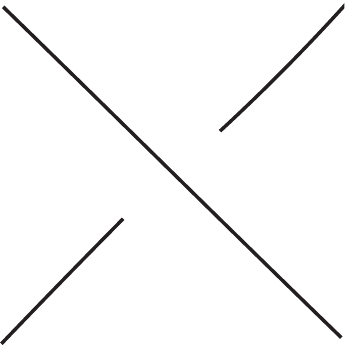} \,\right\rangle= A \left\langle \,\includegraphics[scale=0.12,trim=0 .8cm 0 0]{TanlgeSmothingCrossing.pdf} \,\right\rangle + B \left\langle\,\includegraphics[scale=0.12,trim=0 .8cm 0 0]{IdentitySmothingCrossing.pdf} \,\right\rangle$.
\end{enumerate}
\end{proposition}
\begin{proof}
See \cite[Proposition 3.2]{{kaSNE}}.
\end{proof}
We are going now to study the behaviour  of the Bracket polynomial under  the  Reidemeister moves.
\begin{lemma}\label{respectR1R2}
\begin{enumerate}
\item $\left\langle \, \includegraphics[scale=0.12,trim=0 .8cm 0 0]{SmothingR2.pdf}   \,\right\rangle = AB\left \langle \, \includegraphics[scale=0.12,trim=0 .8cm 0 0]{IdentitySmothingCrossing.pdf}   \,\right\rangle + (A^2 + B^2 + ABz)\left \langle \,\includegraphics[scale=0.12,trim=0 .8cm 0 0]{TanlgeSmothingCrossing.pdf} \,\right\rangle $,
\item $\left\langle \includegraphics[scale=0.2,trim=0 0.4cm 0 0]{Curler1.pdf} \right\rangle = (Az + B)\left \langle \includegraphics[scale=0.2,trim=0 0.4cm 0 0]{Curler0.pdf} \right\rangle$,
\item $\left\langle\includegraphics[scale=0.2,trim=0 0.4cm 0 0]{Curler2.pdf} \right\rangle = (Az + B)\left \langle \includegraphics[scale=0.2,trim=0 0.4cm 0 0]{Curler0.pdf}\right\rangle$.
\end{enumerate}
\end{lemma}
\begin{proof}
See \cite[Proposition 3.3]{{kaSNE}}.
\end{proof}

According to (1) Lemma \ref{respectR1R2}, in order  to obtain that the Bracket polynomial respects the move R2  we must  have:
$AB=1$ and $ABz +  A^2 + B^2 =0$, or equivalently:
\begin{equation}\label{conditionsABz}
B=A^{-1}\quad \text{and}\quad z  = -(A^2 + A^{-2}).
\end{equation}
\begin{proposition}
Under the conditions of (\ref{conditionsABz}), the Bracket polynomial   agrees with  the Reidemeister move R3.
\end{proposition}
\begin{proof}
See \cite[Corollary 3.4]{{kaSNE}}.
\end{proof}
\begin{proposition}
Under the conditions of (\ref{conditionsABz}), we have:
\begin{enumerate}
\item $\left\langle\includegraphics[scale=0.2,trim=0 0.4cm 0 0]{Curler1.pdf}  \right\rangle = -A^3 \left\langle \includegraphics[scale=0.2,trim=0 0.4cm 0 0]{Curler0.pdf} \right\rangle $,
\item $\left\langle \includegraphics[scale=0.2,trim=0 0.4cm 0 0]{Curler2.pdf}  \right\rangle = -A^{-3} \left\langle \includegraphics[scale=0.2,trim=0 0.4cm 0 0]{Curler0.pdf}  \right\rangle $.
\end{enumerate}
\end{proposition}
\begin{proof}
See \cite[Corollary 3.4]{{kaSNE}}.
\end{proof}
Denote by $\langle\hspace{.5cm}\rangle$ the function
$$
\begin{array}{ccll}
\langle\hspace{.5cm}\rangle: & \vert\mathfrak{D}\vert&\longrightarrow &\ZZ[A,A^{-1}]\\
&D&\mapsto &\langle D\rangle
\end{array}
$$
This function is called function Bracket polynomial and is characterized in the following theorem.
\begin{theorem}
The function Bracket polynomial $\langle \quad\rangle$ is the unique function that  satisfies:
\begin{enumerate}
\item $\langle \bigcirc\rangle =1$,
\item $\langle \bigcirc \sqcup D\rangle =-(A^{-2} + A^2)\langle \bigcirc\rangle$,
\item $\left\langle \,\includegraphics[scale=0.12,trim=0 .8cm 0 0]{PositiveCrossing.pdf} \, \right\rangle  = A\langle  \, \includegraphics[scale=0.12,trim=0 .8cm 0 0]{IdentitySmothingCrossing.pdf}   \,\rangle + A^{-1}\langle  \, \includegraphics[scale=0.12,trim=0 .8cm 0 0]{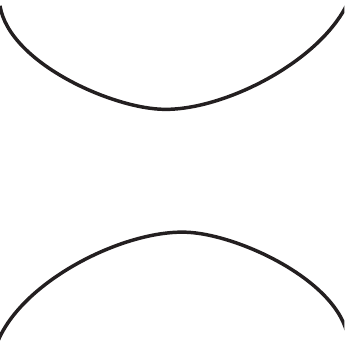}   \,\rangle$.
\end{enumerate}
\end{theorem}

\subsection{}
Let $D$ be a diagram of the oriented link $L$,  we define $f(L)$ as follows:
$$
f(L) := (-A^3)^{-w(D)}\langle \vert D\vert \rangle \in \ZZ[A, A^{-1}].
$$
Then, thanks to the  Lemma \ref{KauffmanRecipe}, we have the following theorem.
\begin{theorem}[{\cite[Proposition 2.5 ]{kaTAMS}}]\label{JonesLAKauffman}
The map $f:L \mapsto f(L)  $ is an invariant of oriented links.
\end{theorem}
\begin{remark}\rm
By making $A=t^{-1/4}$, we have that $f(L)$ becomes the  seminal  Jones polynomial of oriented links. It is a very interesting point, that such an  important modern mathematical object can be constructed in such  a simple way.
\end{remark}

The invariant $f$ is useful to detect if a link is equivalent to that obtained by reflection,  also called mirror image. Given a link $L$, its mirror image  is the link $L^{\ast}$ obtained by  reflection of $L$ in a plane.  Notice that the diagrams  of $L$ are in bijection with the diagrams of $L^{\ast}$: a  diagram $D$ of $L$ determines the diagram $D^{\ast}$ of $L^{\ast}$ obtained by exchanging the positive  with  the negative crossing in $D$.
\begin{definition}\rm
If $L$ and $L^{*}$ are ambient isotopic we say that the links are amphicheiral, otherwise  we say that the links are cheiral.
\end{definition}
\begin{example}
In the first row of the figure below we have, respectively, the Hopf link, the trefoil and the figure--eight. In the second row their respective reflected. The Hopf link and the figure--eight are amphicheiral and the trefoil is cheiral.
\begin{center}
\begin{figure}[H]
\includegraphics
[scale=.4]
{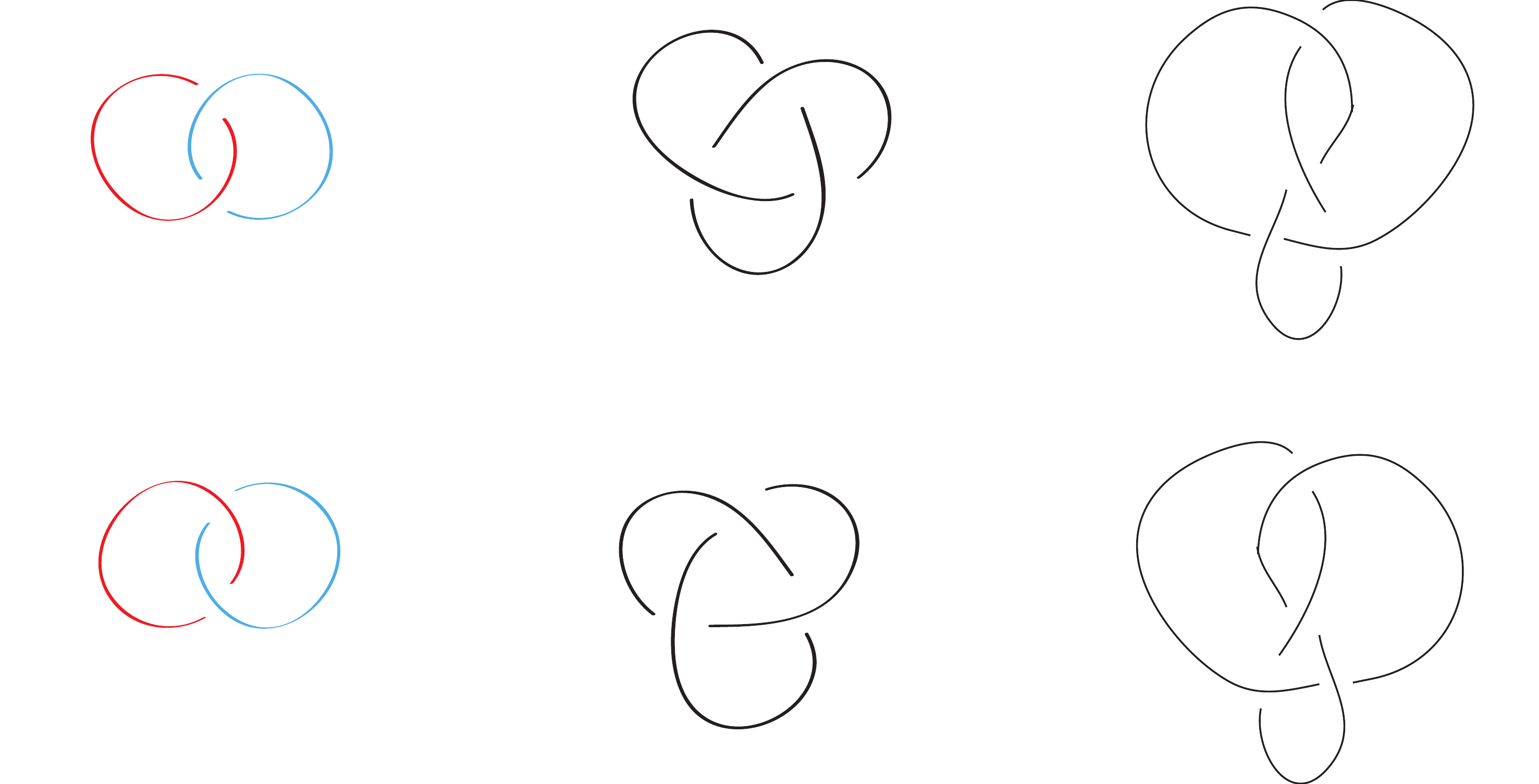}
\end{figure}
\end{center}
\end{example}
\begin{proposition}
Let $L$ an oriented link and $D$ a diagram of $L$, we have:
\begin{enumerate}
\item $\left\langle \vert D^{\ast} \vert \right\rangle = \left\langle \vert D \vert \right\rangle$,
\item $f(L^{\ast}) = f(L)$.
\end{enumerate}
\end{proposition}

\section{Links via braids}
Another way to study knot theory is through  braids. Braids were introduced by E. Artin and the equivalence between braids and knots is due to two theorems:  the Alexander and the Markov theorems. This section is a necessary compilation, for this  exposition, on the equivalence  between knot theory and braid theory.

\smallbreak
Throughout these notes we denote by $\mathtt{S}_n$ the symmetric group on $n$ symbols and we denote by $\mathtt{s}_i$ the elementary transposition $(i, i+1)$.  Recall that the  Coxeter presentation of $\mathtt{S}_n$, for $n>1$, is that with generators $\mathtt{s}_1, \ldots ,\mathtt{s}_{n-1}$ and the following relations:
\begin{eqnarray*}
\mathtt{s}_i\mathtt{s}_j & = & \mathtt{s}_j\mathtt{s}_i \quad \text{for}\quad \vert i-j\vert >1,\\
\mathtt{s}_i\mathtt{s}_j \mathtt{s}_i & = & \mathtt{s}_j\mathtt{s}_i \mathtt{s}_j\quad \text{for}\quad \vert i-j\vert =1.
\end{eqnarray*}

\subsection{}
Let $P_1, \ldots ,P_n$ be points in a plane  and $Q_1,\ldots , Q_n$ points in a\-no\-ther plane parallel to the first.
A $n$--geo\-me\-trical braid  is a collection  of  $n$  arcs $a_1, \ldots , a_n$ connecting the initial points $P_i$'s  with  the  ending  points $Q_{w(i)}$, where $w\in \mathtt{S}_n$,  such that:
\begin{enumerate}
\item for every different  $i$ and $j$ the arcs $a_i$ and $a_j$ are disjoint,
\item every plane parallel to  the plane containing the points $P_i$'s meets each arc in only one point.
\end{enumerate}
Geometrically, the arcs cannot be in the following situations:
\begin{center}
\begin{figure}[H]
\includegraphics
[scale=.5]
{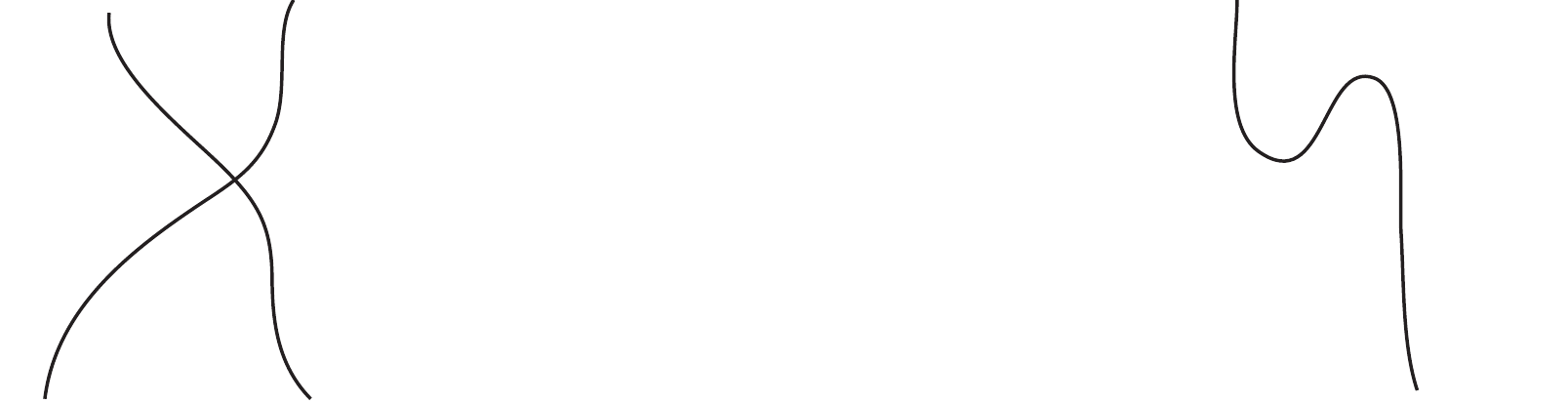}
\end{figure}
\end{center}

Define $\mathfrak{B}_n$ as the set formed by the $n$--geometrical braids. We define on  $\mathfrak{B}_n$ the equivalence relation, denoted by $\approx$,  given by continues deformation. That is, two geometrical braids $\alpha$ and $\beta$ are equivalents if there exists a family of geometrical braids $\{\gamma_t\}_{t\in\left[0,1\right]}$ such that
$\gamma_0= \alpha$ and $\gamma_1=\beta$.

As in knot theory we can translate, equivalently, the  geometrical braids to diagrams of them.
More precisely,  a {\it diagram of a geometrical braid} is  the image of a generic projection of the braid in a plane. Thus, we have the analogous of Reidemeister theorem for braids.
\begin{theorem}
Two geometrical braids are equivalent by $\approx$  if and only if their diagrams are equivalents, that is, every diagram  of one of  them can be transformed in a diagram of  the other,  by using a finite number of times the following  replacement:
\begin{center}
\begin{figure}[H]
\includegraphics
[scale=.5]
{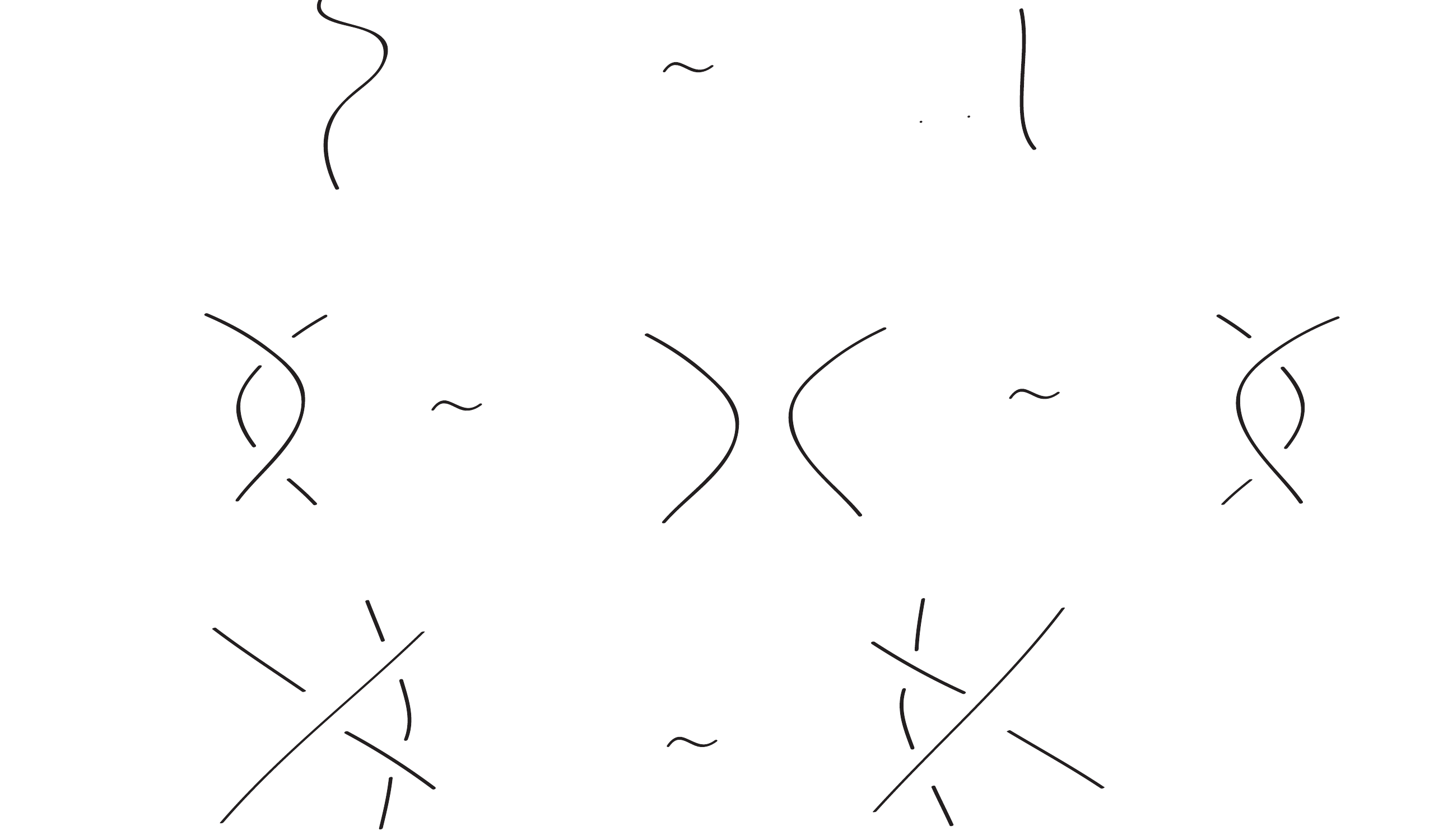}
\end{figure}
\end{center}
\end{theorem}
\begin{notation}
We shall use also  $\mathfrak{B}_n$ to denote  the set of diagrams of $n$--geometrical braids and $\approx$ to denote the equivalence of diagrams of braids.
\end{notation}

Now, we can define the \lq product by concatenation\rq\ between $ n $--geo\-me\-trical braids; more precisely, given $ \alpha $ and $ \beta $ in $ B_n $ we define by $ \alpha \beta \in \mathfrak {B} _n $ as that one defined by rescaling  the result of the $ n $--geometrical braid obtained by identifying the ending points of $ \alpha $ with the initial points of $ \beta $.
\begin{lemma}\label{rightdef}
For $\alpha$, $\alpha'$, $\beta$ and $\beta'$ in $\mathfrak{B}_n$ such that
$\alpha \approx \alpha'$ and $\beta\approx \beta'$, we have $\alpha\beta\approx\alpha'\beta'$.
\end{lemma}
\begin{proof}
We have $\alpha\beta \approx \alpha'\beta$ and also $\alpha'\beta\approx \alpha'\beta'$, then  $\alpha\beta\approx\alpha'\beta'$.
\end{proof}
Denote by $B_n$ the set of equivalence class of $\mathfrak{B}_n$ relative  to $\approx$; thus the elements of $B_n$ are the equivalence classes $[\alpha]$ of $\alpha\in\mathfrak{B}_n$. The Lemma \ref{rightdef} allows to  pass the product  by concatenation of $\mathfrak{B}_n$ to $B_n$:
$$
[\alpha][\beta] = [\alpha\beta]\qquad (\alpha , \beta\in B_n).
$$
\begin{theorem}
$B_n$ is a  group with the product by concatenation.
\end{theorem}
From now on we denote $[\alpha]$ simply by $\alpha$. Thus,
observe that the identity can be pictured by:

\vspace{2cm}

\noindent and the inverse $\alpha^{-1}$ of $\alpha$  is obtained by a reflection of $\alpha$:

\vspace{2cm}

\begin{remark}\rm
Observe that  $B_1$ is the trivial group and $B_2$ is the group of $\ZZ$.
\end{remark}

For $1\leq i\leq n-1$, denote by $\sigma_i$'s  the following {\it elementary braid}:

\vspace{3cm}
\begin{theorem}[Artin]
For $n>1$, $B_n$ can be presented by generators $\sigma_1,\ldots , \sigma_{n-1}$ and the following relations:
\begin{eqnarray}\label{braid1}
\sigma_i\sigma_j & = & \sigma_j\sigma_i \quad \text{for}\quad \vert i-j\vert >1,\\
\label{braid2}
\sigma_i\sigma_j \sigma_i & = & \sigma_j\sigma_i \sigma_j\quad \text{for}\quad \vert i-j\vert =1.
\end{eqnarray}
\end{theorem}
 An immediate consequence of the above is that  we have the following epimorphism:
\begin{equation}\label{BnontoSn}
 \mathsf{p}: B_n\longrightarrow \mathtt{S}_n,\quad \text{defined by   mapping $\sigma_i	\mapsto \mathtt{s}_i$}.
\end{equation}

Notice that for every $n$ we have a natural monomorphism $\iota_n$ from $B_n$ in $B_{n+1}$, where  for every braid
$\sigma\in B_n$, $ \iota_n(\sigma)$ is the  braid  of $B_{n+1}$ coinciding with $\sigma$ up to $n^{th}$ strand and having one more strand  with no crossing with the preceding strand.

\begin{notation}
We denote by $B_{\infty}$ the group obtained as the  inductive limit of $\{(B_n, \iota_n)\}_{n\in \NN}$.
\end{notation}

\subsection{}
Given $\alpha\in B_n$, the identification of the initial points with the end points of $\alpha$ determines a diagram of oriented links, which is denoted by $\widehat{\alpha}$; this process of identification is known as the closure of a braid. Thus we have  the \lq function closure\rq\
$$
\widehat{\,\,}: B_{\infty}\longrightarrow \mathfrak{L}, \quad \alpha\mapsto \widehat{\alpha}.
$$
The proof  that the  function closure is epijective is due to Alexander; we will outline this  proof, since it gives  an efficient  method to compute the preimage, by $\widehat{\,\,}$, of a given link.

\begin{theorem}[Alexander, 1923]\label{Alexandertheorem}
Every  oriented link  is the closure of a braid.
\end{theorem}
\begin{proof}
The sketch of the proof is as follows. Suppose we have a diagram $D$ of an oriented link $L$.
\begin{enumerate}
\item
We fix a point $O$ in the plane  not lying on any arc of $D$, such that  a point moving along  each component of the link  is  always seen  from $O$     going counterclockwise (or clockwise). Alexander proved that such a point $O$ always exists in  the  isotopy  class  of  the  link diagram of $D$.
\item
We divide the diagram  in sectors, by rays  starting from $O$, with the condition  that each sector  contains only one crossing.
\item
Finally, we open the diagram along one of the  rays obtaining  a braid whose closure is the diagram $D$.
\end{enumerate}
\end{proof}
\begin{center}
\begin{figure}[H]
\includegraphics
[scale=.3]
{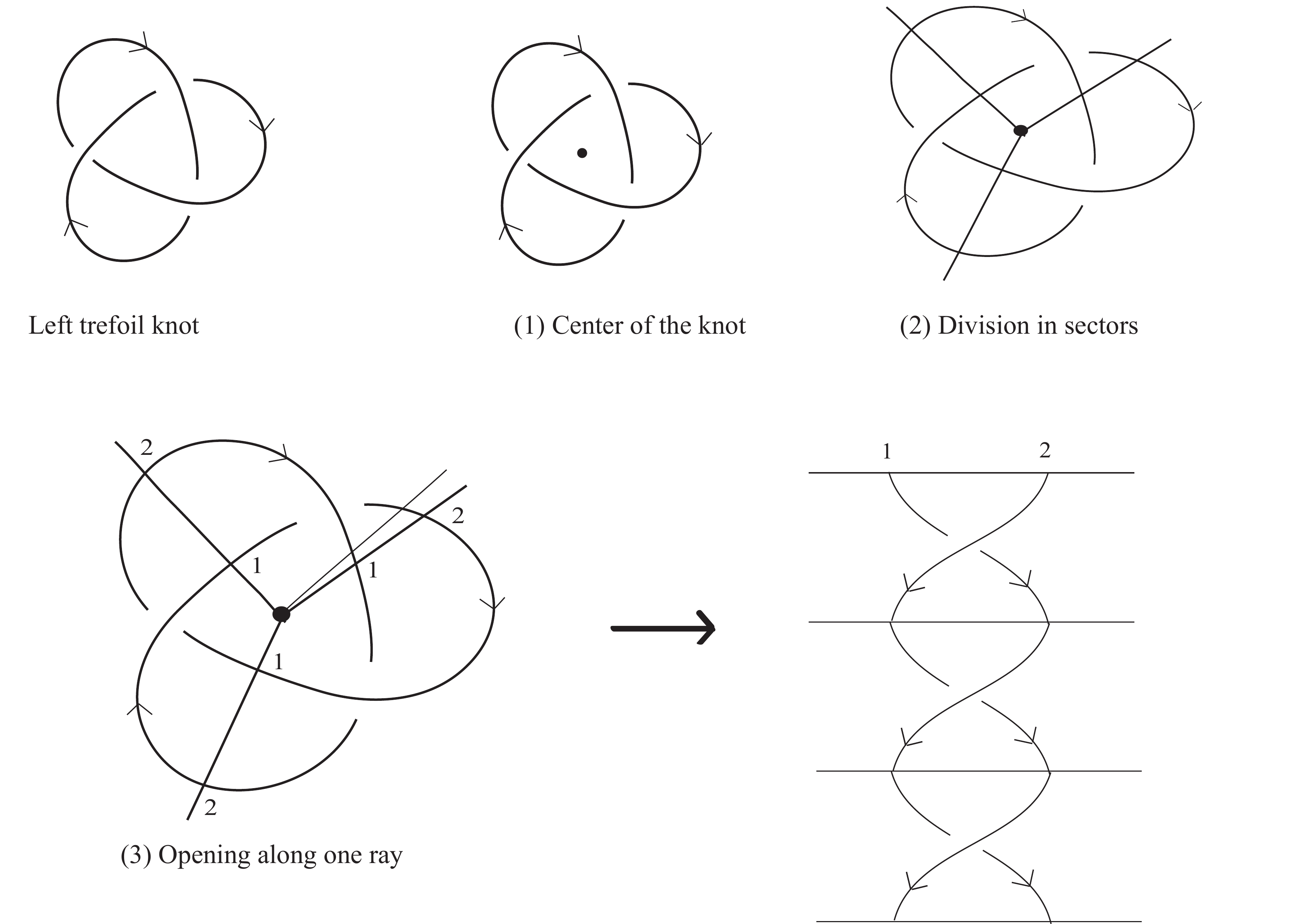}
\caption{}
\label{AlexanderTheorem}
\end{figure}
\end{center}
\begin{example}
Fig. \ref{AlexanderTheorem}  shows  that the  left trefoil  is the closure of the braid $\sigma_1^3$.
\end{example}

In order to describe the links through braids we need to know when the closure of two braids yields the same link. In fact the map closure is not injective; for instance the braids $ 1_{B_n}$ and $\sigma_{n-1}\in B_n$ yield the same link. The answer to which braids in $B_{\infty}$ yield the same link is due to Markov.

Denote by $\sim_M$ the equivalence relation on $B_{\infty}$ generated by the following replacements
(also called moves):
\begin{enumerate}
\item M1: $\alpha\beta$ can be replaced by $\beta\alpha$ (commutation),
\item M2:  $\alpha$ can be replaced by $\alpha\sigma_n$ or by $\alpha\sigma_n^{-1}$ (stabilization),
\end{enumerate}
where $\alpha$ and $\beta$ are in $B_n$.

The relation $\sim_M$ defines in fact an equivalence relation in $B_{\infty}$. Two elements in the same $\sim_M$--class are called Markov equivalent.

\begin{theorem}[Markov]\label{MarkovTh}
For $\alpha$ and $\beta$ in $B_{\infty}$, we have:
$\widehat{\alpha}$ and $\widehat{\beta}$ are  ambient isotopic links  if and only if $\alpha$ and $\beta$ are Markov equivalent.
\end{theorem}
From the  Theorems \ref{Alexandertheorem} and \ref{MarkovTh}, it follows  that:
\begin{corollary}\label{LinksAsBraids}
There is a bijection between    $B_{\infty}/\sim_M$ and $\mathfrak{L}/\sim$ defined through  the mapping $\alpha\mapsto \widehat{\alpha}$.
\end{corollary}
\begin{remark}\rm\label{Invariant}
The Markov theorem says that  constructing  an invariant for links  is equivalent to defining a map $\I: B_{\infty}\longrightarrow \mathrm{Set}$, such that for all $\alpha$ and $\beta$ in $B_n$,  agrees  with the replacements of Markov M1 and M2, that is,  $\I$ satisfies:
\begin{enumerate}
\item $\I(\alpha\beta) = \I(\beta\alpha)$,
\item $\I(\alpha\sigma_n)= \I(\alpha)= \I(\alpha\sigma_n^{-1})$.
\end{enumerate}
\end{remark}
\section{Hecke algebra}
Let $\FF$ be a field, from now on   the denomination  {\it $\FF$--algebra} mean an associative unital, with unity 1, algebra over the field $\FF$; thus we can regard  $\FF$ as a subalgebra of the center of the algebra.
\smallbreak
Given a group  $G$, we denote by $\FF G$ the $\FF$--algebra known as the group algebra   of $G$ over $\FF$. Recall that the set $G$ is a linear basis for
 $\FF G$, regarded  as a $\FF$--vector space.  Also recall that if $G$ has a presentation $\langle X; R\rangle$, then the $\FF$--algebra  $\FF G$ can be presented by generators $X$ and  the relations in $R$.

\subsection{}Let $\u$ be an indeterminate in $\CC$ and set  $\KK$ the field of the rational functions $\CC(\u)$. For  $n\in \NN$, the Hecke algebra, denoted by $\Hn(\u)$ or simply $\Hn$,  is defined by $\mathrm{H}_1=\KK$ and for $n>1$ as  the   $\KK$--algebra  presented  by generators $h_1,\ldots , h_{n-1}$ and the relations:
\begin{eqnarray}
\label{hecke1}
h_ihj & = & h_jh_i \quad \text{for}\quad \vert i-j\vert >1,\\
\label{hecke2}
h_ih_j h_i & = &h_jh_i h_j\quad \text{for}\quad \vert i-j\vert =1,\\
\label{hecke3}
h_i^2 & = & \u + (\u-1)h_i \quad \text{for all $i$}.
\end{eqnarray}
The  $h_i$'s are invertible,  indeed we have
\begin{equation}\label{inverseh}
h_i^{-1} =(\u^{-1}-1) + \u^{-1}h_i.
\end{equation}
\begin{remark}\rm\label{remarkHecke}
\begin{enumerate}
\item Taking $\u$ as power of a prime number,  the Hecke algebra above appears in representation theory as a  centralizer of a natural representation associated to the action of the finite general linear group on the variety of flags. This feature of the Hecke algebra will be the key point to construct here certain new invariants of links by using other Hecke algebra or other algebras of type Hecke. Consequently, the next subsection will be devoted to present  the Hecke algebras in the context of representation theory capturing in particular the Hecke algebra defined above.
\item The natural map $\sigma_i \mapsto h_i$ defines an algebra epimorphism  from $\KK B_n$ to $\Hn$. Then,
we have  that  the Hecke algebra $\Hn$ is the quotient of $\KK B_n$ by the two sided  ideal generator by
$$
\sigma_i^2 - \u \ - (\u-1)\sigma_i\quad \text{for}\quad 1\leq i\leq n-1.
$$
\item
By taking the specialization $\u=1$, the Hecke algebra becomes the group algebra of the symmetric group.
For this reason  the Hecke algebra is known also as  a deformation of the symmetric group.
\end{enumerate}
\end{remark}

We construct now a basis of the Hecke algebra; this basis is constructed in an inductive way and  is used to prove that the algebra  supports a Markov trace. We start with the following lemma.
\begin{lemma}\label{atmost}
In $\Hn$ every word in $1, h_1, \ldots , h_{n-1}$ can be written as a linear combination of words in the $1$ and the $h_i$'s  such that each of them  contains at most one $h_{n-1}$. Hence, $\Hn$ is finite dimensional.
\end{lemma}
\begin{proof}
The proof is by induction on $n$. For $n=2$ the lemma holds since $H_2$ is the algebra generated by $1$ and $h_1$.
Suppose now that the lemma is valid for every $\mathrm{H}_k$, with $k< n+1$. We  prove the lemma for $n+1$; let $M$  be
a word in $1, h_1, \ldots , h_n$ containing  two  times $h_n$, then we can write
$$
M= M_1 h_nM_2h_nM_3,
$$
where $M_i$'s are words in $1, h_1, \ldots , h_{n-1}$. But now, using the induction hypothesis we have to consider two  situations   according to $M_2$ contains none  or only one $h_{n-1}$. If $M_2$ does not contain $h_{n-1}$, we have $ M_1h_nM_2h_nM_3=M_1h_n^2M_2M_3 = M_1(\u + (\u-1)h_n)M_2M_3$.; hence $M= \u M_1M_2M_3 + (\u-1)M_1h_nM_2M_3$, thus
$M$ is reduced as the lemma claims. On the other hand, if  $M_2$ contains only one $h_{n-1}$, we can write it as
$
M_2= M'h_{n-1}M''
$,
where $M'$ and $M''$ are words in $1, h_1, \ldots , h_{n-2}$; so
$
M = M_1 h_n (M'h_{n-1}M'')h_nM_3
$; by applying now (\ref{hecke1}) and (\ref{hecke2}) we obtain $M = M_1  M'h_{n-1}h_nh_{n-1}M''nM_3$, then $M$ is as
the lemma  claims.

In the case that $M$ contains more than two generators  $h_n$,  we reduce two of them using the argument above; so  arguing   inductively we deduce that $M$ can be written as the lemma is claiming.
\end{proof}
In $\Hn$, define: $C_1=\{1, h_1\} $ and $C_i =\{1, h_ix\,;\, x \in C_{i-1}\}$, for $2\leq i\leq n-1.$
\begin{definition}
The elements  $\mathfrak{n}_1\mathfrak{n}_2\cdots \mathfrak{n}_{n-1}\in \Hn$, with $\mathfrak{n}_i\in C_i$, are called normal words. This set formed by the normal words will be denoted by $\mathcal{C}_n$.
\end{definition}
Observe  that:
\begin{equation}\label{Cn}
\mathcal{C}_n = \mathcal{C}_{n-1}\cup \{xh_{n-1}h_{n-2}\cdots h_i\, ; x\in \mathcal{C}_{n-1}\, 1\leq i\leq n-1\}.
\end{equation}
\begin{theorem}\label{basisHecke}
The set $\mathcal{C}_n$ is a linear basis of  $\Hn$. In particular, the dimension of $\Hn$ is $n!$.
\end{theorem}
\begin{proof}
We will prove, by induction on $n$,  that $\mathcal{C}_n$ is a spanning set of $\Hn$. For $n=2$ the theorem is clear. Suppose now that the theorem is true for every $k<n$. From  Lemma \ref{atmost} it follows  that  $\Hn$ is linearly spanned  by the   elements of the form (i) and (ii):
$$
\text{(i)}\, M_0\qquad \text{and} \qquad \text{(ii)}\,M_1h_nM_2,
$$
where $M_i$'s are words in $1, h_1, \ldots , h_{n-2}$. By  the induction hypothesis, follows that  $M_i$'s are linear combination of elements of $\mathcal{C}_{n-1}$, so we can suppose that $M_i$'s belong to $\mathcal{C}_{n-1}$. Thus, it is enough to prove that the elements of (ii) are a  linear combination of the elements of $\mathcal{C}_n$ (notice that $M_0\in \mathcal{C}_n$). Set $M_2=\mathfrak{n}_1\mathfrak{n}_2\cdots \mathfrak{n}_{n-2}$, where $\mathfrak{n}_i\in C_i$; we have
$$
M_1h_{n-1}M_2 = M_1h_{n-1}\mathfrak{n}_1\mathfrak{n}_2\cdots \mathfrak{n}_{n-2} =M_1\mathfrak{n}_1\mathfrak{n}_2\cdots h_{n-1}\mathfrak{n}_{n-2}.
$$
From the induction hypothesis  $M_1\mathfrak{n}_1\mathfrak{n}_2\cdots$ is a linear combination of elements of $\mathcal{C}_{n-1}$ and notice that $h_{n-1}\mathfrak{n}_{n-2}\in \mathcal{C}_{n-1}$. So, having in mind the second observation of (\ref{Cn}) we deduce that the elements in (ii) belong to the linear span of $\mathcal{C}_n$.

Linear independency  (LATER)

\end{proof}
The above theorem  and (\ref{Cn}) imply that we have a natural  tower of algebras
$$
\mathrm{H}_1=\KK \subset \mathrm{H}_2\subset \cdots \subset\Hn\subset\Hnn\subset \cdots
$$
We will denote by $\mathrm{H}_{\infty}$ the inductive limit associated to this tower.

Notice that the inclusion of algebras  $\mathrm{H}_{n}  \subset \mathrm{H}_{n+1} $ allows to obtain a structure of  $(\mathrm{H}_{ n}, \mathrm{H}_{ n})$--bimodule for $\mathrm{H}_{n+1} $; further, we can  consider the $(\mathrm{H}_{ n}, \mathrm{H}_{ n})$--bimodule $\Hn\otimes_{\mathrm{H}_{n-1} }\Hn$ since $\Hn$ is a $(\mathrm{H}_{ n}, \mathrm{H}_{ n-1})$--bimodule and  also $(\mathrm{H}_{ n-1}, \mathrm{H}_{ n})$--bimodule.
\begin{proposition}\label{}
The dimension of the $\KK$--vector space $\Hn\otimes_{\mathrm{H}_{n-1} }\Hn$ is at most $n!n$.
\end{proposition}
\begin{proof}
 Theorem \ref{basisHecke} implies that every element in $\Hn\otimes_{\mathrm{H}_{n-1} }\Hn$ is a  $\KK$--linear combination of elements   of the form $a\otimes b$, where $a, b\in \mathcal{C}_n$. Now, we have two possibilities: $b$ is in $\mathcal{C}_{n-1}$ or $b= xh_{n-1}\cdots h_i$, with $x\in \mathcal{C}_{n-1}$ (see (\ref{Cn})). Now, if  $b\in \mathcal{C}_{n-1}$, we have $a\otimes b= ab\otimes1$ and in the other case we can write  $a\otimes b= ax \otimes h_{n-1}\cdots h_i$. Therefore, every element of $\Hn\otimes_{\mathrm{H}_{n-1} }\Hn$ is a linear combination of elements of the form $a\otimes 1$ and $a\otimes h_{n-1}\cdots h_i$, where $a\in \mathcal{C}_n$ and $1\leq i\leq n-1$. Hence the proof follows.
\end{proof}
The following lemma will be used in the next section.
 \begin{lemma}\label{bimoduleHecke}
 The  map $\phi: \Hn \oplus \Hn\otimes_{\mathrm{H}_{n-1} }\Hn \longrightarrow \Hnn$,  defined by
 $$
 x+ \sum_i y_i\otimes z_i\mapsto x + \sum_i y_ih_n z_i
 $$
 is an isomorphism of $(\mathrm{H}_{ n}, \mathrm{H}_{ n})$--bimodules.
 \end{lemma}
\begin{proof}

\end{proof}

\subsection{}From now on  $\z$ denotes a  new variable commuting with $\u$.
\begin{theorem}[Ocneanu]\label{Ocneanu}
There exists a unique family of linear maps $\tau=\{\tau_n\}_{n\in \NN}$, where $\tau_n: \Hn\longrightarrow\KK(\z)$ is defined inductively by the following rules:
\begin{enumerate}
\item $\tau_1 (1) =1$,
\item $\tau_n(ab) =\tau_n (ba)$,
\item $\tau_{n+1}(ah_nb) = \z\tau_n (ab)$,
\end{enumerate}
where $a,b\in \Hn$.
\end{theorem}
\begin{proof}
The definition of $\tau_n$ is based on the homomorphism $\phi$ of Lemma \ref{bimoduleHecke}. For $n=1$, $\tau_1$ is  defined as the identity on $\KK$. Now, given $a\in \Hnn$, the Lemma \ref{bimoduleHecke} says that, we can write uniquely, $a = \phi (x + \sum_iy_i\otimes z_i) $, then we define $\tau_{n+1}$ by
$$
\tau_{n+1}(a) := \tau_n(x) + \z\sum_i\tau_{n}(y_iz_i).
$$
For instance $\tau_2(h_1) = \z$, since $\phi (0 + 1\otimes 1) =h_1$.

For every $a,b\in \Hn$, we have $\phi(a\otimes b)=ah_nb$,  then the rule (3) is satisfied.

We are going to check now the rule (2) BLABLA...
\end{proof}

\section{The Homflypt polynomial}\label{sectionHomflypt}
We  show now the construction of  the Homflypt polynomial due to V. Jones. This Jones construction gives a method (or Jones recipe) which is our main tool  to  construct   invariants.
\smallbreak
\subsection{}
We have a natural representation from $B_n$ in $\Hn$, defined by mapping $\sigma_i$ in $h_i$, however we need to consider a
slightly more general  representation, denoted by $\pi_{\theta}$  and defined by mapping $\sigma_i$ in $ \theta h_i$, where $\theta$ is  a scalar factor; the reason for taking this factor $\theta$ will be clear soon. Now, composing $\pi_{\theta}$  with the Markov trace $\tau_n$, we have the maps
$$
\tau_n\circ \pi_{\theta}: B_n \longrightarrow \KK(\z) \quad (n\in \NN).
$$
This family of maps yields  a unique map $\mathrm{X}_{\theta}$ from $B_{\infty}$ to $\KK(\z)$.  Now, according to Remark \ref{Invariant}, the   function $\mathrm{X}_{\theta}$ defines  an invariant of links, if it agrees with the replacements of Markov M1 and M2, or equivalently,  for every $n$ the function $\tau_n\circ \pi_{\theta}$  agrees with M1 and M2. The fact that $\pi_{\theta}$ is a homomorphism and the rule (2) of the Ocneanu trace implies that, for every $n$ and $\theta$ the maps $\tau_n\circ \pi_{\theta}$  agree with the Markov replacement M1.  For the replacement M2 we note that, in particular, the maps $\tau_n\circ \pi_{\theta}$ must  satisfy:
$$
(\tau_n\circ \pi_{\theta})(\sigma_n)=(\tau_n\circ \pi_{\theta})(\sigma_n^{-1})\quad\text{for all}\; n.
$$
We have $(\tau_n\circ \pi_{\theta})(\sigma_n)= \theta\z$  and from (\ref{inverseh}), we get:
$$
(\tau_n\circ \pi_{\theta})(\sigma_n^{-1})= \theta^{-1}\tau_n(h_n^{-1}) = \theta^{-1}((\u^{-1}) + \u^{-1}\z)
$$
Then we derive  that  $\theta\z=\theta^{-1}((\u^{-1}-1) + \u^{-1}\z)$, from where the factor $\theta$ satisfies:
\begin{equation}\label{RelationLambdaZeta}
\lambda:=\theta^2 =\frac{1-\u+\z}{\u\z} \quad \text{or equivalently}\quad \sqrt{\lambda} \z=(\u^{-1}-1) +\u^{-1}.
\end{equation}
So, extending the ground field $\KK(\z)$ to $\KK(\z,\sqrt{\lambda})=\mathbb{C}(\z, \sqrt{\lambda})$, the family of maps $\{\tau_n\circ \pi_{\sqrt{\lambda}}\}_{n\in \NN}$  agrees with
the Markov replacements M1 and M2. However, it  is desirable  that the invariant takes the values 1 on the unknot, that is, we want  $(\tau_n\circ \pi_{\sqrt{\lambda}}) (\sigma) =1$, for every $n$ and  every $\sigma$ whose closure is the unknot;
notice that the unknot is the closure of the braid $\sigma_1\sigma_2\cdots \cdots \sigma_{n-1}$, for all $n$. So, we have:
\begin{equation}\label{homflypFactor}
(\tau_n\circ \pi_{\sqrt{\lambda}}) (\sigma_1\sigma_2\cdots \cdots \sigma_{n-1} )= ( \theta \z)^{n-1} = ( \sqrt{\lambda}\,\z)^{n-1}.
\end{equation}
Then,  we need to  normalize $\tau_n\circ \pi_{\sqrt{\lambda}}$ by  $( \sqrt{\lambda}\,\z)^{n-1} $.
\begin{theorem}\label{Homflypt}
Let $L$ be  an oriented link obtained as the closure of the braid $\alpha\in B_n$. We define $\X : \mathfrak{L}\longrightarrow \mathbb{C}(\z, \sqrt{\lambda})$, by
\begin{equation}
\X (L) :=\left(\frac{1}{\sqrt{\lambda}\,\z}\right)^{n-1}(\tau_n\circ \pi_{\sqrt{\lambda}})(\alpha).
\end{equation}
Then  $\X$ is an invariant of ambient isotopy for oriented links.
\end{theorem}
\begin{proof}
Thanks to Corollary \ref{LinksAsBraids}, we need only to check that:
 $$
\text{(i)} \; \X(\widehat{\alpha}\; \widehat{\beta}) = \X(\widehat{\beta}\;\widehat{\alpha})\quad \text{and}\quad
  \text{(ii)} \; \X(\widehat{\alpha}) =  \X(\widehat{\alpha\sigma_n}) =   \X(\widehat{\alpha\sigma_n^{-1}}),
 $$
where $\alpha , \beta\in B_n$. Clearly   (i)  holds. We are going to check now only the second  equality of (ii), the checking of first equality is left to the reader. We have:
\begin{eqnarray*}
(\tau_{n+1}\circ \pi_{\sqrt{\lambda}})(\alpha\sigma_n^{-1}) & = &\tau_{n+1} (\pi_{\lambda}(\alpha)\pi_{\lambda}(\sigma_n^{-1}))
\, = \, (\sqrt{\lambda})^{-1}\tau_{n+1} (\pi_{\lambda}(\alpha)h_n^{-1})\\
& = & (\sqrt{\lambda})^{-1}\tau_{n+1} (\pi_{\lambda}(\alpha)((\u^{-1}-1) + \u^{-1}h_n))\\
& = & (\sqrt{\lambda})^{-1}((\u^{-1}-1) + \z\u^{-1}) \tau_{n} (\pi_{\lambda}(\alpha))\\
& = & (\sqrt{\lambda})^{-1}\sqrt{\lambda}\,\z\tau_{n} (\pi_{\lambda}(\alpha)) \quad (\text{see}\, (\ref{RelationLambdaZeta})).
\end{eqnarray*}
Then, by using now  the rule (3) of the Ocneanu trace, we get $(\tau_{n+1}\circ \pi_{\sqrt{\lambda}})(\alpha\sigma_n^{-1}) = \tau_{n+1} (\pi_{\lambda}(\alpha\sigma_n))$; hence  $\X(\widehat{\alpha\sigma_n}) =   \X(\widehat{\alpha\sigma_n^{-1}})$.
\end{proof}
\begin{example}
Trefoil
\begin{figure}
\end{figure}
\end{example}
\vspace{3cm}
\subsection{}
The Homflypt polynomial has a definition by skein rules. This definition is useful to  calculate it and also to study its relations with other invariants such as the Jones polynomial and the Alexander polynomial.

Denote by  $L_+, L_-$ and $L_0$ three oriented links  with, respectively, diagrams $D_+, D_-$ and $D_0$, which are different only  inside a disk, where they are respectively placed,  as shows Fig \ref{ConwayTriple}.
\begin{center}
\begin{figure}[H]
\includegraphics[scale=.5]{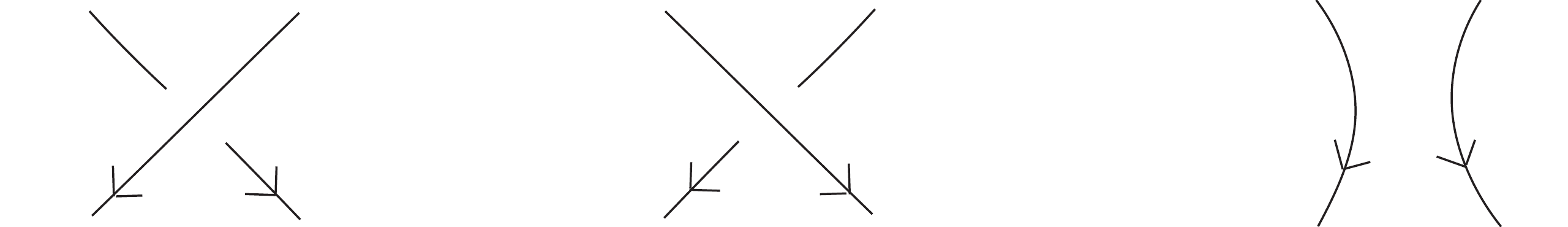}
\caption{ }\label{ConwayTriple}
\end{figure}
\end{center}
\vspace{-1cm}
The links  $L_+, L_-$ and $L_0$ are called a Conway triple; notice that in terms of braids,  they can be written as:
$$
L_+ = \widehat{w\sigma_i},\quad L_- = \widehat{w\sigma_i^{-1}}\quad \text{and}\quad L_0 = \widehat{w}
$$
for some  braid $w$ in $B_n$. 

Keeping the notation above, we are going to compute the Homflypt polynomial on a Conway triple. First, we have,
\begin{eqnarray*}
 \X(L_-)
  & = & 
  D^{n-1}(\sqrt{\lambda})^{e(w\sigma_i^{-1})} (\tau_n\circ\pi)(w\sigma_i^{-1}) \\
 &=&
 D^{n-1}(\sqrt{\lambda})^{e(w)}(\sqrt{\lambda})^{-1} \tau_n(\pi(w)h_i^{-1}),
\end{eqnarray*}
by considering now (\ref{inverseh}), we obtain:
$$
 \X(L_-) =D^{n-1}(\sqrt{\lambda})^{e(w)}(\sqrt{\lambda})^{-1}((\u^{-1}-1) (\tau_n\circ\pi)(w)+ \u^{-1}(\tau_n\circ\pi)(w\sigma_i))
$$
Also a direct computation yields:
 $$
\X(L_+) =D^{n-1}(\sqrt{\lambda})^{e(w)}\sqrt{\lambda} (\tau_n\circ\pi)(w\sigma_i),
 $$
 $$
 \X (L_0) = D^{n-1}(\sqrt{\lambda})^{e(w)}( \tau_n\circ\pi)(w).
 $$
From these last three equations we obtain the following proposition.
\begin{proposition}\label{SkeinX}
The invariant $\X$ satisfies the following skein relation.
$$
\frac{1}{\sqrt{\lambda \u}}\,\X (L_+) - \sqrt{\lambda \u}\,\X (L_-)  = \left( \sqrt{\u}- \frac{1}{\sqrt{\u}}\right)\X (L_0).
$$
\end{proposition}

\begin{theorem}\label{SkeinHomflypt}
There exists a unique  function 
$$\P:\mathfrak{L}\longrightarrow \ZZ [\t, \t^{-1}, \x ,\x^{-1}]$$ such that:
\begin{enumerate}
\item $\P(\bigcirc) = 1$,
\item $\t^{-1}\P (L_+) - \t\P (L_-)  = \x \P (L_0)$.
\end{enumerate}
\end{theorem}
\begin{proof}
After a suitable change of variables, $\X$  satisfies the defining properties of $\P$, so it remains to prove the uniqueness of $\P$.
\end{proof}

\begin{remark}\rm
Making $\x=\t$ the polynomial $\P$ becomes the Jones polynomial and making $\x=\t^{-1}$ the polynomial $\P$ becomes the Alexander polynomial.
\end{remark}
 
\section{Hecke algebras in representation theory}
\subsection{}
Let $G$ be a finite group. A complex representation of $G$ is a pair $(V,\rho)$, where $V$ is a finite dimensional  space over $\CC$ and $\rho$ is a homomorphism group from $G$ to $\mathrm{GL}(V)$. A subspace $U$ of $V$ is called $G$--stable if $\rho_g(U)=U$; the representation is called irreducible if the unique  $G$--stable subspaces are  trivial. It is well  known that every complex representation can be decomposed  as  a direct sum of stable subspaces, see  \cite{seGTM42}. A fundamental problem in representation theory  is:  given a representation of $G$,  write out  such a decomposition. A powerful tool that helps the  understanding of the decomposition of a  representation is  its centralizer, that is,  the algebra formed by the  endomorphisms of $V$ commuting with $\rho_g$, for all $g\in G$. The centralizer of the representation $(V,\rho)$ is denoted by $\mathrm{End}_G(V)$.
\smallbreak
Now, given a subgroup $H$ of $G$ we can  construct the so called natural or induced representation, of $G$ relative a $H$. More precisely,  this representation  can be made explicit  as $(\mathrm{Ind}_H^G1, \rho)$, where:
$$
\mathrm{Ind}_H^G1:= \{ f: G/H\longrightarrow \CC\, ; \, f\,\text{is  function}\}
$$
and
$$
\rho_g (f)(xH)=f(g^{-1}xH).
$$
The centralizer of this representation is known as the Hecke algebra of $G$ with respect to $H$ and is  usually denoted by $\mathcal{H}(G, H)$.
\smallbreak
A fundamental piece in the   theory of finite group of Lie type is the   representation theory of the finite general linear group
$G=\mathrm{GL}_n(\FF_q)$, where $\FF_q$  denotes the finite field with $q$ elements; and  an important family of irreducible  representation of $G$ appears in the decomposition of  $\mathrm{Ind}_B^G1$, where $B$ is the subgroup of $G$ formed by the upper triangular matrices. The centralizer $\mathcal{H}(G, B)$ of this representation was studied by N. Iwahori in the sixties, in  a more general context for $G$: the finite Chevalley groups.
In the case $G$ is as above, that is, the finite general group, the algebra $\mathcal{H}(G, B)$ corresponds  to
 those of type $A$ in the classification of Chevalley groups and the  Iwahori theorem for  $\mathcal{H}(G, B)$ is  the following.
\begin{theorem}[N. Iwahori, \cite{iwJFSUT}]\label{Iwahori}
The Hecke algebra $\mathcal{H}(G, B)$ can be presented, as $\CC$--algebra, by generators $\phi_1, \phi_2, \ldots , \phi_{n-1}$ and the following relations:
\begin{enumerate}
\item $\phi_i^2 = q +(q-1)\phi_i$, for all $i$,
\item $\phi_i\phi_j = \phi_j\phi_i$, for $\vert i- j\vert >1$,
\item  $\phi_i\phi_j\phi_i =\phi_j \phi_j\phi_i$, for $\vert  i-j \vert =1$.
\end{enumerate}
\end{theorem}

\smallbreak
Now, we want to  explain a little bit how the $\phi_i$'s  look and work. To better explain we shall  work  in a   more general setting.
Let $G$ be a finite group and set $X$ a finite $G$--space. Define $\CC(X)$ the $\CC$--vector space of all complex valued functions on $X$. Then, we have the representation $(\CC (X) , \rho)$ of $G$, where $(\rho_gf)(x):= f(g^{-1}x)$.  Consider now   the $G$--space $X\times X$ with action $g(x,y)=(gx,gy)$ and the vector space  $\mathcal{K}(X):=\CC(\X\times X)$, this vector space results to be an algebra with the convolution product:
$$
 I  \ast J : (x,y)\mapsto \sum_{z\in X}I(x,z)J(z,y)\qquad (I,J\in \mathcal{K}(X) ).
$$

\begin{theorem}
The function $\Phi$ from $\mathcal{K}(X)$ to $\mathrm{End}(\CC(X))$ is an algebra isomorphism, where $\Phi : I\mapsto \Phi_I$, is defined by
$$
(\Phi_If)(x):= \sum_{y\in X}I(x,y)f(y) \quad (f\in \CC(X), x\in X).
$$
\end{theorem}
\begin{proof}

\end{proof}
Recall that  $\mathrm{End}_G(\CC(X))$  denote the centralizer of the representation $(\CC(X), \rho)$ and denote now by $\mathcal{K}_G(X)$ the subalgebra of  all $I\in \mathcal{K}(X) $,  that are $G$--invariant, that is, $I(gx,gy)=I(x,y)$, for
$(x,y)\in X\times X$ and $g\in G$.
\begin{theorem}
The algebras $\mathrm{End}_G(\CC(X))$ and $\mathcal{K}_G(X)$ are isomorphic.
\end{theorem}
\begin{proof}
\end{proof}
Let $O_1, \ldots , O_{m}$ be the orbits of the $G$--space $X\times X$, hence the canonical basis of $\mathcal{K}_G(X)$ is
$\{K_1, \ldots ,K_m\}$, where the $K_i $'s are the function delta of Dirac,
$$
K_i(x,y) :=\left\{\begin{array}{l}
1 \quad \text{if} \;(x,y)\; \in O_i\\
0 \quad \text{if} \;(x,y)\; \notin O_i.
\end{array}\right.
$$
\begin{corollary}
 A basis of $\mathrm{End}_G(\CC(X))$ is the set $\{\phi_1,\ldots , \phi_m\}$, where
$$
(\phi_i f)(x) = \sum_{y \;\text{s.t.}\; (x,y)\in O_i}f(y).
$$
Hence, the dimension  of $\mathrm{End}_G(\CC(X))$  is the number of $G$--orbits of $X\times X$.
\end{corollary}
\begin{proof}
Notice that  $\phi_i :=\Phi(K_i)$, so the proof is clear.
\end{proof}
We shall finish the section by showing the basis $\phi_i$ in the case $G=\mathrm{GL}_n(\FF_q)$ and $X$ the variety of flags in $\FF_q^n$.
\begin{remark}\rm\label{remarkJones}
In \cite[p. 336 ]{joAM} Jones raised the question whether  his method of  construction of the Homflypt polynomial can be used for  Hecke algebras of  other type. In this light it is  natural to ask also whether   the Jones method can be used  still keeping $G=\mathrm{GL}_n(\FF_q)$ but changing $B$ for other notable subgroups of $G$. This is what we  did  to define new invariants for links. More precisely, we change  $B$ by his  unipotent part, that is, the  subgroup formed by the upper unitriangular matrices.
\end{remark}

\section{Yokonuma--Hecke algebra}
In this section we introduce   the Yokonuma--Hecke whose origin is in representation theory of finite Chevalley groups; indeed, it  appears  as a centralizer of a natural representation of a Chevalley group. With the aim to applying this algebra to knot theory, we noted that it  is linked, not only to the braid group, but also to the framed braid group.  We show an inductive  basis, the existence of a Markov trace  and the E--system, all that fundamental ingredients to construct our invariants.

\subsection{}
The Yokonuma--Hecke algebra comes from a centralizer  of the permutation representation associated to the finite general linear  group $\mathrm{GL}_n(\FF_q)$ with respect to the subgroup  consisting of the upper  unitriangular matrix, cf. Remark \ref{remarkJones}. This centralizer was studied by T. Yokonuma who obtained  a presentation  of it, analogous to that found by N. Iwahori for the Hecke algebra, see Theorem \ref{Iwahori}. However, the presentation  of Yokonuma was slightly modified in order to be applied  to knot theory; a little modification of this presentation defines the so called  Yokounuma--Hecke algebra.
\begin{definition}
The  Yokonuma--Hecke algebra, in short called Y--H algebra and denoted by  ${\rm Y}_{d,n}(\u)$, is defined as follows:
$\mathrm{Y}_{1,1}(\u) = \KK $,  and for $d, n\geq 2$ as the algebra presented  by  braid generators  $g_1, \ldots , g_{n-1}$  and  framing generators $t_1, \ldots , t_n$, subject to the following relations:
\begin{eqnarray}
\label{yh1}
 g_i g_j & = & g_j g_i \quad \text{for $\vert i-j\vert >1$}, \\
\label{yh2}
g_i g_jg_i & = & g_jg_ig_j \quad  \text{for $\vert i-j\vert =1$}, \\
\label{yh3}
t_i t_j & = & t_j t_i \quad  \text{for all $i,j$}, \\
\label{yh4}
 t_j g_i & = & g_it_{\mathtt{s}_i(j)}\quad  \text{for all $i,j$},\\
\label{yh5}
 t_i^d & = & 1 \quad  \text{for all $i$,} \\
\label{yh6}
 g_i^2 & = &  1 + (\u-1)e_i(1 + g_i) \quad  \text{for all $i$},
\end{eqnarray}

where  $\mathtt{s}_i= (i,  i+1)$ and
$$
e_i := \frac{1}{d}\sum_{s=0}^{d-1}t_i^st_{i+1}^{d-s}.
$$
\end{definition}
Whenever the variable $\u$ is irrelevant,  we shall denote $\Yn(\u)$  simply  by  $\Yn$.
\begin{remark}\label{remarkYn}
\begin{enumerate}
\item The elements $e_i$'s result to be idempotents, i.e. $e_i^2= e_i$ and the $g_i$'s are invertible,
\begin{equation}\label{inversgi}
g_i^{-1} = g_i + (\u^{-1}-1)e_i + (\u^{-1}-1)e_ig_i.
\end{equation}
These facts will be frequently used  in what  follows.
\item
Notice that for $d=1$, the algebra $\Yn$ becomes $\Hn$, while
for $n=1$, $\Yn$ becomes the group algebra of the cyclic group with $d$ elements.
\item
The mapping $g_i\mapsto h_i$ and $t_i\mapsto 1$, define an homomorphism algebra from $\Yn$ onto $\Hn$.
\end{enumerate}
 \end{remark}
\subsection{}
We have a natural representation of the braid group in the Y--H algebra since relations (i) and  (ii) correspond to the defining relations of the braid group. Moreover, the  relations (i)--(iv) correspond to the defining relations of the framing group.  To be more precise, the framed braid group, denoted by $\mathcal{F}_n$, is the semi--direct product $\ZZ^n\rtimes B_n$, where the action that defines the semidirect product  is  through the homomorphism $\mathsf{p}$ from $B_n$ onto $\mathtt{S}_n$ (see (\ref{BnontoSn})), that is:
$$
\sigma (z_1,\ldots, z_n) = (z_{\sigma(1)},\ldots, z_{\sigma(n)}),\quad  \text{where $\sigma(i):=\mathsf{p}(\sigma)(i)$}.
$$
In multiplicative notation  the group $\ZZ$  has the presentation $\langle t\,;\,-\rangle$ and the direct product $\ZZ^n$ has the presentation $\langle t_1,\ldots ,t_n\;;\;t_it_j =t_jt_i \quad \text{for all $i,j$} \rangle$. Consequently, $\mathcal{F}_n$ can be presented by  (braids) generators $\sigma_1,\ldots ,\sigma_{n-1}$ together with the (framing) generators $t_1,\ldots ,t_n$ subject to the  relations (\ref{braid1}), ( \ref{braid2}), $t_it_j =t_jt_i $, for all $i,j$ and the relations:
\begin{equation}
t_j \sigma_i  =  \sigma_it_{\mathtt{s}_i(j)}  \quad \text{for all $i,j$}.
\end{equation}

Now, because $\mathcal{F}_n$ is a semidirect product, we have that every element in it can be written in the form
$t_1^{a_1}\cdots t_n^{a_n}\sigma$,  where $\sigma\in B_n$ and the  $a_i$'s are integers called the framing.  Further, observe that:
\begin{equation}\label{productFn}
(t_1^{a_1}\cdots t_n^{a_n}\sigma)(t_1^{b_1}\cdots t_n^{b_n}\tau) = t_1^{a_1 + b_{\sigma(1)}}\cdots t_n^{a_n+ b_{\sigma(n)}}\sigma\tau.
\end{equation}

In diagrams, the element $t_1^{a_1}\cdots t_n^{a_n}\sigma$ can be represented by the usual diagram braid for $\sigma$ together
with $a_1,\ldots ,a_n$ written at   the top of the braid: $a_i$ is placed where  the strand $i$ starts. For instance,  Fig. \ref{FramedBraid} represents the framed braid $t_1^at_2^bt_3^ct_4^d  \sigma_2 \sigma_3^{-1} \sigma_1^{-1}\sigma_2^2\sigma_3$.
\begin{center}
\begin{figure}[H]
\includegraphics
[scale=.3]
{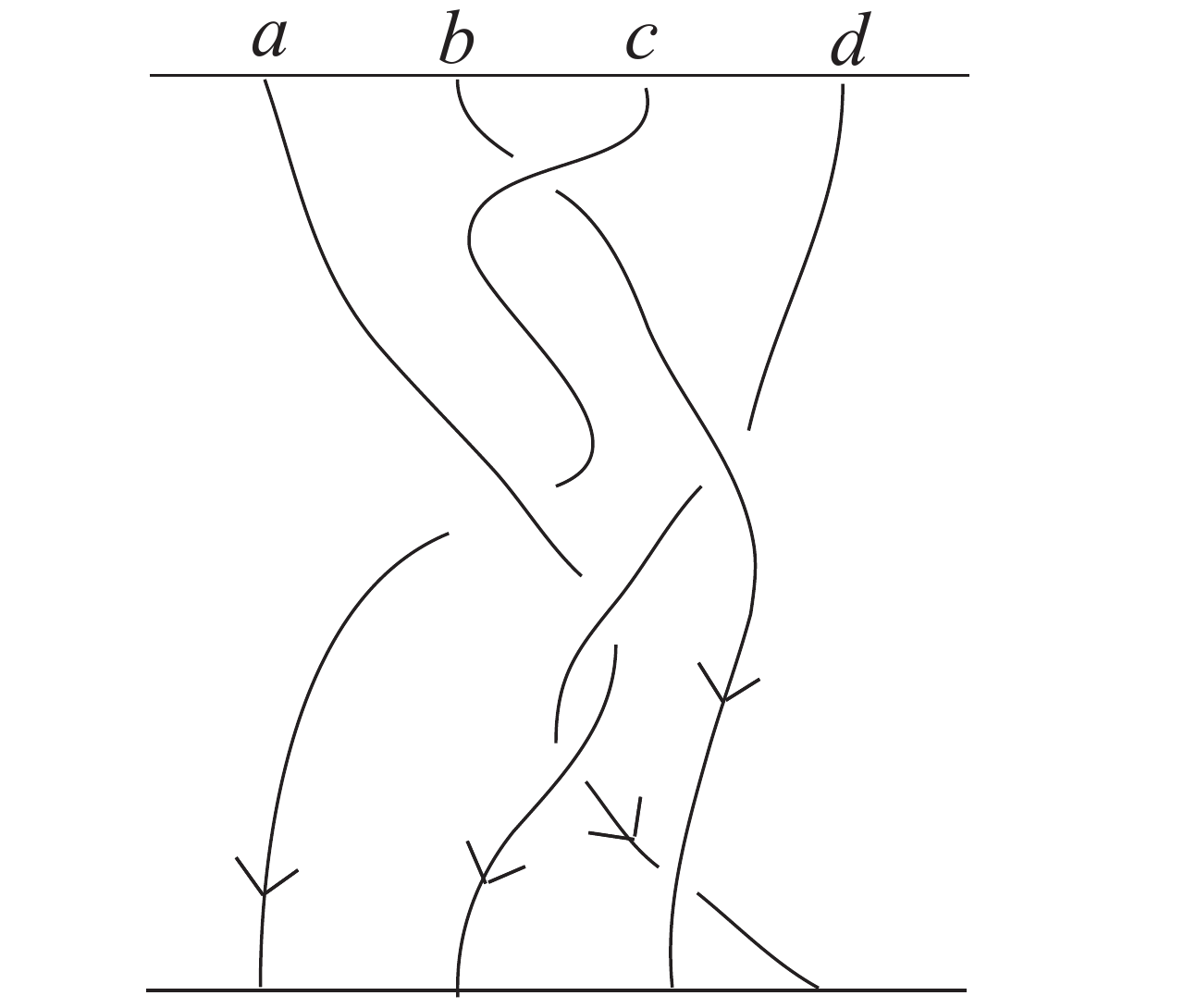}
\caption{}
\label{FramedBraid}
\end{figure}
\end{center}

In terms of diagrams, the formula  (\ref{productFn})  is translated as follows: we place the diagram of the braid $\tau$ under the diagram of the braid $\sigma$ and  the framing $b_i$ travels along the strand up to the top of the diagram of $\sigma\tau$, so that the framings of the product are $(a_i + b_{\sigma(i)})$'s. For instance, for  $\sigma =\sigma_3\sigma_1^{-1}\sigma_2^{-1}\sigma_3^{-1}$ and $\tau = \sigma_1\sigma_3^{-1}$, the product $(t_1^at_2^bt_3^ct_4^d\sigma)(t_1^xt_2^yt_3^zt_4^w\tau)$
in terms of diagrams is showed in Fig. \ref{ProductFramedBraid}.
\begin{center}
\begin{figure}[H]
\includegraphics
[scale=.3]
{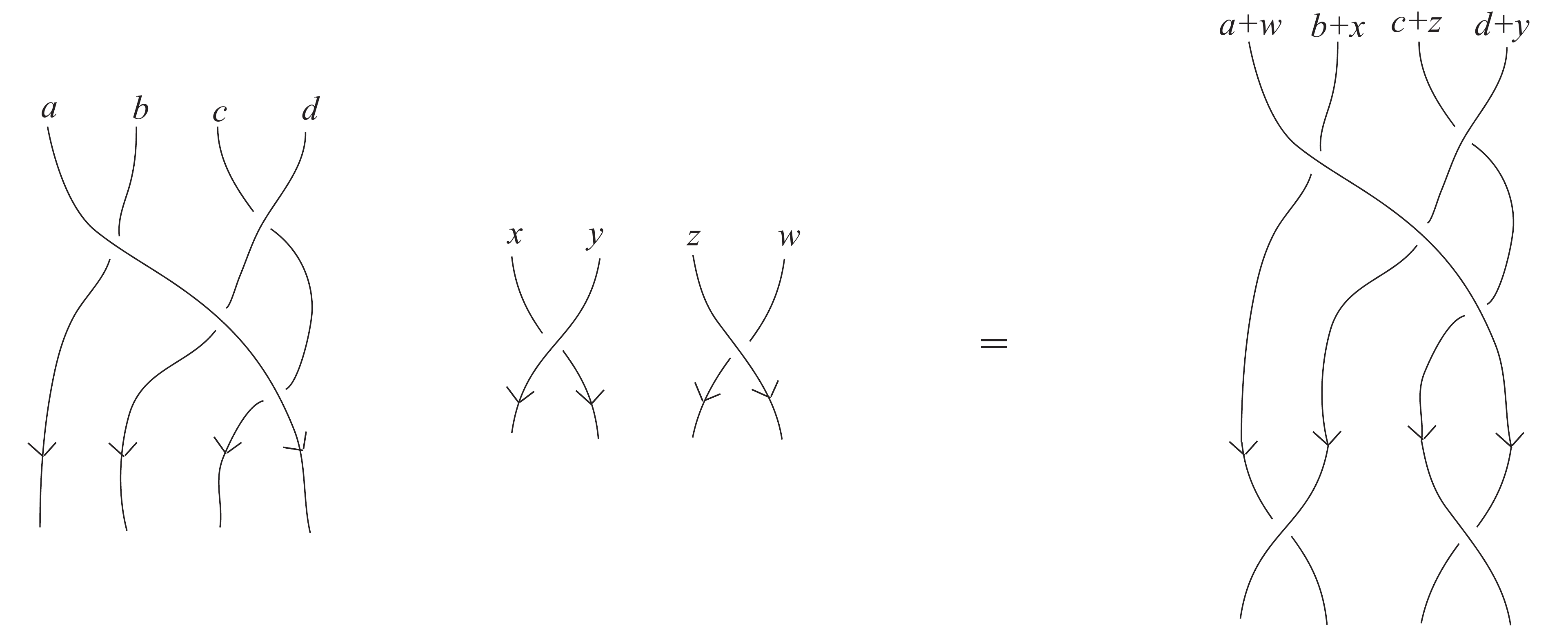}
\caption{}
\label{ProductFramedBraid}
\end{figure}
\end{center}

Finally,  relations (i)--(v) correspond to the defining relations of the framing  module  $d$. The $d$--modular framed braid group  $\mathcal{F}_{d,n}$ is the semidirect product $\ZZ/d\ZZ \rtimes B_n$; in other terms, it  is the group obtained by imposing  the relation $t_i^d = 1$ to the above presentation of $\mathcal{F}_{n}$.

\subsection{}
As  for the Hecke algebra, we can construct an  inductive basis for the Y--H algebra. To do that,  we define in
$\Yn$, the following sets:
$$
R_0=\{1, t_1^a\,;\, 1\leq  a \leq d-1\}, \quad R_1=\{1, t_2^a, g_1r \,;\, x\in R_0, 1\leq  a \leq d-1\}
$$
and
$$
R_{k} = \{1, t_{k+1}^a, g_kr\,;\, 1\leq  a \leq d-1, r\in R_{k-1}\}, \quad \text{for}\quad 2\leq k\leq n-1.
$$
\begin{definition}
Every element in $\Yn$ of the form $r_1\cdots r_{n-1}$, with $r_i\in R_i$, $0\leq i\leq n-1$ is called normal word. We denote by $\mathcal{R}_n$ the set of normal words.
\end{definition}

\begin{theorem}
The set $\mathcal{R}_n$  is  linear basis for $\Yn$; hence $\Yn$ has dimension $d^nn!$.
\end{theorem}
Observe that every element of $\mathcal{R}_n$ has one of the following forms:
\begin{equation}\label{Rn}
\mathfrak{r}t_n^a\quad  \text{or }\quad \mathfrak{r}g_{n-1}\ldots g_1 t_i^a,
\end{equation}
where $\mathfrak{r}\in \mathcal{R}_{n-1}$ and $0\leq a\leq d-1 $.
\begin{example} For $n=3$, we have:
$$
R_0=\{1, t_1^a\,;\, 1\leq  a \leq d-1\}, \quad R_1=\{1, t_2^a, g_1r \,;\, r\in R_0, 1\leq  a \leq d-1\}
$$
and
$$
R_2=\{1, t_3^a, g_2r \,;\, r\in R_1, 1\leq  a \leq d-1\}
$$
Thus,  $\mathcal{R}_1$ of $\mathrm{Y}_{d,1}$ is $R_0$. The basis $\mathcal{R}_2$ of $\mathrm{Y}_{d,2}$ is formed by  the elements in the form:
$$
t_1^at_2^b, \quad t_1^ag_1t_1^b.
$$
The elements of the basis $\mathcal{R}_3$ of $\mathrm{Y}_{d,3}$, are of the form:
$$
t_1^at_2^b\underline{t_3^c},\quad t_1^at_2^b\underline{g_2t_2^c}, \quad t_1^at_2^b\underline{g_2g_1t_1^c}, \quad t_1^ag_1t_1^b\underline{t_3^c}\quad t_1^ag_1t_1^b\underline{g_2t_2^c},\quad
t_1^ag_1t_1^b\underline{g_2g_1t_1^c}.
$$
We used the underline to indicate the form  of  (\ref{Rn}) for the elements of $\mathcal{R}_3$.
\end{example}

Now, in particular,  (\ref{Rn}) said that $\mathcal{R}_{n-1}\subset\mathcal{R}_n$, for all $n$. Then, for every $d$ we have the
following tower of algebras:
\begin{equation}\label{towerYn}
\mathrm{Y}_{d,1}\subset \mathrm{Y}_{d,2}\cdots \subset\cdots  \mathrm{Y}_{d,n}\subset  \mathrm{Y}_{d,n+1}\subset\cdots
\end{equation}
\begin{notation}
We denote by $ \mathrm{Y}_{d, \infty}$ the algebra associated to the tower of algebras  above.
\end{notation}
\subsection{}
Set $\x_0:=1$ and let $ \x_1,\ldots \x_{d-1}$ be $d-1$ independent parameters commuting among them and with the parameters $\z$.

\begin{theorem}
The algebra $ \mathrm{Y}_{d, \infty}$ supports a unique Markov trace, that is,  a family $\tr_d =\{\tr_{d,n}\}_{n\in \NN}$ of linear maps 
$$\tr_{d,n} :\Yn \longrightarrow\CC[\z, \x_1,\ldots \x_{d-1}]$$ defined uniquely by the following rules:
\begin{enumerate}
\item $\tr_{d,n}(1)= 1$,
\item $\tr_{d,n} (ab) = \tr_{d,n} (ba)$,
\item $\tr_{d,n+1} (ag_n) = \z\tr_{d,n} (a)$,
\item $\tr_{d,n+1} (at_{n+1}^k) = \x_k\tr_{d,n} (a)$,
\end{enumerate}
where $a,b\in \Yn$, $0\leq k\leq d-1$.
\end{theorem}
Notice that for $d=1$  the trace $\tr$ becomes the Ocneanu trace.

From now on  we fix a positive integer $d$, thus we shall write  $\tr_{n}$ instead of $\tr_{d,n}$. Moreover, whenever it is not necessary to explicit $n$, we write simply $\tr$ instead of $\tr_{n}$.
\smallbreak
We compute some values of  $\tr$ that will be used later. First, we compute $\tr (	\alpha e_ng_n)$, for every $\alpha\in \Yn$.
\smallbreak
Notice that by using the rule of  commutation $ t_n^st_{n+1}^{d-s} g_n = 	t_n^s g_n t_n^{d-s}$. Then, for every $\alpha\in \Yn$, we have:
$$
\tr (\alpha t_n^st_{n+1}^{d-s} g_n)  = \tr(	\alpha t_n^s g_n t_n^{d-s} ) =\z	\tr(\alpha t_n^s  t_n^{d-s})  = \tr (\alpha).
$$
Thus $\tr(\alpha e_ng_n) = d^{-1}\sum_s\tr(\alpha t_n^st_{n+1}^{d-s}g_n) = \z d^{-1}\sum_s \tr(\alpha) $. Then,
\begin{equation}\label{treng}
\tr (\alpha e_n g_n) = \z \tr(\alpha ) =\tr (\alpha)\tr(g_n).
\end{equation}

Now, we compute   $\tr (e_i)$.

$$
\tr (e_i)  =  \frac{1}{d}\sum_{s}^{}\tr(t_i^st_{i+1}^{d-s})
   =  \frac{1}{d}\sum_{s}^{}\tr(t_i^s)\tr(t_{i+1}^{d-s}) =\frac{1}{d}\sum_{s}^{}\x_s\x_{d-s}.
$$
Hence for all $i$ the trace $\tr$ takes the same values on $e_i$; we denote these values by $E$. More generally,
 we define the elements $e_i^{(m)}$ as follows,
\begin{equation}\label{em}
e_i^{(m)} := \frac{1}{d}\sum_{s=0}^{d-1}t_i^{m+s}t_{i+1}^{d-s},\quad \text{where $0\leq m\leq d-1$},
\end{equation}
where the subindices are regarded as module $d$.
Then, denoting $\tr (e_i^{(m)})$ by $E^{(m)}$, we have
\begin{equation}\label{em}
E^{(m)}:= \tr (e_i^{(m)}) =\frac{1}{d}\sum_{s=0}^{d-1}\x_{m+s}\x_{d-s}\quad \text{where $0\leq m\leq d-1$}.
\end{equation}
Notice that $E^{(0)}= E :=\tr(e_i)$.
\smallbreak
The 	following lemmas will be useful in the next section and their proofs   are a good example to see how  $\tr$ works.	
\begin{lemma}\label{trasis1}
Set $\alpha = w t_n^k$, with $w\in \mathrm{Y}_{n-1}$. We have
$$
\tr( \alpha e_n^{(m)} )= \frac{E^{(m+k)}}{\x_k}\tr(\alpha).
$$
\end{lemma}
\begin{proof}
A direct computation shows:
$$
\tr (w_{n-1}t_n^kt_n^{m+s}t_{n+1}^{d-s} )= \x_{d-s}\tr(w_{n-1}t_n^{m+k+s})=\x_{d-s}\x_{m+k+s}\tr(w_{n-1}).
$$
Then
$$
\tr (\alpha e_n^{(m)} ) =  \frac{1}{d}\sum_s\tr (w_{n-1}t_n^kt_n^{m+s}t_{n+1}^{d-s}) = \frac{1}{d}\sum_s \x_{d-s}\x_{m+k+s}\tr(w_{n-1}),
$$
so, $\tr (\alpha e_n^{(m)} ) = \tr(w_{n-1})E^{(m+k)}$.
Now, $\tr(\alpha)=\x_k \tr (w_{n-1})$. Therefore, the proof follows.
\end{proof}
\begin{lemma}\label{trbasis2}
Set $\alpha = wg_{n-1}g_{n-2}\cdots g_i t_i^k$, with $w\in \mathrm{Y}_{n-1}$. We have
$$
\tr (\alpha e_n )=\z\tr(\alpha^{\prime}e_{n-1} ),
$$
where $\alpha^{\prime} = g_{n-2}\cdots g_it_i^kw$.
\end{lemma}
\begin{proof}
We have to compute firstly the trace of $A_s$, 
$$
A_s:=wg_{n-1}g_{n-2}\cdots g_i t_i^kt_n^st_{n+1}^{d-s}.
$$
We have
$$
\tr (A_s)  =  \x_{d-s}\tr(wg_{n-1}g_{n-2}\cdots g_i t_i^kt_n^s)  = \x_{d-s}\tr(wt_{n-1}^sg_{n-1}g_{n-2}\cdots g_i t_i^k),
$$
where the second equality is obtained by moving $t_{n-1}^{s}$ to the left. Using  first the trace rule (4) and later the trace rule (2),  we get
$$
\tr (A_s) =\z \x_{d-s}\tr(wt_{n-1}^sg_{n-2}\cdots g_i t_i^k) = \z \x_{d-s}\tr((g_{n-2}\cdots g_i t_i^k)(wt_{n-1}^s)),
$$
By using again the trace rule (4), we obtain
$$
\tr (A_s)= \z\tr ( \alpha^{\prime}t_{n-1}^st_n^{d-s}).
$$
Then,
$$\tr(\alpha e_n) = \frac{1}{d}\sum_s\tr(A_s) =\frac{1}{d}\sum_s \z\tr ( \alpha^{\prime}t_{n-1}^st_n^{d-s})= \z\tr(\alpha^{\prime}e_{n-1}).
$$
\end{proof}
\section{The $E$--system}
The $E$--system is the following non--linear system equation of $(d-1)$ equation in the variable $\x_1, \ldots ,\x_{d-1}$.

$$
\begin{array}{ccc}
E^{(1)} & = & \x_1 E\\
E^{(2)} & = & \x_2 E\\
 & \vdots & \\
E^{(d-1)} & = & \x_{d-1} E\\
\end{array}
$$
\begin{example}
For  $d=3$, the $E$--system is:
$$
\begin{array}{ccc}
\x_1 +  \x_2^2 = 2\x_1^2\x_2\\
\x_1^2 + \x_2 = 2\x_1\x_2^2.
\end{array}
$$

For  $d=4$, the $E$--system is:
$$
 \begin{array}{ccc}
 \x_1 + 2\x_2\x_3 & = &  2\x_1^2 \x_3 + \x_1\x_2^2  \\
 \x_1^2 + \x_2 +\x_3^2 & = & 2\x_1\x_2\x_3 + \x_2^3 \\
\x_3 + 2\x_1 \x_2 & = &  2\x_1\x_3^2 + \x_2^2\x_3.
 \end{array}
$$
\end{example}
The $E$--system plays a key role to define the invariants in the next section. The $E$--system was solved by P. Gerardin, see
 \cite[Appendix]{julaAM}.
 \begin{theorem}[P. Gerardin, 2013]
 The solutions of the $E$--system are parametrized by the non--empty subset of the group $\ZZ/d\ZZ$. Moreover, given a  such subset  $S$, the solutions are:
 $$
\x_{k} = \frac{1}{d}\sum_{s\in S}\exp\left( \frac{2\pi ks}{d}\right),\quad 0\leq k\leq d-1.
 $$
\end{theorem}
\begin{remark}\rm Let $S$ be a non--empty subset of $\ZZ/d\ZZ$. We have:
\begin{enumerate}
\item For $\vert S\vert=1$, the solutions of the $E$--system are the $d$--roots of the unity.
\item If $S$ consists of the coprimes with $d$, the solution of the $E$--system  are Ramanujan sums.
\item By taking the parameters trace $\x_1, \ldots , \x_{d-1}$ as a solution of the system, then
\begin{equation}\label{tr(ei)}
E=\tr (e_i) = \frac{1}{\vert S\vert}.
\end{equation}
\end{enumerate}
\end{remark}
\begin{proposition}\label{tralphaen}
If the parameters trace $\x_k$'s are taken as solution of the $E$--system, then for every $\alpha \in \Yn$, we have:
$$
\tr(\alpha e_n) = \tr(\alpha)\tr(e_n).
$$
\end{proposition}
\begin{proof}
From the linearity of $\tr$, it is enough to consider $\alpha$ in the basis $\mathcal{R}_n$. The proof will be done by induction on $n$. For $n=1$, the claim is clear since $\alpha\in \mathrm{Y}_1 =\KK$; suppose now the proposition be true for every $n-1$ and let $\alpha\in \mathcal{R}_{n-1}$. Because (\ref{Rn}), we have two cases according to the form of $\alpha$: (i)
$\alpha=\mathfrak{r}t_n^a\quad$ or  (ii) $\alpha = \mathfrak{r}g_{n-1}\ldots g_1 t_n^a$,
where $\mathfrak{r}\in \mathcal{R}_{n-1}$ and $0\leq a\leq d-1 $.

For the case (i),  from Lemma \ref{trasis1} we have: $\tr (\alpha e_n) = \left(E^k/\x_k \right)\tr (\alpha)  = E\tr (\alpha)$ since  the $\x_k$'s are solution of the $E$--system.

For the case (ii), we use  the Lemma \ref{trbasis2}, so:
$\tr (\alpha e_n) = \z\tr (\alpha^{\prime}e_{n-1} )$, where $ \alpha^{\prime} = g_{n-2}\cdots g_it_i^k\mathfrak{r}$. From the inductive hyphotesis we have $\tr (\alpha' e_{n-1}) =  \z\tr (\alpha^{\prime})\tr (e_{n-1} )$; but $ \z\tr (\alpha^{\prime})= \tr (\alpha) $. Then,
$ \tr (\alpha e_n) = \tr (\alpha)\tr (e_{n-1} )= \tr (\alpha)\tr (e_{n} )$.

\end{proof}
\section{The invariants  $\Delta_m$ and $\Theta_m$ }
In this section we define the  invariants $\Delta_m$ and $\Theta_m$ which are constructed using the method  due to V. Jones to  construct the Homflypt polynomial, see Section \ref{sectionHomflypt}. Essentially, in the Jones method   we use now the Y--H algebra instead of the Hecke algebra and instead of the Ocneanu trace we use the trace $\tr$.
Observe that the construction below follows  what has been done in Section \ref{sectionHomflypt}.
\subsection{}
For $\theta$, we denote by $\pi_{\theta}$ the homomorphism from $B_n$ to $\Yn$, such that  $\sigma_i\mapsto\theta g_i$.
Since  $\pi_{\theta}$ is an homomorphism and the trace rule (2) of $\tr$, it follows that $\tr \circ \pi_{\theta}$  agrees with the Markov replacement M1; thus, it remains to see if $\tr \circ \pi_{\theta}$  agrees with the Markov replacements M2, that is, we need
\begin{equation}\label{theta1}
(\tr \circ \pi_{\theta} )(\sigma\sigma_n^{-1}) = (\tr \circ \pi_{\theta} )(\sigma\sigma_n)\quad \text{for all $\sigma\in B_n$.}
\end{equation}
Put $\alpha :=\pi_{\theta}(\sigma)$, then we have $(\tr \circ \pi_{\theta} )(\sigma\sigma_n) =\theta\tr (\alpha g_n) =\z\theta\tr (\alpha)$ and
 $(\tr \circ \pi_{\theta} )(\alpha\sigma_n^{-1}) =\theta^{-1} \tr (\alpha g_n^{-1})$.
 So, the equation (\ref{theta1}) is equivalent to:
\begin{equation}\label{theta2}
\theta^2 =\frac{\tr (\alpha g_n^{-1})}{\z\tr(\alpha)}.
\end{equation}
 Now, applying the formula  (\ref{inversgi}) to $g_n^{-1}$  and the linearity of $\tr$, we get:
\begin{equation}\label{theta3}
\tr (\alpha g_n^{-1}) = \tr(\alpha g_n) + (\u^{-1}-1)\tr (\alpha e_n)  +(\u^{-1} -1)\tr (\alpha e_ng_n).
\end{equation}
 Therefore, in order to get the condition on $\theta$ in (\ref{theta2}), we need to factorize by $\tr (\alpha)$ in the second member of the equality of (\ref{theta3}). Now, observe that $ \tr(\alpha g_n) = \z\tr (\alpha)$ and from (\ref{treng}) we get $\tr (\alpha e_ng_n) =\z \tr(\alpha)$. Unfortunately,  we cannot take out the factor $\tr(\alpha)$ in $ \tr (\alpha e_n)$. However, resorting to Proposition \ref{tralphaen}, we can do it whenever  the $\x_i$'s are the solution of the E--system. Thus, if  $\x_S:=(\x_1,\ldots ,\x_{d-1})$ is the solution of the $E$--system determined by the set $ S$, the equation  (\ref{theta3}) can be written as
 $$
 \tr (\alpha g_n^{-1}) = \tr(\alpha)\tr(g_n) + (\u^{-1}-1)\tr (\alpha)\tr(e_n)  +(\u^{-1} -1)\tr (\alpha) \tr(e_ng_n).
 $$
 Then, (\ref{theta2}) yields
 \begin{equation*}
 \theta^2 =\frac{\z + (\u^{-1}-1)E +  (\u^{-1}-1)\z }{\z} =\frac{ (1-\u)E +  \z }{\u\z}.
 \end{equation*}
 So, by (\ref{tr(ei)}) we get:
 \begin{equation*}
  \theta^2  = \frac{ (1 -\u) + \z \vert S\vert}{\u\z \vert S\vert}.
 \end{equation*}
Because  $\theta^2$ is depending only on the  cardinal of $S$,  we can define   $\lambda_m$ for every $m\in   \NN$,  as  the $\theta^2$ above, i.e.
\begin{equation} \label{lambdam}
\lambda_m :=\frac{  (1 -\u) + \z m}{\u\z m}.
\end{equation}
Recapitulating, by extending the field $\KK$ to $\KK(\sqrt{\lambda_m})$, we can  consider the following homomorphism
$\pi_{\sqrt{\lambda_m}}$,
$$
\pi_{\sqrt{\lambda_m}}: B_n \longrightarrow \Yn,	\quad\text{through} \quad \sigma_i\mapsto\sqrt{\lambda_m}\,g_i.
$$
Thus, we have a family of functions $\{\tr_n \circ \pi_{\sqrt{\lambda_m}} \}_{n}$ agreeing with the replacements M1 and M2.
Proceeding   as in  (\ref{homflypFactor}) we get  that, now, the factor of normalization for $\tr_n \circ \pi_{\sqrt{\lambda_m}}$ is  $\left(\z\sqrt{\lambda_m}\right)^{n-1}$.
\begin{definition}
Let $m\in \NN$. For $\sigma\in B_n$, we define
$$
\Delta_m (\sigma) := \left( \frac{1}{\z\sqrt{\lambda_m}}\right)^{n-1}(\tr_n \circ \pi_{\sqrt{\lambda_m}} )(\sigma).
$$
\end{definition}
\begin{theorem}
Let $L$ be  a link s.t. $L=\widehat{\sigma}$, where $\sigma\in B_n$, then the map
$\Delta_m$ is an invariant of ambient isotopy for oriented links,
$$
\Delta_m : \mathfrak{L}\longrightarrow \CC(\z, \sqrt{\lambda_m}),  \quad L\mapsto \Delta_m (\sigma).
$$
\end{theorem}
\begin{remark}\label{Delta1}\rm
Regarding (\ref{RelationLambdaZeta}) and (\ref{lambdam}) it results clear that $\Delta_1$ is the Homflypt polynomial.
\end{remark}

For the  benefit of  the writing, the main results on the invariant $\Delta_m$ will be established after we define a cousin of $\Delta_m$,  denoted denoted by $\Theta_m$. These invariants are not equivalents but share several  properties.

\subsection{}\label{Yn(v)}
The Yokonuma--Hecke algebra has another presentation due to M. Chlouveraki and  L. Poulain d'Andecy \cite{chpoAM}, cf. \cite{chjukalaIMRN, esryJPAA}. This presentation is constructed as follows.  Firstly,  the field ${\Bbb K}$ is extended to ${\Bbb K}(\v)$ with $\v^2=\u$; secondly, new generators  $f_i$  are defined  by
\begin{equation}\label{deffi}
f_i := g_i + (\v^{-1}-1)e_ig_i, \quad (1\leq i\leq n-1 ).
\end{equation}
It is a routine to check that  the $f_i$'s and the $e_i$'s  satisfy the relations (\ref{yh1})--(\ref{yh5}) if one substitutes $g_i$  with  $f_i$ and the relation,
\begin{equation}\label{quadfi}
f_i^2 = 1 + (\v - \v^{-1})e_if_i.
\end{equation}
Notice that $f_i$'s are invertibles, and
\begin{equation}\label{inversefi}
f_i^{-1} = f_i - (\v - \v^{-1})e_i.
\end{equation}
Thus one obtains a presentation of the Yokonuma--Hecke algebra by the generators $f_i$'s and $e_i$'s and the same relations of those  given for the defining generators $g_i$'s  and $e_i$'s except for the relation (\ref{yh6}) which is replaced by the relation (\ref{quadfi}). We shall use the notation $\mathrm{Y}_{d,n}(\v)$, or simply $\mathrm{Y}_{n}(\v)$, whenever the Yokonuma--Hecke algebra is considered with the presentation by the $f_i$'s and $e_i$'s.
\begin{remark}\rm
Notice that $\mathrm{Y}_{1,n}(\v)$ coincides  with the presentation of the Hecke through  the generators $\tilde{h}_i$'s  of Exercise X.
\end{remark}

By using now, in the Jones recipe,  $\mathrm{Y}_{d,n}(\v)$ instead of $\Yn(\u)$,  one obtains  an invariant  $\Theta_m$ instead $\Delta_m$. To be precise, notice firstly  that for every $\alpha\in \Yn$, according to  properties of $\tr$, (\ref{deffi}) and  (\ref{treng}), we have:
$$
\tr (\alpha f_n) = \tr (g_n) \tr (\alpha) + (\v^{-1} -1)\tr(g_n)\tr(e_n\alpha) =\z\v^{-1}\tr(\alpha).
$$
Secondly, given a non--empty subset  $S$  of $\ZZ/d\ZZ$ and $(\x_1, \ldots ,\x_{d-1})$  the solutions determined by $S$,  for every $\alpha\in\Yn$ the  Proposition \ref{tralphaen} together with  (\ref{deffi}) and (\ref{inversefi},) imply:
\begin{eqnarray*}
\tr (\alpha f_n^{-1}) 
&= & 
\v^{-1}\tr(g_n)\tr (\alpha) -  (\v - \v^{-1})\tr(\alpha)\tr(e_n) \\
& = & 
\tr (\alpha) (\v^{-1}\z -(\v -\v^{-1})E).
\end{eqnarray*}
Therefore, the rescaling factor  is:
$$
\frac{\tr (\alpha f_n^{-1})}{\tr (\alpha f_n) } =\frac{\z\v^{-1}-(\v - \v^{-1})E}{\z\v^{-1} } = \frac{\z-(\v^2-1)E}{\z}
= \frac{\vert S\vert \z - (\v^2-1)}{\vert S\vert\z};  
$$
the last equality is due to (\ref{tr(ei)}). Again the rescaling factor depends only on the cardinal $\vert S\vert$, so, by using the  procedure used   to get the  definition of $\Delta_m$, we get firstly   the  rescaling factor, denoted by $\lambda'_m$, namely
\begin{equation}\label{lambdamprime}
\lambda'_m = \frac{\z m - (\v^2-1)}{\z m} \quad (m\in \NN);
\end{equation}
and secondly $\Theta_m$ as follows.
\begin{definition}\label{Theta}
For $\sigma\in B_n$, we define
\begin{equation*}
\Theta_m (\sigma) := \left( \frac{\v}{\z\sqrt{\lambda'_m}}\right)^{n-1}(\tr_n \circ \pi_{\sqrt{\lambda'_m}} )(\sigma), \quad (m\in \NN).
\end{equation*}
\end{definition}
Keeping the notations above, we have the following theorem.
\begin{theorem}
The function $\Theta_m : \mathfrak{L}\longrightarrow \KK(\z, \sqrt{\lambda'_S})$, defined by $L\mapsto \Theta_m(\sigma)$, where $L $ is the closure of the braid $\sigma$,
is an invariant of ambient isotopy        of oriented    links.
\end{theorem}
\begin{remark}\rm
Exactly as $\Delta_1$ in Remark \ref{Delta1}, we have that $\Theta_1$ coincide with the  Homflypt polynomial.
\end{remark}
\begin{theorem}
The invariants $\Delta_m$ and $\Theta_m$ coincide with the Homflypt polynomial whenever they are evaluated on knots.
\end{theorem}
To prove that the invariants $\Delta_m$ and $\Theta_m$ were not equivalents to the Hompflypt polynomial was a not trivial matter. Firstly, in \cite{chjukalaIMRN} were found six pairs of links with equivalents Homflypt polynomial but different $\Theta_m$, see \cite[Table1]{chjukalaIMRN};  a such pair is shown in Fig. \ref{FirstTheta}.
\begin{center}
\begin{figure}[H]
\includegraphics[scale=.36]{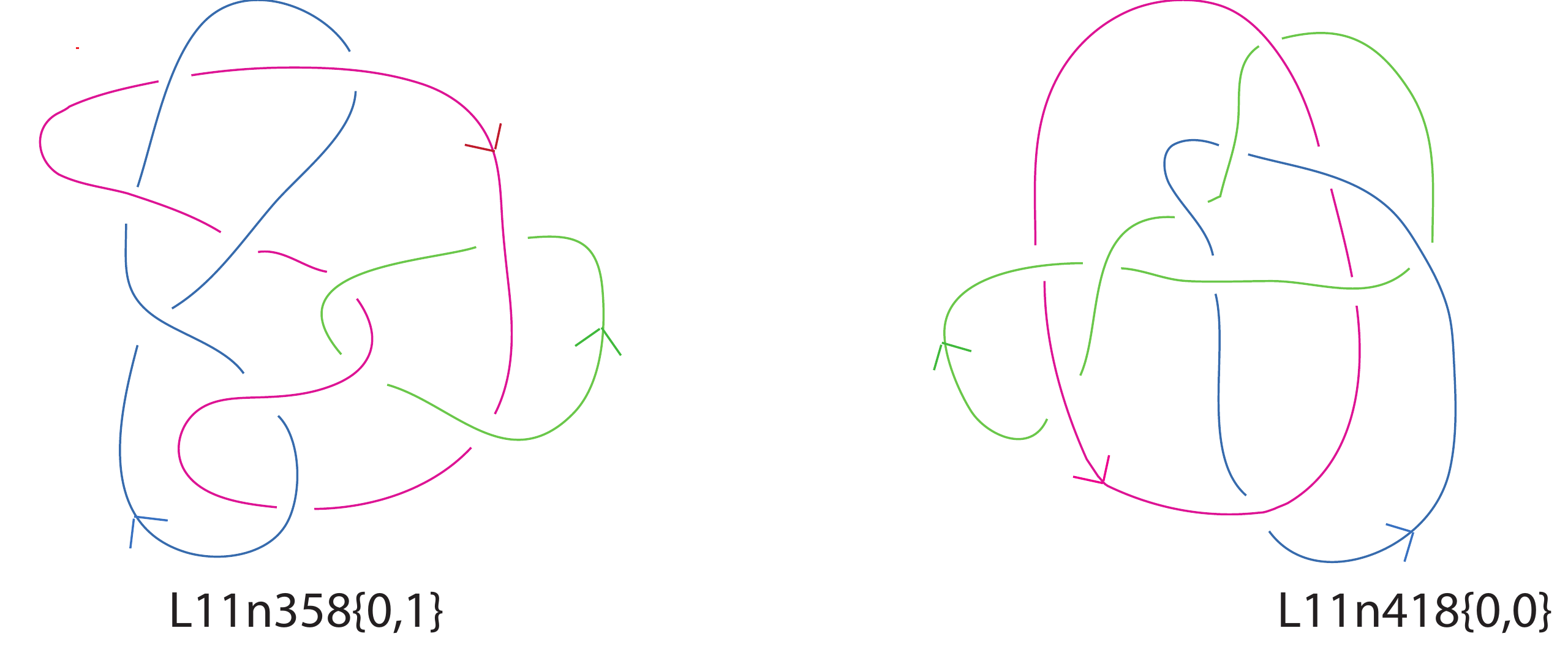}
\caption{}
\label{FirstTheta}
\end{figure}
\end{center}
Hence.
\begin{theorem}
For every $m\geq 2$, the invariants $\Theta_m$ are not equivalent to the Homflypt polynomial.
\end{theorem}
The people working on theses invariants thought  that the invariants  $\Theta_m$ and $\Delta_m$  were equivalents\footnote{This is due to the fact that in the definitions of $\Theta_m$ and $\Delta_m$ the unique difference  is the change of presentations for the Y--H algebra,  apparently unimportant thing.} but surprisingly F. Aicardi shows that it is not the case.
\begin{theorem}[Aicardi]
For every $m\geq 2$, the invariants $\Delta_m$ and $\Theta_m$ are not topologically equivalent to the Homflypt polynomial: indeed,
there exists pairs of non isotopic links distinguished by $\Delta_m$  and/or $\Theta_m$ but not by the Homflypt polynomial.
\end{theorem}
\begin{proof} In \cite{aica} we can find several pairs proving this theorem. In Fig.  \ref{DistingueDeltaNoTheta}, we have a pair of non isotopic links distinguished by $\Delta_m$, $m\geq 2$, but neither by Homflypt nor $\Theta_m$.
\begin{center}
\begin{figure}[H]
\includegraphics
[scale=.36]
{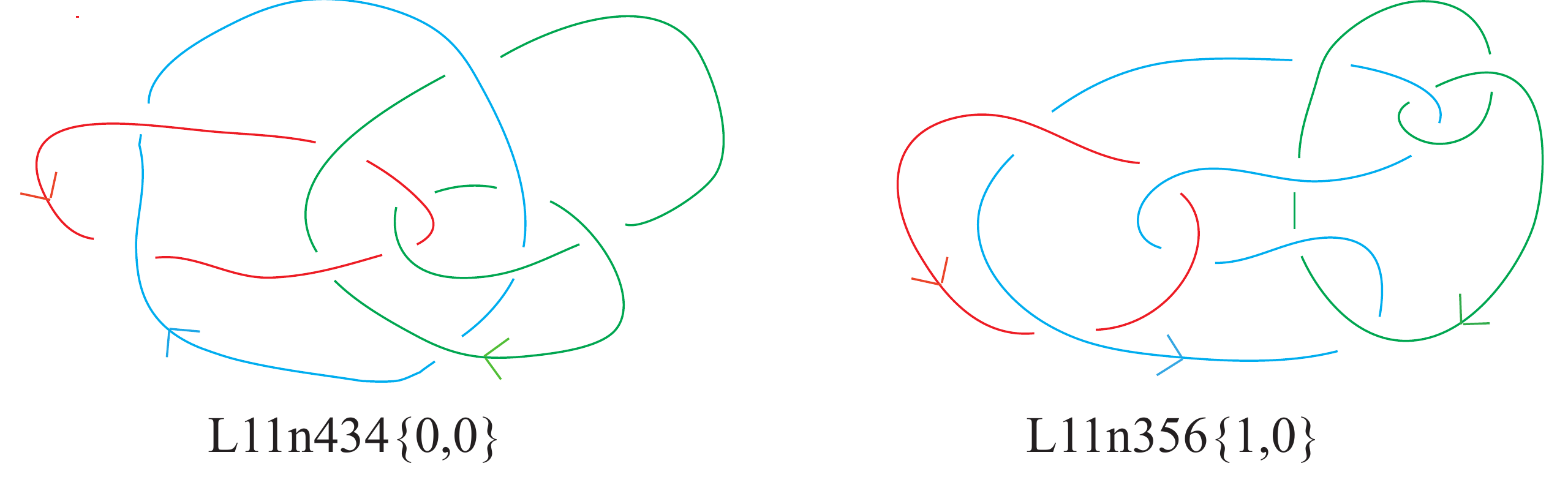}
\caption{}
\label{DistingueDeltaNoTheta}
\end{figure}
\end{center}

\end{proof}
The next theorems establish the main properties of the invariants $\Delta_m$ and $\Theta_m$.

\begin{proposition}
The invariants $\Delta_m$ and $\Theta_m$ share several properties with  the Homflypt polynomial, e.g.: the  behavior under connected sums and mirror image.
\end{proposition}

In the next section we generalize, respectively,  the invariants
 $\Delta_m$ and $\Theta_m$ to, respectively, certain    invariants in three parameters, $\overline{\Delta}$ and $\overline{\Theta}$ for classical links.  We will do this,  by using  the Jones recipe  applied to  the bt--algebra. Also we  will see, that $\overline{\Theta}$ can be defined through skein relations, while  $\overline{\Delta}$ cannot.

\section{The bt--algebra}
In this section we introduce   the so--called bt--algebra (or algebra of braids and ties) which is constructed by abstracting the braid generators $g_i$'s of the Yokonuma--Hecke algebra and the idempotents $e_i$'s appearing in the square
of the braid generators, see (\ref{yh6}). This algebra is used to generalize the invariants $\Delta_m$ and $\Theta_m$ to invariants with three parameters as well as  its understanding  by skein  relations. Before  introducing the bt--algebra we shall recall the main facts on set partitions since these facts will be useful in the rest of these notes.

\subsection{}
For $n\in {\Bbb N}$, we denote by ${\bf n}$ the set  $\{1, \ldots , n\}$ and by $\Pn$ the set formed by the set partitions of ${\bf n}$, that is, an  element of $\Pn$ is a collection    $I  = \{I_1, \ldots , I_k\}$ of pairwise--disjoint non--empty sets whose union is ${\bf n}$; the sets $I_1, \ldots , I_k$ are called the blocks of $I$; the cardinal of $\Pn$, denoted $b_n$, is called the   $n^{th}$ Bell number.

We can regard $\Pn$ as subset of $\mathsf{P}_{n+1}$ through the natural injective map $\iota_n: \Pn\longrightarrow  \mathsf{P}_{n+1}$,  where for $I\in \Pn$,
 the image $\iota_n(I)\in  \mathsf{P}_{n+1}$    is defined   by adding   to  $I$ the   block $\{n+1\}$.

Typically,  the set partitions are represented by scheme of arcs, see \cite[ Subsection 3.2.4.3]{MaBook}, that is:  the   point $i$ is   connected  by an  arc to  the  point $j$,    if  $j$ is   the minimum in the  same  block of $i$  satisfying $j>i$.   Figure \ref{SetPartition}  shows the  set partition $I=\{\{1,3\},\{2,5,6\},\{4\}\}$ as  represented by arcs.
\begin{center}
\begin{figure}[H] 
\begin{picture}(120,25)
\put(0,2){$\bullet$}
\put(23,2){$\bullet$}
\put(46,2){$\bullet$}
\put(69, 2){$\bullet$}
\put(92, 2){$\bullet$}
\put(115, 2){$\bullet$}
\put(0, -5){\tiny{$1$}}
\put(24, -5){\tiny{$2$}}
\put(47, -5){\tiny{$3$}}
\put(71, -5){\tiny{$4$}}
\put(93, -5){\tiny{$5$}}
\put(116, -5){\tiny{$6$}}
\qbezier(2,5)(25,30)(50,5)
\qbezier(95,5)(107,25)(119,5)
\qbezier(25,5)(60,30)(95,5)
\end{picture}
\caption{}\label{SetPartition}
\end{figure}
\end{center}

The representation by arcs of a set partition   induces  a
natural indexation of its blocks. More precisely, we say  that  the blocks $I_j$'s of the set partition $I= \{I_1, \ldots , I_m\}$ of ${\bf n}$ are  {\it standard indexed} if $\min(I_j)< \min(I_{j+1})$, for  all $j$.   For instance, in the set partition of Figure \ref{SetPartition} the blocks are indexed as:
$I_1= \{1,3\}$, $I_2=\{2,5,6\}$ and $I_3=\{4\}$.

The natural  action of $\mathtt{S}_n$ on ${\bf n}$  induces, in the obvious way,   an action of $\mathtt{S}_n$ on $\Pn$  that is,
for $I=\{I_1, \ldots , I_m\}$ we have
\begin{equation}\label{SnOnPn}
w(I) :=\{w(I_1), \ldots , w(I_m)\}.
\end{equation}
Notice that this action   preserves  the cardinal of  each block   of  the  set partition.

Now, we shall say that two set partitions $I$ and $I'$ in $\Pn$ are conjugate, denoted by $I \sim I'$, if there exists $w\in \mathtt{S}_n$ such that, $I' = w(I) $; if it is necessary to precise   such   $w$, we write $I \sim_w I'$.
Further, observe that if  $I$ and  $I^{\prime}$ are  standard  indexed with  $m$  blocks,  then the  permutation $w$  induces a permutation  of $S_m$ of the indices of the   blocks, which we   denote by $w_{I,I'}.$

\begin{example}\label{exw} Let    $I=\{\{1,2\}_1,\{3\}_2,\{4,5\}_3,\{6\}_4\}$ and  
$$
I^{\prime}=\{ \{1\}_1,\{2,5\}_2,\{3,6\}_3\{4\}_4\},$$
 so  $n=6$ and   $m=4$. We have
$I \sim_w I^{\prime}$, where:
$$
w=(1,6)(2,3,4,5)\quad \text{and}\quad   w_{I,I^{\prime}}=  (1,3,2,4).
$$
\end{example}

Given a  permutation  $w\in \mathtt{S}_n$ and  writing $w=c_1\cdots c_m$ as
product of disjoint cycles, we denote by $I_w$  the set partition  whose blocks are the cycles $c_i$'s, regarded now as
subsets of $\bf n$.    Reciprocally, given a set partition $I=\{I_1, \ldots ,I_m\}$ of $\bf n$ we denote by $w_I$ an element of $\mathtt{S}_n$ whose cycles are the blocks $I_i$'s.   Moreover, we shall say that the cycles of $w_I$ are standard indexed, if they are indexed according to the standard indexation of $I$.

\begin{notation} \rm
When there is no risk of  confusion,   we  will  omit in the  partitions the  blocks  with a  single  element.
 \end{notation}

$\Pn$ is a poset with  structure of commutative monoid. Indeed, the partial order on $\Pn$ is defined as follows:
$I\preceq J$ if and only if each block of $J$ is  a union of blocks of $I$. The  product  $I\ast J$, between $I$ and $J$ is  defined as the minimal  set partition, containing $I$ and $J$,  according to  $\preceq$;  the identity of this monoid is $ {\bold 1}_n := \{\{1\}, \{2\}, \ldots ,\{n\}\} $. Observe that:
\begin{equation}\label{partial}
I \ast J = J, \quad \text{whenever}\quad  I\preceq J,
\end{equation}
\begin{equation}\label{astwithw}
I\ast J = I \ast w_I (J).
\end{equation}
Further, the injective maps $\iota_n$'s  result to be a monoid homomorphisms and respect  $\preceq$,  that is, for every    $I, J\in\Pn$,  we have:
\begin{equation}\label{agreewithposet}
I\preceq J \quad \text{then} \quad \iota_n(I)\preceq \iota_n(J).
\end{equation}
\begin{notation}\label{Pinf}
The inductive limit associated to the    family of monoids $\{(\Pn,\iota_n)\}_{n\in \mathbb{N}}$ is denoted by $\mathsf{P}_{\infty}$,
\end{notation}

We are going to give an  abstract description of $\Pn$, that is, via a presentation which will be used in the next section.
For every $1\leq i < j\leq n$ with $i\not= j$, define $\mu_{i,j}\in \Pn$ as the set partition whose blocks are $\{i,j\}$ and $\{k\}$ where $1\leq k\leq n$ and $k\not=i,j$.    We shall write  $\mu_{i,j}\mu_{k,h}$  instead of $\mu_{i,j} \ast \mu_{k,h}$.
\begin{proposition}\label{presentationPn}
The monoid $\Pn$ can be presented by the set partitions  $\mu_{i,j}$'s subject to   the following relations:
\begin{equation}\label{relationsPn}
\mu_{i,j}^2 = \mu_{i,j} \quad \text{and }\quad \mu_{i,j}\mu_{r,s} = \mu_{r,s}\mu_{i,j}.
\end{equation}
 \end{proposition}

\subsection{}
The original definition of the bt--algebra is the following.

\begin{definition} [See \cite{ajICTP1, rhJAC, aijuMMJ1}] The bt--algebra , denoted by $\En(\u)$, is defined by $\mathcal{E}_1(\u):={\Bbb K}$  and for $n\geq  2$ as   the unital associative ${\Bbb K}$--algebra, with unity $1$, defined by braid generators $T_1,\ldots , T_{n-1}$ and ties generators $E_1,\ldots , E_{n-1}$ subjected to the following relations:
\begin{eqnarray}
E_iE_j & =  &  E_j E_i \qquad \text{ for all $i,j$},
\label{bt1}\\
E_i^2 & = &  E_i\qquad \text{ for all $i$},
\label{bt2}\\
E_i T_j  & = &
 T_j E_i \qquad \text{for $\vert i  -  j\vert >1$},
\label{bt3}\\
E_i T_i & = &   T_i E_i,
\label{bt4}\\
E_iT_jT_i & = &  T_jT_iE_i, \qquad \text{ for $\vert i  -  j\vert =1$}
\label{bt5}\\
E_iE_jT_i & = & E_j T_i E_j  \quad = \quad T_iE_iE_j \qquad \text{ for $\vert i  -  j\vert =1$},
\label{bt6}\\
 T_i T_j & = & T_j T_i\qquad \text{ for $\vert i - j\vert > 1$},
 \label{bt7}\\
 T_i T_j T_i& = & T_j T_iT_j\qquad \text{ for $\vert i - j\vert = 1$},
 \label{bt8}\\
 T_i^2 & = & 1  + (\u-1)E_i + (\u-1)E_i T_i\qquad \text{ for all $i$}.
 \label{bt9}
\end{eqnarray}
\end{definition}
The $T_i$'s are invertible,   with $T^{-1}$ given by:
\begin{equation}\label{Tinverse}
T_i^{-1} = T_i + (\u^{-1} - 1)E_i + (\u^{-1} - 1)E_iT_i.
\end{equation}
\begin{remark}\label{remarkbt}\rm
\begin{enumerate}
\item
The bt--algebra can be seen as  a generalization of the Hecke algebra since by making $E_i=1$, the definition of the bt--algebra becomes the Hecke algebra. Further, observe that the mapping $T_i\mapsto h_i$ and $E_i\mapsto 1$ defines  an epimorphism from the bt--algebra to the Hecke algebra.
\item The mapping $T_i\mapsto g_i$ and $E_i\mapsto e_i$ defines an algebra homomorphism from the bt--algebra to the YH--algebra, which is injective if and only if $d\geq n$, see \cite{esryJPAA}, cf. \cite[Remark 3]{aijuMMJ1}.
\end{enumerate}
\end{remark}
Diagrammatically the generators $T_i$'s can be regarded as usual braids  and the generator  $E_i$ as a tie between the $i$ and $i+1$ strands, this tie doesn't have a topological  meaning: it is an auxiliary artefact to reflect the monomial--homogeneous defining relation of the bt--algebra. Thus, the tie is pictured as a spring or a dashed line  between the strands. More precisely, the diagrams for, respectively, $T_i$ and $E_i$, are:
\begin{center}
\begin{figure}[H]
\includegraphics
[scale=.4]
{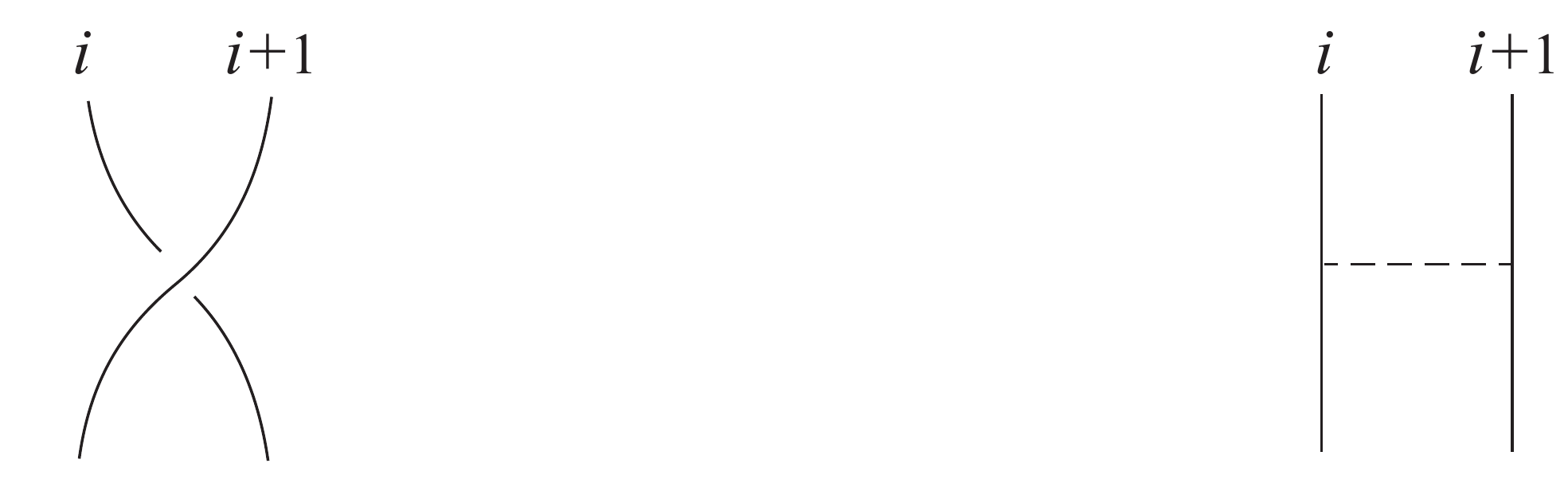}
\caption{}
\label{TiEi}
\end{figure}
\end{center}
Here comment!
\subsection{}
The bt--algebra is a finite dimensional algebra. Moreover, there is a basis due to S. Ryom--Hansen \cite{rhJAC}.
Before expliciting    the Ryom--Hansen basis we need   to introduce the  tools below.

For $i<j$, we define $E_{i,j}$ by
 \begin{equation}\label{Eij}
 E_{i,j} = \left\{\begin{array}{ll}
 E_i & \text{for} \quad j = i +1,\\
 T_i \cdots T_{j-2}E_{j-1}T_{j-2}^{-1}\cdots T_{i}^{-1}& \text{otherwise.}
 \end{array}\right.
 \end{equation}
 For any  nonempty  subset $J$  of $\bf n$ we define $E_J=1$  for  $|J|=1$ and otherwise by
$$
 E_J := \prod_{(i,j )\in J\times J, i<j}  E_{i,j}.
$$
Note that $E_{\{i,j\}} = E_{i,j}$.
For $I = \{I_1, \ldots , I_m\} \in \Pn$,  we define
 $E_I$  by
 \begin{equation}\label{EI}
 E_I = \prod_{k}E_{I_k}.
 \end{equation}
Now,  if   $w=\mathtt{s}_{i_1}\cdots \mathtt{s}_{i_k}$ is a reduced expression of $w\in \mathtt{S}_n$, then the  element $T_w:= T_{i_1}\cdots T_{i_k}$ is well defined.
The action of $\mathtt{S}_n$ on $\Pn$ is inherited from the  $E_{I}$'s and  we have:

\begin{equation}\label{w(E)}
T_w E_I T_w^{-1} =  E_{w(I)}\qquad (\text{see \cite[Corollary 1]{rhJAC}}).
\end{equation}

\begin{theorem}[{\cite[Corollary 3]{rhJAC}}] \label{basEn}
The set $\{E_I T_w  \,; \,w\in \mathtt{S}_n,\, I\in  \Pn \}$ is a ${\Bbb K}$--linear basis of
${\mathcal E}_n(\u)$. Hence the dimension of ${\mathcal E}_n(\u)$ is $b_nn!$.
\end{theorem}
\begin{example}
\end{example}
Having in mind  (\ref{agreewithposet}), the natural inclusion $\mathtt{S}_n\subset \mathtt{S}_{n+1}$, for every $n$, together with Theorem 	\ref{basEn}, we deduce the  tower of algebras:
 $$
    {\mathcal E}_1(\u)\subset {\mathcal E}_{2}(\u)\subset  \cdots\subset {\mathcal E}_n(\u)\subset {\mathcal E}_{n+1}(\u)\subset\cdots
 $$
\begin{notation}
Denote by ${\mathcal E}_{\infty}(\u)$ the inductive limit associated the bt--algebras  above.
\end{notation}

\begin{remark}[Cf. Subsection \ref{Yn(v)}]\label{Vi}\rm
Extending the field ${\Bbb K}$ to ${\Bbb K}(\v)$ with $\v^2=\u$,
 we can define  (cf. \cite[Subsection 2.3]{maIMRN}):
 \begin{equation}\label{defVi} V_i := T_i + (\v^{-1}-1)E_iT_i.\end{equation} Then the $V_i$'s and the $E_i$'s  satisfy the relations (\ref{bt3})--(\ref{bt8}) and the quadratic relation (\ref{bt9}) is transformed in
\begin{equation}\label{newbt9}
V_i^2 = 1 + (\v - \v^{-1})E_iV_i.
\end{equation}
So,
\begin{equation}\label{Ttildeinverse}
V_i^{-1} = V_i - (\v - \v^{-1})E_i.
\end{equation}
In \cite{chjukalaIMRN, esryJPAA,  jaPoJLMS}  this quadratic relation is used to define the bt--algebra. Although at algebraic level these algebras are the same, we will see  that they lead   to different invariants. Thus, in order to distinguish these two presentations of the bt--algebra, we will write $\mathcal{E}_n(\v)$ when the bt--algebra is  defined by using the quadratic relation (\ref{newbt9}).
\end{remark}
\subsection{}
In \cite{aijuMMJ1} it was proved that the bt--algebra supports a Markov trace, this was proved using the method of relative trace and the Ryom--Hansen basis.

Let $\a$ and $\b$  be two  variables commutative independent    commuting with $\u$.
\begin{theorem}[{\cite[Theorem 3]{aijuMMJ1}}]\label{tracerho}
There exists a  unique Markov trace $\rho$ on ${\mathcal E}_{\infty}(\u)$, i.e., a family $\rho :=\{\rho_n\}_{n\in {\Bbb N}}$, where  $\rho_n$'s are linear maps, defined inductively,  from $\mathcal{E}_n(\u)$ in   $\mathbb{K}[\a, \b]$   such  that  $\rho_n(1) =1$ and satisfying,  for all   $X, Y \in\mathcal{E}_n(\u)$,    the following rules:
\begin{enumerate}
\item $\rho_n (XY)=\rho_n(YX)$,
\item $\rho_{n+1}(XT_n)=\rho_{n+1}(XT_nE_n) = \a\rho_n(X)$,
\item $\rho_{n+1}(XE_n)= \b\rho_n(X)$.
\end{enumerate}
\end{theorem}
With this theorem and because the braid group is represented (naturally) in the bt--algebra  we are ready to define an invariant for links. More precisely, denote by  $\pi_{\sqrt{\mathsf{L}} }$ the  (natural) representation of $B_n$ in $\mathcal{E}_n$, namely $\sigma_i \mapsto \sqrt{\mathsf{L}} T_i$. With the same procedure used to get   (\ref{RelationLambdaZeta})  and (\ref{lambdam}), we define
\begin{equation}\label{L}
\mathsf{L}:=\frac{\a +  (1 -\u)\b}{\a\u}.
\end{equation}
\begin{definition}\label{barDelta}
For $\sigma\in B_n$, we define
$$
 \overline{\Delta}(\sigma) =\left( \frac{1}{\a \sqrt{\mathsf{L}}}\right)^{n-1}(\rho_n\circ \pi_{\sqrt{\mathsf{L}}})(\sigma).
$$
\end{definition}
\begin{theorem}\label{DeltaBarra}
 Let  $L$  be a link  obtained by  closing the braid $\sigma\in B_n$. Then the map $L\mapsto \overline{\Delta}(\sigma) $ defines an ambient isotopy  invariant for oriented links, taking values in $\KK(\a, \sqrt{\mathsf{L}})$.
\end{theorem}
\begin{proof}
We need only to prove that $\overline{\Delta}$  agrees with the Markov replacements. Because $ \pi_{\sqrt{\mathsf{L}}}$  is an  homomorphism and the properties of $\rho$ imply that  $\overline{\Delta}$  agrees with M1. On the second replacement, we note that   it is a routine to check $\overline{\Delta}(\sigma) = \overline{\Delta} (\sigma \sigma_n )$. Thus it remains only to check $\overline{\Delta}(\sigma)=\overline{\Delta}(\sigma \sigma_n^{-1})$,   for every $\sigma\in B_n$ . Now, put  $\alpha = \pi_{\sqrt{\mathsf{L}}}(\sigma)$, then:
\begin{eqnarray*}
(\rho_{n+1}\circ \pi_{\sqrt{\mathsf{L}}})(\sigma\sigma_n^{-1})
& = &
\sqrt{\mathsf{L}}^{-1}\rho_{n+1}(\alpha T_n^{-1})\\
& = & \sqrt{\mathsf{L}}^{-1}\rho_{n+1}(\alpha (  T_n + (\u^{-1}-1)E_n \\
& & + (\u^{-1}-1)E_nT_n )) \\
& = & \sqrt{\mathsf{L}}^{-1}(\u^{-1}\rho_n(\alpha) \a + (\u^{-1}-1)\rho_n(\alpha)\b)\\
& =& \sqrt{\mathsf{L}}^{-1}(\u^{-1} \a + (\u^{-1}-1)\b)\rho_n(\alpha).
\end{eqnarray*}
Therefore, $ \overline{\Delta}(\sigma\sigma_{n}^{-1})= \overline{\Delta}(\sigma)$, for  all $\alpha\in B_n$.
\end{proof}

Now, in the way we  define $\Theta_m$ we can define $\bar{\Theta}$,  that is, taking  out its definition by using now the presentation of the bt--algebra $\mathcal{E}_n(\v)$ in the Jones recipe. Having in mind what we did above,  (\ref{lambdamprime}) and Definition \ref{Theta}, we define $ \mathsf{L}' $ and $\overline{\Theta}$ as follows:
\begin{equation}\label{Lprime}
\mathsf{L}' := \frac{\a - (\v^2-1)\b}{\a} ;
\end{equation}
\begin{equation}\label{ThetaPrime}
\overline{\Theta}(\sigma) := \left( \frac{\v}{\a\sqrt{\mathsf{L}^{\prime}}}\right)^{n-1}(\rho_n\circ \pi_{\sqrt{\mathsf{L}^\prime}} )
(\sigma), \quad (\sigma \in B_n).
\end{equation}
\begin{theorem}\label{ThetaBarra}
 Let  $L$  be a link  obtained by  closing the braid $\alpha\in B_n$. Then the map $L\mapsto \overline{\Theta}(\alpha) $ defines an invariant of  ambient isotopy for oriented links, which take values in $\KK(\a, \sqrt{\mathsf{L}'})$.
\end{theorem}
\begin{proof}
Same proof as Theorem \ref{DeltaBarra}.
\end{proof}
\begin{remark}\rm
The invariant $\overline{\Delta}$ contains the invariant $\Delta_m$, for every $m\in\NN$; that is specializing the variable $\b$ to $1/m$ we get $\Delta_m$.  With the same specialization,  the invariant $\overline{\Theta}$ contains the invariants $\Theta_m$, for every $m$.
\end{remark}

\section{Tied links}
In this  section we introduce the tied links. This class of knotted-like objects contains the classical link, so invariants of tied links yield invariants of classical links. We start the section recalling the definition of tied links and the monoid of braid tied links. Also, the respective   Alexander and  Markov theorems are exhibited.
\subsection{}
Tied links were introduced in \cite{aijuJKTR1} and roughly correspond   to  links  whose components may be  connected   by   ties; thus ties are connecting pairs of  points  of  two  components or of the  same  component.

The   ties   in the picture of the tied links  are   drawn  as  springs or dashed lines,  to outline   that they  can be contracted  and  extended,  letting  their  extremes  to  slide  along the  components.

\begin{center}
 \begin{figure}[H]
 \includegraphics[scale=.3]{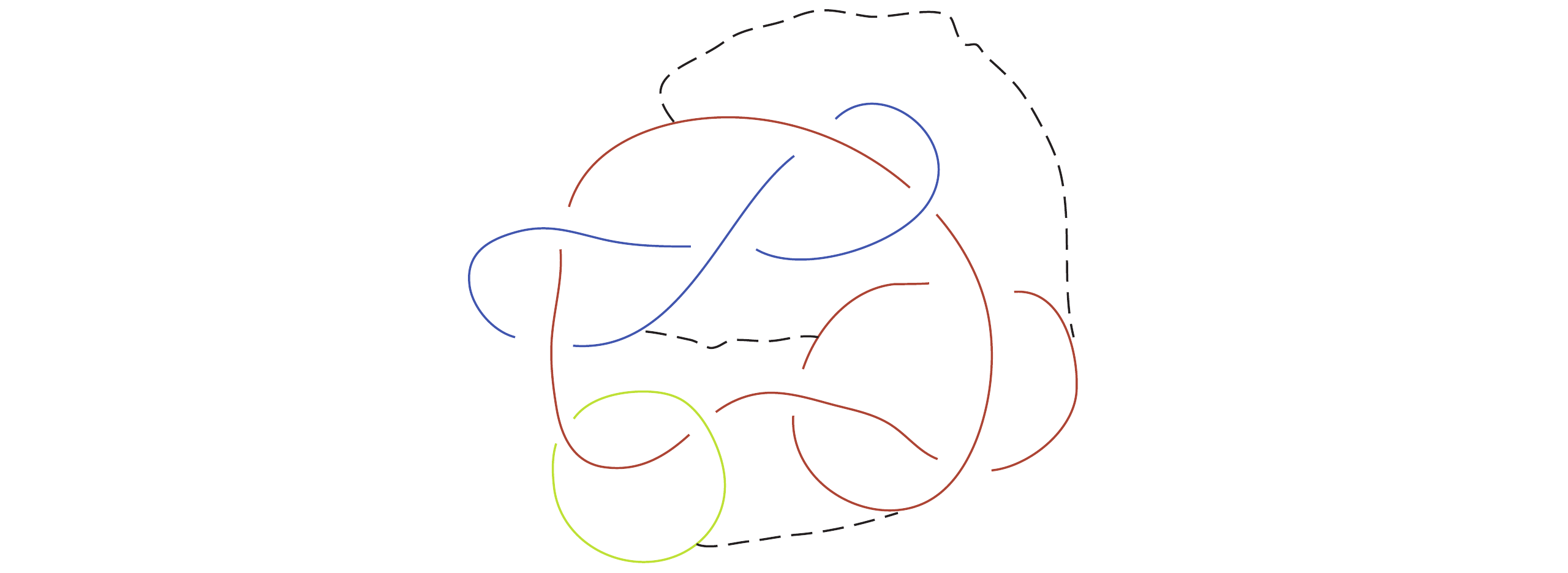}
\caption{ }\label{ExampleTiedLink}
 \end{figure}
 \end{center}
\begin{notation}
 We will use the notation $C_i \leftrightsquigarrow C_j$  to indicate that either there is a tie between the components $C_i$ and $C_j$ of a link, or
  $C_i$ and $C_j$ are  the     extremes of  a  chain  of $m>2$ components $C_1, \dots , C_m$, such that there is  a  tie  between $C_i$  and  $C_{i+1}$,  for  $i=1,\dots, m-1$.
\end{notation}
\begin{definition}[{\cite[Definition 1.1]{aijuJKTR1}}]\label{TiedLinks}
 Every 1--link is by definition a tied 1--link.   For $k > 1$, a  tied $k$--link  is a   link  whose  set of  components $\{C_1, \ldots , C_k\}$ is partitioned  into parts according to: two components $C_i$ and $C_j$ belong to  the  same   part  if
  $C_i \leftrightsquigarrow C_{i+1}$.
\end{definition}

\begin{notation}
We denote by $\widetilde{\mathfrak{L}}$ the set of oriented tied links.
\end{notation}

\begin{center}
 \begin{figure}[H]
 \includegraphics[scale=.35]{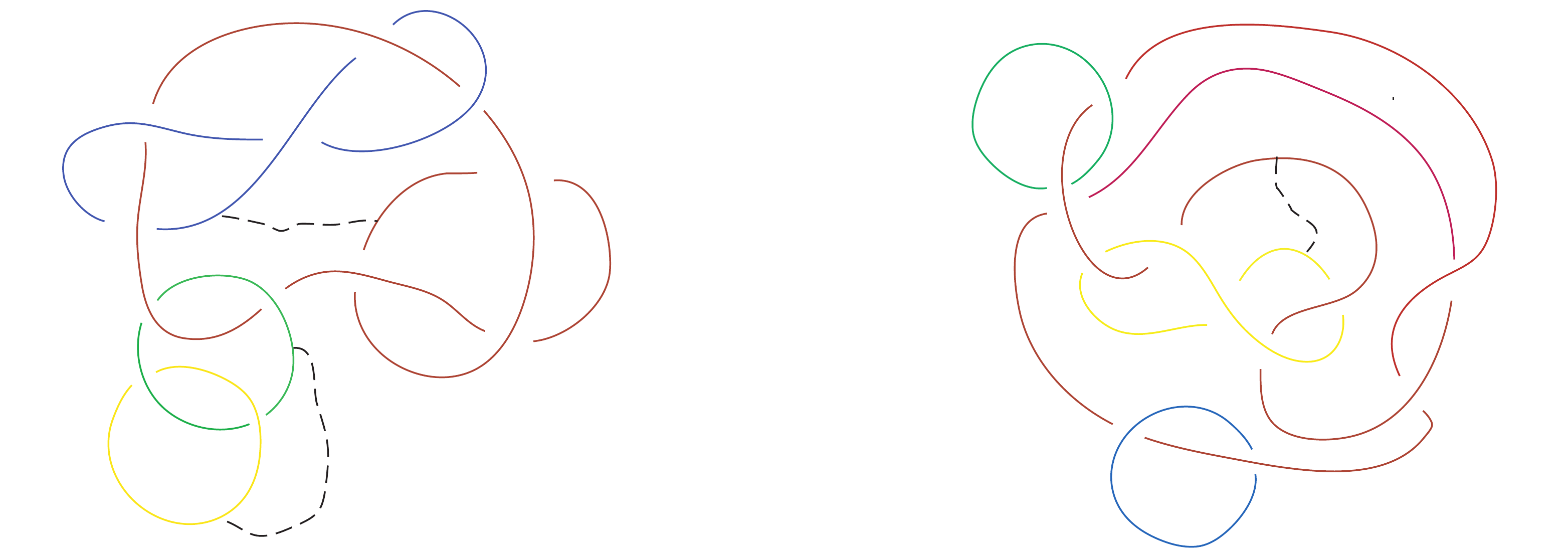}
\caption{ }\label{ExamplesTiedLinks}
 \end{figure}
 \end{center}

  In   Fig. \ref{ExamplesTiedLinks}   two  tied links with four components; moreover, if $C_1$ is the blue component, $C_2$  the red component, $C_3$  the yellow components and $C_4$  the green component, then the partition associated to the first tied link is $\{\{C_1, C_2\}, \{C_3,  C_4\}\}$ and the partitions associated to the second tied link is $\{\{C_1\},  \{C_2 , C_3\}, \{C_4\}\}$.

 Notice that a tied  $k$--link $L$,   with  components' set  $C_L=\{C_1, \ldots , C_k\}$, determines a pair $(L, I(C_L))$ in $\mathfrak{L}_k \times \mathsf{P}_k$, where $i$ and $j$ belong to the same block of $I(C_L)$,   if   $C_i \leftrightsquigarrow C_j$.
\begin{example}
For instance in Fig. \ref{ExamplesTiedLinks}, the set partition determined by the first tied link is $\{\{1,2\},\{3,4\} \}$ and the set partition determined by the second tied link is $\{\{1\}, \{4\}, \{2,3\}\}$.
The set partition determined by the tied link of Fig. \ref{ExampleTiedLink} is $\{\{1,2,3,4\}\}$.
\end{example}
\begin{definition}[{Cf. \cite[Definition 1.6]{aijuJKTR1}}]
A tie of a tied link   is said essential  if it cannot  be  removed without modifying  the  partition $I(C_L)$, otherwise  the  tie is  said  unessential.
\end{definition}
 Notice   that  between  the $c$ components indexed by   the  same block  of  the set  partition,  the  number of       essential ties    is  $c-1$; for instance, in the tied link of  Fig. \ref{ExamplesTiedLinks}, left, among the  three  ties  connecting  the first three components, only two  are  essential. The  number of  unessential ties  is  arbitrary.   Ties  connecting     one  component  with itself are  unessential.
\begin{definition}\label{tisotopy}
The $k$--tied link   $L$ and  the $k'$--tied link $L'$  with, respectively,  components $C=\{C_1, \ldots ,C_k\}$ and $C'=\{C'_1, \ldots ,C'_{k'}\}$, are t--isotopic if:
\begin{enumerate}
\item The links $L$ and $L'$ are ambient isotopic (hence $k=k')$.
\item The set partitions $I(C)$ and $I(C')$  satisfy $I(C')=w_{L, L'}(I(C))$, where  $ w_{L, L'}$  is  the bijection from $C(L)$ to $C(L')$ induced by the isotopy.
\end{enumerate}
\end{definition}
\begin{example}
In the Fig. \ref{tisotopyLinks}, we have that $L1$ and $L2$ are not t--isotopic.
Indeed, the set partition, respectively,  of $L1$  and $L2$ are  $I_1=\{\{1,2,3,4\}\}$ and  $I_2=\{\{1,3,4\},\{2\}\}$ and $w_{L1, L2}= (1,2,4,3)$, but $I_2\not=w_{L1,L2}(I_1)$.

Now,  $L3$  has associated the set partition $I_3=\{ \{1,2,3\},\{4\} \}$   and $L4$ has associated the set partition $I_4=\{\{	1\},\{2,3,4\}	\}$ and $w_{L3,L4}= (4,1,2,3)$; thus,  $I_4= w_{L3,L4} (I_3)$. Then, $L3$ and $L4$ are t--isotopic
\begin{center}
 \begin{figure}[H]
 \includegraphics[scale=.4]{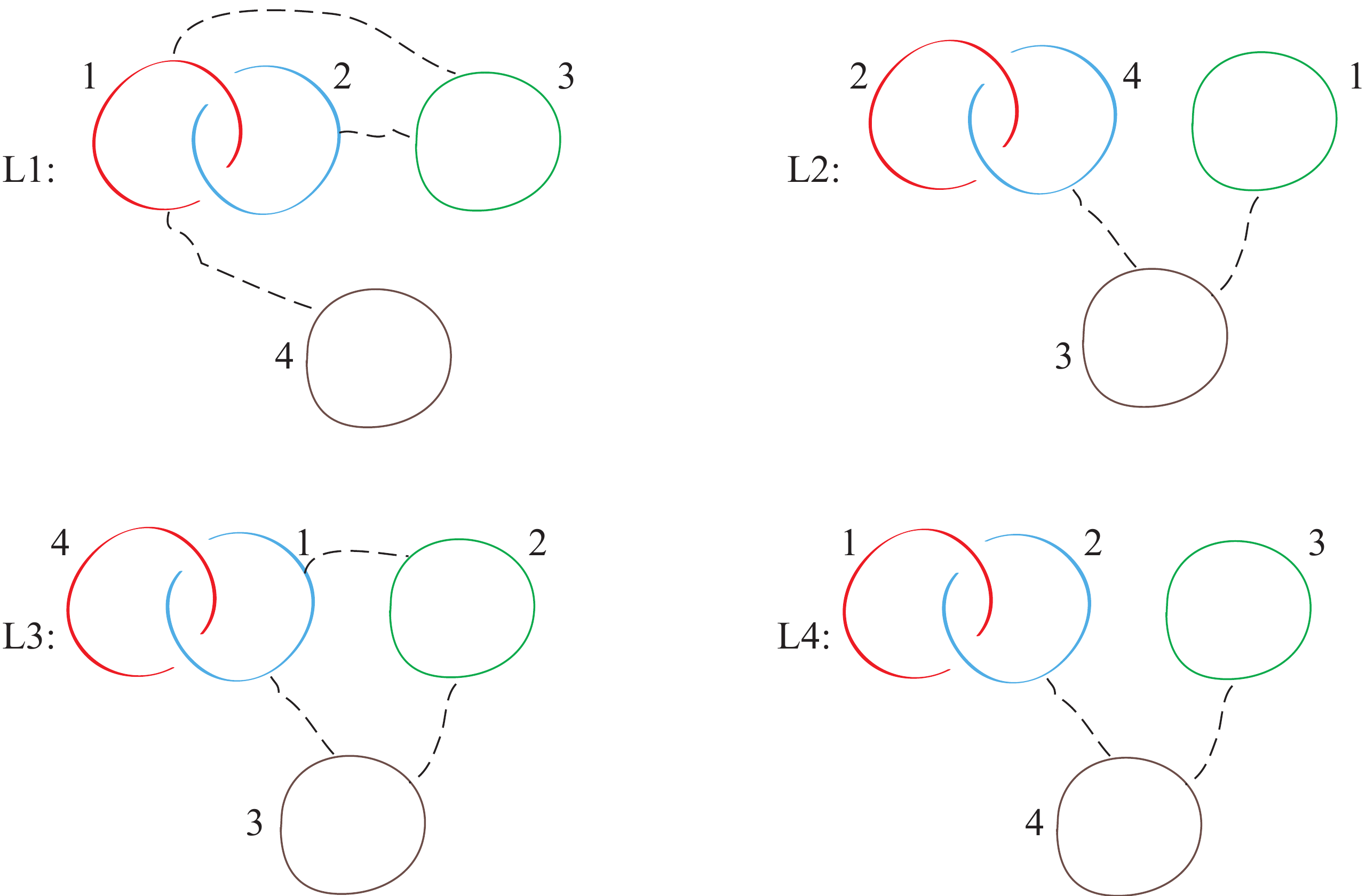}
\caption{ }\label{tisotopyLinks}
 \end{figure}
 \end{center}
\end{example}
\begin{remark}\label{LinksAsTiedLinks}\rm
The Definitions \ref{TiedLinks}  and \ref{tisotopy} not only say that classical links are included in tied links but also that the classical links can be   identified to the set of tied links which components are all tied. Both ways to see the classical links allows to study the isotopy of links  through the t--istopy of tied links.
Observe that for $k$--tied links with set of components $C$, we have $I(C)$ is $\{1, 2, \ldots , n\}$ if it  has  all components tied and is $\{\{1\}, \{2\}, \ldots ,\{n\}\}$ if does not have ties or  has  only ties that are unessential.
\end{remark}
\begin{remark}\rm
Everything established for tied links can be translated in terms of  diagrams in the obvious way. Informally, it is
enough  to change \lq links\rq\ by \lq diagrams of a link\rq\  and   so on.
\end{remark}
\subsection{}
The classical theorems of Alexander and Markov in knot theory have their analogous in the world of tied links. The starting point to establish these theorems for tied links is the so--called  tied braid monoid.

\begin{definition}{\rm \cite[Definition 3.1]{aijuJKTR1}}\label{monoidtb}
 The tied braid monoid $T\!B_n$ is  the monoid generated by usual braids
 $\sigma_1, \ldots , \sigma_{n-1}$  and the  tied generators  $\eta_i, \ldots ,\eta_{n-1}$, such the $\sigma_i$'s satisfy braid relations among them  together with the following relations:
\begin{eqnarray}
\eta_i\eta_j & =  &  \eta_j \eta_i \qquad \text{ for all $i,j$}
\label{eta1}\\
\eta_i\sigma_i  & = &   \sigma_i \eta_i \qquad \text{ for all $i$}
\label{eta2}\\
\eta_i\sigma_j  & = &   \sigma_j \eta_i \qquad \text{for $\vert i  -  j\vert >1$ }
\label{eta3}\\
\eta_i \sigma_j \sigma_i & = &   \sigma_j \sigma_i \eta_j \qquad
\text{ for $\vert i - j\vert =1$}
\label{eta4}\\
\eta_i\sigma_j\sigma_i^{-1} & = &  \sigma_j\sigma_i^{-1}\eta_j \qquad \text{ for $\vert i  -  j\vert =1$}
\label{eta5}\\
\eta_i\eta_j\sigma_i & = & \eta_j\sigma_i\eta_j  \quad = \quad\sigma_i\eta_i\eta_j \qquad \text{ for $\vert i  -  j\vert =1$}
\label{eta6}\\
 \eta_i\eta_i & = & \eta_i \qquad \text{ for all $i$}.
 \label{eta7}
\end{eqnarray}
\end{definition}

We denote $TB_{\infty}$ the inductive limit determined by the natural mono\-morphism  monoid  from $TB_n$ into $TB_{n+1}$. 
\smallbreak
 Diagrammatically, as usual, $\sigma_i$ is represented as the usual braid and the tied generator $\eta_i$,  as the diagram
  of $E_i$, that is,  a  tie connecting   the $i$ with $(i+1)$--strands, see Fig. \ref{TiEi}.
\smallbreak
The tied braid monoid is to the bt--algebra as the braid group is to the Hecke algebra.  In particular, we have the following proposition and its corollary.

\begin{proposition}\label{TildePi}
 The mapping $\eta_i\mapsto E_i$, $\sigma_i\mapsto T_i$ defines an homomorphism, denoted by $\widetilde{\pi}$, from  $TB_n$ to $\En$.
\end{proposition}
\begin{proof}
\end{proof}
\begin{corollary}
$\En$ The bt--algebra is a quotient of $\KK TB_n$. More precisely,
$$
\En \cong \KK TB_n/I,
$$
where $I$ is the two--sided ideal generated by $T_i^2- (\u-1)E_i(1+T_i)$, for $1\leq i\leq n-1$.
\end{corollary}

$TB_n$ has   a decomposition like  semidirect product of groups, this decomposition allows a study   purely algebraic--combinatorics of the tied link, see \cite{aijuSubmitted}, and will be used below. To establish this decomposition, we start  by noting that the action of $\mathtt{S}_n$ on $\mathsf{P}_n$, see (\ref{SnOnPn}),  together  with the natural projection $\mathsf{p}$ of (\ref{BnontoSn}), define an action of $B_n$ on $\mathsf{P}_n$. We denote by $\sigma(I)$, the action of the braid  $\sigma$ on $I\in \mathsf{P}_n$, that is, $\sigma(I)$ is the result of the application of the permutation $\mathsf{p}(\sigma)$ to the set partition $I$. Define now  the following product in $\mathsf{P}_n\times B_n$:
 $$
(I, \sigma)(I', \sigma') = (I	\ast\sigma(I'), \sigma\sigma').
 $$
 $\mathsf{P}_n\times B_n$  with this product is a monoid, which is  denoted by $\mathsf{P}_n\rtimes B_n$.   We shall  denote $I\sigma$ instead  $(I,\sigma)$.

 \begin{theorem}[{\cite[Theorem 9]{aijuSubmitted}}]\label{Semidirect}
 The monoid $TB_n$ and $\mathsf{P}_n\rtimes B_n$ are isomorphic.
 \end{theorem}

Now, as for braid, we can define the closure of a tied braid in the same way  of the closure of the braids. Evidently, the closure of a tied braid is a tied link. Moreover, in \cite[Theorem 3.5]{aijuJKTR1} we have proved the  Alexander theorem for tied links; namely.

\begin{theorem}[{\cite[Theorem 3.5]{aijuJKTR1}}]
Every oriented tied link can be obtained as closure of  a tied braid.
\end{theorem}

Before  establishing the Markov theorem for tied links, notice that according to Theorem \ref{Semidirect}, every element $a$ in  $TB_n$ can be written uniquely in the form $a=I\sigma$, where $I\in \Pn$ and $\sigma\in B_n$. We use this fact in the definition of Markov moves for tied links and also we use the notation of the $\mu_{i,j}$'s as in (\ref{relationsPn}).
\begin{definition}
We say that $a,b\in TB_{\infty}$ are t--Markov equivalents, denoted  $a\sim_{tM}b$, if $b$ can be obtained from $a$ by using   a finite sequence of the following replacements:
\begin{enumerate}
\item[tM1.] t--Stabilization: for all $a=I\sigma\in TB_n$, we can replace
$a$ by $ a\mu_{i,j}$, if $i,j$ belong to the  same  cycle of  $\mathsf{p}(\sigma)$.
\item[tM2.] Commuting in $TB_n$. For all $a, b\in TB_n$, we can replace
$ab$ by $ ba $,
\item[tM3.] Stabilizations: for all $a \in TB_n$, we can replace
$a$ by $ a\sigma_n $ or $a \sigma_n^{-1}.$
\end{enumerate}
\end{definition}

The relation $\sim_{tM}$ is an  equivalence relation on $TB_{\infty}$ and in \cite[Theorem 3.7]{aijuJKTR1}, cf. \cite[Theorem 5]{aijuSubmitted},  the following  theorem  was proved.

 \begin{theorem}
Two tied links define t--isotopic  tied links if and only if they are t--Markov equivalents.
\end{theorem}
\begin{remark}\rm
The Markov replacements $tM2$ and $tM3$ are the classical Markov moves but including now the ties, so their  geometrical meaning is clear. The replacement tM1 says that the ties provided by $\mu_{i,j}$ in the closure of $a\mu_{i,j}$ become unessential ties, so the closure of $a$ and $a\mu_{i,j}$  defines  the same set partition.
\end{remark}
\section{The invariant $\mathcal{F}$}

We will define  the invariant  $\mathcal{F}$ for tied links. This invariant is of type Homflypt since  evaluated on  classical links coincide with the Homflypt polynomial. We start by defining it    by skein relations and later  their definition by the Jones recipe.

\subsection{}To define $\mathcal{F}$ by skein relations we need, as in (\ref{ConwayTriple}),  to introduce the following notation:
$L_+, L_-, L_0, L_{+,\sim}, L_{-,\sim}$ and  $L_{0,\sim}$ denote oriented tied links which have, respectively, oriented  diagrams of tied links  $D_+, D_-, D_0, D_{+,\sim}, D_{-,\sim}$ and  $D_{0,\sim}$, that are identical outside a small disk into which enter two strands, whereas inside the disk the strands look, respectively,  as Fig. \ref{ConwayTied} shows.
\begin{center}
\begin{figure}[H]
\includegraphics[scale=.5]{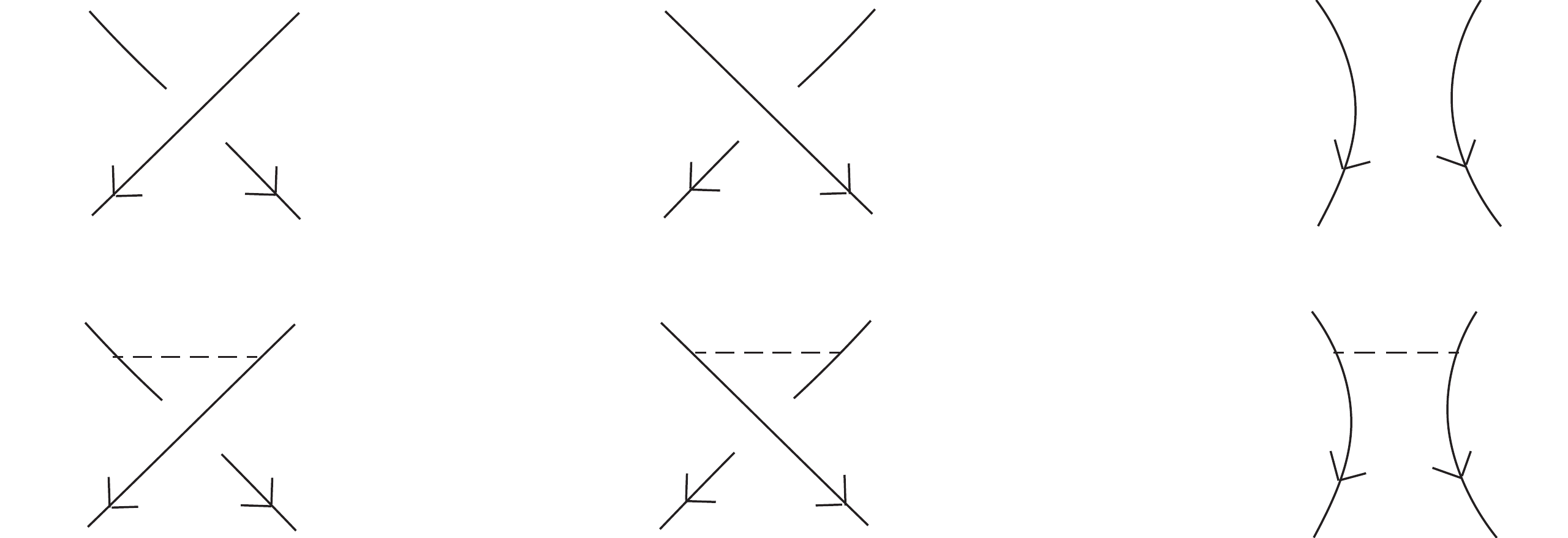}
\caption{ }\label{ConwayTied}
\end{figure}
\end{center}
We shall call  $L_{+,\sim}, L_{-,\sim}$ and  $L_{0,\sim}$ a  tied Conway triple.

Set $\w$ a variable with unique condition that  commutes with $\u$ and  $\a$.
\begin{theorem}[{\cite[Theorem 2.1]{aijuJKTR1}}]\label{SkeinF}
There is a unique function $\mathcal{F}: \widetilde{\mathfrak{L}} \longrightarrow \KK(\a, \w)$, defined by the rules:
\begin{enumerate}
\item  $\mathcal{F} (\bigcirc)=  1$,
\item For all tied link $L$,
$$
\mathcal{F} (\bigcirc\sqcup  L)=  \frac{1}{\a\w} \mathcal{F} (L),
$$
\item The Skein rule:
$$
\frac{1}{\w} \mathcal{F}(L_+) -\w\mathcal{F}(L_- )= (1-\u^{-1})\mathcal{F}(L_{0,\sim}) -\frac{1-\u^{-1}}{\w}\mathcal{F}(L_{+,\sim}).
$$
\end{enumerate}
\end{theorem}
\begin{proposition}\label{OthersSkeinF}
The defining skein relation of $\mathcal{F}$ is equivalent to  the following skein rules:
\begin{enumerate}
\item $$
\frac{1}{\u\w}\mathcal{F}(L_{+,\sim})-\w\mathcal{F}L_{-,\sim} = (1-\u^{-1})\mathcal{F}L_{0,\sim},
$$
\item $$
\frac{1}{\w}\mathcal{F}(L_{+})= \w\left[\mathcal{F}(L_{-} )+ (\u-1)\mathcal{F}(L_{-,\sim})\right] + (\u-1)\mathcal{F}(L_{0,\sim}),
$$
\item $$
\w\mathcal{F}(L_{-})= \frac{1}{\w}\left[\mathcal{F}(L_{+} )+ (\u^{-1}-1)\mathcal{F}(L_{+,\sim})\right] + (\u^{-1}-1)\mathcal{F}(L_{0,\sim}).
$$
\end{enumerate}
\end{proposition}
\begin{proof}
\end{proof}

Define $\widetilde{\mathfrak{L}}_0$  the set of tied links with all components tied.
\begin{proposition}\label{FatL0}
The restriction of $\mathcal{F}$ to $\widetilde{\mathfrak{L}}_0$, is determined uniquely by the rules:
\begin{enumerate}
\item $\mathcal{F} (\bigcirc)=  1$,
\item Skein relation on tied Conway triple $L_{+, \sim}$,  $L_{-,\sim}$ and $ L_{0, \sim}$,
$$
\r^{-1}\mathcal{F}(L_{+, \sim}) - \r\mathcal{F} (L_{-,\sim})  = \s \mathcal{F} (L_{0,	\sim}),
$$
where $\r=\w \sqrt{u}$ and $\s= \sqrt{\u}- \sqrt{\u}^{-1}$.
\end{enumerate}
\end{proposition}
\begin{proof}
 The computation of $\mathcal{F}$ on a such  tied link can be realized by rules (1) Theorem \ref{SkeinF}
and the skein relation  (1) Proposition \ref{OthersSkeinF}.  Multiplying this skein relations by  $\sqrt{\u}$,
 we get the values of $\r$ and $\s$. Hence the proof is concluded.
\end{proof}
\begin{remark}\rm
According to Remark \ref{LinksAsTiedLinks}, $\widetilde{\mathfrak{L}}_0$ is identified to classical links, so, the above proposition  and Theorem(\ref{SkeinHomflypt})) say that
$\mathcal{F}$  and Homflypt polynomial coincide on classical links. For this reason we say that $\mathcal{F}$ is invariant of type Homflypt.
\end{remark}

\begin{proposition}
The restriction of $\mathcal{F}$ to classical links is more powerful than the Homflypt polynomial.
\end{proposition}
\begin{proof}

\end{proof}

\subsection{}

In order to define $\mathcal{F}$ through the Jones recipe, we extend first
 the domain $\pi_{\sqrt{\mathsf{L}}}$  to $TB_n$; we denote this extension by $\widetilde{\pi}_{\sqrt{\mathsf{L}}}$. Secondly,   we define $\widetilde{\Delta}$ by
$$
\widetilde{\Delta} (\eta) := \left( \frac{1}{\a \sqrt{\mathsf{L}}}\right)^{n-1}(\rho_n\circ \widetilde{\pi}_{\sqrt{\mathsf{L}}})(\eta) \quad (\eta\in TB_n).
$$
\begin{theorem}
Let $L$ be a tied link obtained as the closure of the tied braid $\eta$, then
$$
\widetilde{\Delta} (\eta) =\mathcal{F}(L).
$$
The relation among, $\u$, $\w$ and parameters trace $\a$ and $\b$ of $\rho$, is given by the equation:
$$
\b = \frac{\a(\u\w^2-1)}{1-\u}.
$$
\end{theorem}

\end{document}